\documentclass[a4paper,11pt,oneside,reqno]{article}
\usepackage[top=1in, left=1in, right=1in, bottom=1in]{geometry}

\usepackage{hyperref, amsthm, amsmath, amssymb, amsfonts, mathrsfs, graphicx, xcolor, multicol}

\newtheorem{theorem}{Theorem}[section]
   
\newtheorem{remark}[theorem]{Remark}

\newtheorem{lemma}[theorem]{Lemma}
\newtheorem{corollary}[theorem]{Corollary}
\newtheorem{problem}[theorem]{Problem}

\usepackage{bm}
\newcommand{\lrangle}[2]{\left\langle{#1},{#2}\right\rangle} 
\newcommand{\vect}[1]{\boldsymbol{\mathbf{#1}}} 
\newcommand{\vertiii}[1]{{\vert\kern-0.25ex\vert\kern-0.25ex\vert #1 
    \vert\kern-0.25ex\vert\kern-0.25ex\vert}}

\newcommand{\dom}{\Omega}
\newcommand{\dett}{I_{t}}
\newcommand{\dive}{\operatorname{div}}

\newcommand{\uup}{\vect{u}^{\prime}}
\newcommand{\uu}{\vect{u}}
\newcommand{\uut}{\vect{u}_{t}}

\newcommand{\realu}{\vect{u}_{r}}
\newcommand{\imagu}{\vect{u}_{i}}

\newcommand{\vv}{\vect{v}}

\newcommand{\ww}{\vect{w}}

\newcommand{\realp}{p_{r}}
\newcommand{\imagp}{p_{i}}

\newcommand{\pd}{p_{D}}
\newcommand{\pn}{p_{N}}

\newcommand{\ff}{\vect{f}}
\newcommand{\bgg}{\vect{g}}
\newcommand{\ggb}{{G}}

\newcommand{\cv}{\overline{\vv}}
\newcommand{\cw}{\overline{\ww}}

\newcommand{\Mt}{M_{t}}

\newcommand{\Vgamma}{V_{\Gamma}}

\usepackage{stackengine}

\newcommand{\oQ}{{Q}_{0}}

\newcommand{\un}{\vect{u}_{N}}
\newcommand{\ud}{\vect{u}_{D}}

\newcommand{\bphi}{\vect{\varphi}}
\newcommand{\bpsi}{\vect{\psi}}
\newcommand{\cbphi}{\overline{\vect{\varphi}}}
\newcommand{\cbpsi}{\overline{\vect{\psi}}}

\newcommand{\rt}{r^{t}}
\newcommand{\yt}{\vect{y}^{t}}
\newcommand{\rti}{r_{i}^{t}}
\newcommand{\yti}{\vect{y}_{i}^{t}}
\newcommand{\cyt}{\overline{\vect{y}}^{t}}

\newcommand{\ut}{\uu^{t}}
\newcommand{\ur}{\uu_{r}}
\newcommand{\ui}{\uu_{i}}

\newcommand{\pr}{p_{r}}
\newcommand{\pim}{p_{i}}

\newcommand{\vr}{\vv_{r}}

\newcommand{\vi}{\vv_{i}}

\newcommand{\aaa}{a}
\newcommand{\aat}{a^{t}} 
\newcommand{\taat}{\tilde{a}^{t}} 
\newcommand{\LL}{\vect{\mathsf{L}}}

\newcommand{\sfTheta}{\vect{\Theta}}

\newcommand{\HH}{\vect{\mathsf{H}}}

\newcommand{\VV}{\vect{\theta}} 
\newcommand{\Vn}{\theta_n}

\newcommand{\nn}{\vect{n}}

\newcommand{\nnt}{\vect{n}_t}
 
\newcommand{\abs}[1]{\vert{#1}\vert}
\newcommand{\norm}[1]{\left\|{#1}\right\|}

\newcommand{\dn}[1]{\partial_{\nn}{#1}}

\newcommand{\intO}[1]{\int_{\dom}{#1}{\, {d} x}}
\newcommand{\intdO}[1]{\int_{\partial\dom}{#1}{\, {d} \sigma}}
\newcommand{\intOt}[1]{\int_{\dom_{t}}{#1}{\, {d} x_{t}}}  

\newcommand{\intS}[1]{\int_{\Sigma}{#1}{\, {d} \sigma}} 

\newcommand{\intSt}[1]{\int_{\Sigma_{t}}{#1}{\, {d} \sigma_{t}}}
\newcommand{\intG}[1]{\int_{\Gamma}{#1}{\, {d} \sigma}}

\begin{document}

\title{Detecting immersed obstacle in Stokes fluid flow using the coupled complex boundary method}

\author{J. F. T. Rabago
  \and
  L. Afraites
  \and
  H. Notsu
}

\newcommand{\Addresses}{{
  \bigskip
  \footnotesize

  J.~.F.~.T.~Rabago, \textsc{Faculty of Mathematics and Physics, Institute of Science and Engineering,
        Kanazawa University, Kakumamachi, Kanazawa 920-1192, Japan}\par\nopagebreak
  \textit{E-mail address}, J.~.F.~.T.~Rabago: \texttt{jfrabago@gmail.com}

  \medskip

  L.~Afraites, \textsc{EMI, FST, Universit\'{e} Sultan Moulay Slimane,	 B\'{e}ni-Mellal 23000, Morocco}\par\nopagebreak
  \textit{E-mail address}, L.~Afraites: \texttt{l.afraites@usms.ma}

  \medskip

  H.~Notsu, \textsc{Faculty of Mathematics and Physics, Institute of Science and Engineering,
        Kanazawa University, Kakumamachi, Kanazawa 920-1192, Japan}\par\nopagebreak
  \textit{E-mail address}, H.~Notsu: \texttt{notsu@se.kanazawa-u.ac.jp}

}}

\maketitle







\begin{abstract} 
A non-conventional shape optimization approach is introduced to address the identification of an obstacle immersed in a fluid described by the Stokes equation within a larger bounded domain, relying on boundary measurements on the accessible surface. 
The approach employs tools from shape optimization, utilizing the coupled complex boundary method to transform the over-specified problem into a complex boundary value problem by incorporating a complex Robin boundary condition. 
This condition is derived by coupling the Dirichlet and Neumann boundary conditions along the accessible boundary. 
The identification of the obstacle involves optimizing a cost function constructed based on the imaginary part of the solution across the entire domain. 
The subsequent calculation of the shape gradient of this cost function, rigorously performed via the rearrangement method, enables the iterative solution of the optimization problem using a Sobolev gradient descent algorithm. 
The feasibility of the method is illustrated through numerical experiments in both two and three spatial dimensions, demonstrating its effectiveness in reconstructing obstacles with pronounced concavities under high-level noise-contaminated data, all without perimeter or volume functional penalization.

\end{abstract}

\tableofcontents
\section{Introduction}\label{sec:Introduction}
In this study, we are interested in the problem of reconstructing the shape of an obstacle or inclusion $\omega$ immersed in a fluid living within a larger bounded domain $D$ from measurements taken on the accessible boundary.
Our main objective is to numerically solve this shape inverse problem through a non-conventional shape optimization approach, which consists of a system of complex partial differential equations (PDEs).
%
%
%
\subsection{Problem setting}
Let $D \subset \mathbb{R}^{d}$, $d \in \{2,3\}$, be a given bounded Lipschitz domain, i.e., $D$ is a bounded, open, and connected set.
We fix a small number $\delta > 0$ and denote by $\mathcal{O}_{\delta}$ the set of all subdomain $\omega$ with $\mathcal{C}^{1,1}$ boundary $\Gamma := \partial\omega$ such that $d(x,\Sigma) > \delta$ for all $x \in \omega$, where $\Sigma:=\partial{D}$, and $\Omega := D \setminus \overline{\omega}$ is connected.
We call $\mathcal{O}_{\delta}$ the set of admissible geometries (or inclusions). 
Hereinafter, we call $\Omega$ an \textit{admissible domain} if $\Omega = D \setminus \overline{\omega}$ for some $\omega \in \mathcal{O}_{\delta}$.

For a given Neumann boundary data $\vect{g} \in H^{-1/2}(\Sigma)^{d} := (H^{1/2}(\Sigma; \mathbb{R}^{d}))'$, $\vect{g} \not\equiv \vect{0}$, and a corresponding Dirichlet boundary measurement $\ff \in H^{1/2}(\Sigma)^{d}:=H^{1/2}(\Sigma; \mathbb{R}^{d})$ (or vice versa), the inclusion $\omega$, the vector function $\uu=\uu(\Omega)$ representing the velocity of fluid, and a scalar function $p=p(\Omega)$ representing the pressure, satisfy the over-specified boundary value problem (BVP):
\begin{equation}
\label{eq:overdetermined_original}
	-\nabla \cdot \sigma(\uu,p)	 = \vect{0}	\ \text{in $\Omega$},\quad
	\dive{\uu}	= 0 \ \text{in $\Omega$},\quad
	\uu = \ff \ \ \text{and} \ \ \sigma(\uu,p)\nn = \vect{g} \ \text{on $\Sigma$},\quad
	\uu = \vect{0} \ \text{on $\Gamma$}.
\end{equation}
For technical reasons, we assume -- without further notice -- that the boundary data $\ff \in H^{1/2}(\Sigma)^{d}$, $\ff \not\equiv \vect{0}$, satisfies the standard flux compatibility condition
\begin{equation}\label{eq:compatibility_condition} 
	\intS{\uu \cdot \nn} = \intS{\ff \cdot \nn} = \langle \ff, \nn \rangle_{\Sigma} = 0.
\end{equation}
In \eqref{eq:overdetermined_original}, $\sigma$ denotes the stress tensor defined by $\sigma(\uu,p):={\alpha} \big(\nabla \uu + \nabla \uu^{\top} \big)-p\,{id}$, and ${\alpha}$ is the coefficient of kinematic viscosity while $id$ stands for the $d$-dimensional identity matrix\footnote{Here, $id$ also denotes the identity operator. If there is no confusion, this abuse of notation is used throughout the paper.}.
The notation $\nabla^{\top} \uu$ stands for the transposed of the matrix $\nabla{\uu}$.
Because $\dive{\uu}=0$ in  $\Omega$, then we can actually write $-\dive{\sigma(\uu,p)}=-{\alpha} \Delta \uu  + \nabla p$ in $\Omega$.
In this case, we simply express $\sigma(\uu,p)\nn={\alpha} \dn{\uu}  - p\nn$ on $\Sigma$, where $\partial_{n}\uu$ denotes the outward normal derivative (with $\nn$ being the outward unit normal vector to $\Omega$).
Here, we will use both of these expressions interchangeably.

Henceforth, we assume here that there exists $\omega \in \mathcal{O}_{\delta}$ such that \eqref{eq:overdetermined_original} admits a solution. 
This means that the measurement $\ff$ is perfect, that is to say without error.
Therefore, we can consider the shape inverse problem formulated as follows:
\begin{problem}\label{problem:gip}
	Find $w\in \mathcal{O}_{\delta}$ and the pair $(\uu, p)$ that satisfies the over-specified system \eqref{eq:overdetermined_original}.
\end{problem}  
If $(\omega,\uu,p)$ solves Problem \ref{problem:gip}, then we call the triplet a \textit{solution} of \eqref{eq:overdetermined_original}.

As mentioned earlier, we will tackle the inverse problem of reconstructing $\omega$, employing tools from shape optimization. 
In this direction, we first need to discuss the identifiability of the inclusion from boundary measurements. 
This fundamental question in the Dirichlet case was addressed by Alvarez-Conca-Friz-Kavian-Ortega in \cite{Alvarezetal2005}:
\begin{theorem}[{\cite{Alvarezetal2005}}]
	Let ${D} \subset \mathbb{R}^{d}$, $d \in \{2,3\}$, be a bounded $\mathcal{C}^{1,1}$ domain, and $\Sigma^{0}$ be a non-empty open subset of $\Sigma := \partial{D}$.
	Moreover,  let
	\[
		\omega^{1}, \omega^{2} \in \{ \omega \Subset {D} \mid \ \text{$\omega$ is open, Lipschitz, and $\Omega:=D\setminus\overline{\omega}$ is connected}\},
	\]
	and $\ff \in H^{1/2}(\Sigma)^{d}$ with $\ff \neq \vect{0}$, satisfying the flux condition $\displaystyle \intS{\ff \cdot \nn} = 0$.
	
	Let $(\uu^{j}, p^{j})$, for $j=1,2$, be a solution of
	\[
	\left\{\arraycolsep=1.4pt\def\arraystretch{1.1}
	\begin{array}{rcll}
	-\dive{\sigma(\uu^{j},p^{j})} 
	&=&\vect{0} & \ \text{in $\Omega^{j}$},\\
	\dive{\uu^{j}}			&=&0 & \ \text{in $\Omega^{j}$},\\
	\uu^{j} 				&=&\ff  & \ \text{on $\Sigma:=\partial{D}$},\\ 
	\uu^{j} 				&=&\vect{0} & \ \text{on $\Gamma^{j}:=\partial\omega^{j}$},
	\end{array}
	\right.
	\]
	If $(\uu^{1}, p^{1})$ and $(\uu^{2}, p^{2})$ are such that
	\[
		\sigma(\uu^{1},p^{1})\nn = \sigma(\uu^{2},p^{2})\nn \quad \text{on $\Sigma^{0}$},
	\]
	then
	\[
		\omega^{1} \equiv \omega^{2}.
	\]
\end{theorem}
A similar result can also be proven for the case when the working problem is given with purely the Neumann boundary condition, and with ${\alpha} \dn{\uu} - p\nn$ is given instead of $\sigma(\uu,p)\nn$, $\dom$ a Lipschitz regular domain, and $\ff \in H^{1/2}(\Sigma)^{d}$, see \cite[Thm. 2.2]{BadraCaubetDambrine2011}.
This aforementioned result states that given a fixed $\ff$, two different geometries $\omega^{1}$ and $\omega^{2}$ in $\mathcal{O}_{\delta}$ yield two different measures ${\bgg}^{1}$ and ${\bgg}^{2}$.
%
%

%
%
%
\subsection{Shape optimization techniques for shape recovery}
To recover the shape of the obstacle $\omega$, a common strategy involves minimizing a cost function. 
Various types of cost functions may be considered for this purpose; see \cite{BadraCaubetDambrine2011,Caubet2013,CaubetDambrineKatebTimimoun2013} for examples of existing cost functions.
A conventional approach is to consider the Neumann-least-squares-fitting cost function given by \cite{BadraCaubetDambrine2011, Caubet2013}
\begin{equation}
\label{eq:cost_tracking_Neumann_data}
	J_{N}(\omega)=\cfrac{1}{2}\int_{\Sigma} |\sigma(\ud(\omega),\pd(\omega))\nn-{\bgg}|^{2},
\end{equation}
where ${\bgg}$ is the boundary measurement and $(\ud,\pd)$ is the solution of the pure Dirichlet BVP:
\begin{equation}
\label{eq:state_ud}
-{\alpha} \Delta  \ud+ \nabla \pd = \vect{0} \ \text{ in } \Omega, \qquad
\dive{\ud}	= 0  \ \text{ in } \Omega, \qquad
\ud = \vect{0} \ \text{ on } \Gamma, \qquad
\ud = \ff \ \text{ on } \Sigma.
\end{equation}
Given the compatibility condition \eqref{eq:compatibility_condition}, problem \eqref{eq:state_ud} admits a unique solution once a normalization condition on the pressure $p$ is imposed (see, e.g., \cite[Chap. 1]{Temam2001}). 
In the minimization setting `$J_{N}(\omega) \to \inf$', one can choose the normalization $\langle \sigma(\ud,\pd)\nn, \nn \rangle_{\Sigma} = \langle \bgg, \nn \rangle_{\Sigma}$. 
To ensure that the cost function $J_{N}$ is well-defined, it is necessary to assume that the Dirichlet data input is such that $\ff \in H^{3/2}(\Sigma)^{d}$. 
To avoid this higher regularity assumption, it is usually preferable to fit Dirichlet data rather than Neumann data.
In this case, one evaluates the misfit of the Dirichlet data using the least-squares method:
\begin{equation}\label{eq:cost_tracking_Dirichlet_data}
	J_{D}(\omega)=\cfrac{1}{2}\int_{\Sigma} |\un(\omega) - \ff|^{2},
\end{equation}
where  $(\un,\pn)$ satisfies the mixed Dirichlet-Neumann BVP:
\begin{equation}
\label{eq:state_un}
    -{\alpha} \Delta \un+\nabla \pn = \vect{0} \ \text{ in } \Omega ,\quad
    \dive{\un} = 0 \ \text{ in } \Omega,\quad
    \un = \vect{0} \ \text{ on } \Gamma,\quad
    \sigma(\un,\pn)\nn = {\bgg} \ \text{ on } \Sigma.
\end{equation}
It can be shown that system \eqref{eq:state_un} admits a unique solution once a normalization condition on $\un$ is imposed.

Another cost function one may take into consideration is the called Kohn-Vogelius functional defined as follows \cite{CaubetDambrineKateb2013}:
\begin{equation}\label{eq:cost_KV}
	J_{KV}(\omega)=\cfrac{1}{2}\intO{ {\alpha} \abs{\nabla(\ud(\omega)-\un(\omega))}^{2}}, 
\end{equation}
where $\un$ is the solution of Neumman problem \eqref{eq:state_un} and $\ud$ is the solution of the Dirichlet one \eqref{eq:state_ud}.
The cost function $J_{KV}$ measures the energy-gap between the solutions of \eqref{eq:state_ud} and \eqref{eq:state_un}.
This function vanishes only if $\ud=\un$ which is the case when the shape $\omega$ fits the exact inclusion.

The optimization problems and shape functionals mentioned above were studied in the context of closely related shape inverse problems for the Stokes equations, utilizing tools from shape optimization, as discussed in \cite{CaubetDambrineKateb2013, CaubetDambrineKatebTimimoun2013}. 
This involves employing the shape derivative of the associated cost functional with respect to the shape of the region to be reconstructed. 
In the three works mentioned above, the authors characterized the shape gradient of the same functionals to facilitate numerical resolution. 
Inspired by these works, we will do the same here, but through a proposal of a new shape optimization formulation.

%
%
%
%
\begin{remark}[Existence, uniqueness, and regularity]
For detailed discussions on the existence, uniqueness, and regularity of solutions to the Stokes problem with Dirichlet boundary conditions, we recommend referring to \cite{BoyerFabrie2006} or \cite{Galdi1994}. 
Meanwhile, for the existence, uniqueness, and local regularity results for the Stokes' problem with mixed conditions \eqref{eq:state_un}, one may consult Theorem A.1 and Theorem A.2 in \cite[Appx. A]{CaubetDambrineKatebTimimoun2013}.
\end{remark}
In this study, we propose the use of the Coupled Complex Boundary Method (CCBM) in a shape optimization setting to numerically approximate a solution to the inverse problem. We first emphasize that, in the classical approach \eqref{eq:cost_tracking_Neumann_data} and \eqref{eq:cost_tracking_Dirichlet_data}, the required data for fitting is carried out on the boundary. 
However, this condition often renders the numerical solution unstable.
In contrast, in the case of CCBM, the data needed for fitting is transferred from the boundary to the interior of the domain -- similar to \eqref{eq:cost_KV}. 
This is because CCBM couples the Neumann and Dirichlet data to derive a Robin condition. 
Consequently, the Neumann data and the Dirichlet data are represented as the real and imaginary parts, respectively, of the Robin boundary condition. 
As a result, this method is more favorable in terms of numerical realization.
However, we note that the new boundary value problem is complex, leading to an increase in the dimension of the corresponding discrete system. 
Nevertheless, the method proves to be more robust, and effective numerical methods could be explored based on the new framework.
%
%
%
%
\subsection{The CCBM and contribution to the literature} 
Initially introduced by Cheng et al. in \cite{Chengetal2014} for an inverse source problem, CCBM has found applications in various problmes.
Gong et al. utilized it for an inverse conductivity problem in \cite{Gongetal2017}, and Cheng et al. applied it to parameter identification in elliptic problems in \cite{ZhengChengGong2020}. 
Afraites extended the method to shape inverse problems for the Laplace equation in \cite{Afraites2022} and geometric inverse source problems in \cite{AfraitesMasnaouiNachaoui2022}. 
Rabago pioneered its use in free boundary problems, specifically solving the exterior case of the Bernoulli problem in \cite{Rabago2022}.
Rabago and Notsu further extended the use of CCBM to free surface problems for the Stokes equations in \cite{RabagoNotsu2024}. 
Notably, there has been no prior application of CCBM to obstacle problems except in \cite{Ouiassaetal2022}, where it was used to solve an inverse Cauchy Stokes problem. 
In this study, the authors employed CCBM as a numerical strategy to determine fluid velocity and flux over a specific boundary portion, reconstructing an obstacle immersed in a fluid described by the Stokes equation through numerical shape optimization techniques. 
The novelty of our work lies in applying CCBM to shape inverse obstacle problems in fluid flows. 
Emphasizing the motivation, it is worth noting that, in the context of free boundary problems \cite{Rabago2022} and free surface problems \cite{RabagoNotsu2024}, CCBM outperforms the classical least-squares approach in terms of stability and accuracy in obtaining the expected optimal shape solution.
Additionally, in terms of computational cost, CCBM is roughly comparable to the Kohn-Vogelius method, see \cite{Rabago2022}.

We highlight the contribution of this examination to the literature. 
From a theoretical point of view, we emphasize the following.
\begin{itemize}
\item This study introduces a novel cost functional \eqref{eq:ccbm_cost_function} to address the shape inverse obstacle Problem \ref{problem:gip}. 
The formulation is derived from the CCBM formulation \eqref{eq:ccbm}. 
Notably, this functional has not been explored in the current literature within the context of the present problem.
\item The study is centered on the computation of the first-order shape derivative of the cost functional associated with CCBM (Theorem \ref{thm:shape_derivative}). 
This computation is achieved only through the H\"{o}lder continuity of state variables (Lemma \ref{lem:holder_continuity}), following the methodology outlined in \cite{RabagoNotsu2024}. 
Our approach deviates from the conventional chain rule method (see, e.g., \cite{BadraCaubetDambrine2011, CaubetDambrineKatebTimimoun2013}) and the traditional Lagrangian formulation technique (see \cite{KovtunekoOhtsuka2022}), eliminating the need to calculate material or shape derivatives of states.
\end{itemize}
Meanwhile, from a numerical standpoint, we depart innovatively from the methodology applied in \cite{CaubetDambrineKatebTimimoun2013}. 
Our approach distinguishes itself in three key aspects. 
\begin{itemize}
\item In the previously mentioned work, the authors addressed the ill-posedness of the problem by employing parametric regularization, incorporating a parametric model of shape variations. 
This involved imposing restrictions on domains, confining them to star-shaped configurations, and using polar coordinates for parametrization. 
This approach led to a unique method for calculating deformation fields.
{In contrast, our approach relies on a straightforward Sobolev gradient-based method inspired by \cite{Doganetal2007}. 
Unlike the parametric method, we avoid the complexities associated with determining the number of parameters needed for a satisfactory approximation of the target shape. 
The parametric approach, as presented in \cite{CaubetDambrineKatebTimimoun2013}, faces challenges in accurately determining this number. 
Insufficient parameters may hinder the detection of non-trivial shapes, while an excess may result in functional degeneracy. 
Consequently, this approach, in our view, poses implementation difficulties, especially in three dimensions (3D).}
\item In our numerical simulations, we use lower-degree finite elements, in contrast to \cite{CaubetDambrineKatebTimimoun2013}, which employed higher-degree elements (refer to Remark \ref{rem:lower_degree_finite_elements}). 
It is crucial to highlight that utilizing higher-degree finite elements, like $P4/P3$ finite elements, for computing velocity-pressure variables in the Stokes equations, especially in 3D cases, is computationally expensive.
\item {The work in \cite{KovtunekoOhtsuka2022} solely computes the shape derivative of a Lagrangian functional without including any numerical experiments. 
Additionally, to the best of our knowledge, prior numerical investigations on shape inverse obstacle problems have exclusively concentrated on two-dimensional (2D) settings (refer to, e.g., \cite{CaubetDambrineKatebTimimoun2013}).
In contrast, our current study extends this scope by evaluating the effectiveness of our proposed methodology in 3D cases. 
This further warrants the execution of the present work.}
\end{itemize}

The outline of this paper is organized as follows.
In Section \ref{sec:preliminaries}, we present the coupled complex boundary method appropriate for our inverse obstacle problem. 
We also reformulate it into a shape optimization setting, introducing the least-squares fitting for the imaginary part of the complex PDE's solution.
Then, in Section \ref{sec:Computation_of_the_shape_derivative}, we rigorously derive the shape derivative of this cost functional.
After that, in Section \ref{sec:numerical_experiments}, we provide an algorithm based on the Sobolev gradient method with the finite element scheme. We then offer various numerical examples in two and three spatial dimensions to illustrate the feasibility and effectiveness of the proposed method.
We conclude the paper with a brief summary in Section \ref{sec:conclusion}.
Finally, Appendix \ref{appendix:lemmata_proofs} is dedicated to proving the key lemmas used to derive the shape gradient via the rearrangement method.
%
%
%
%
%
\section{Preliminaries}\label{sec:preliminaries}
Our immediate goal is to rephrase the inverse problem \eqref{eq:cost_tracking_Dirichlet_data} using CCBM-based shape optimization. 
To start, we express \eqref{eq:overdetermined_original} using the concept of CCBM and analyze the well-posedness of the resulting PDE system. 
To save space, we introduce notations aligned with those in \cite{RabagoNotsu2024}.

\textbf{Notations.} Let $\partial_{i} := \partial/\partial x_{i}$ and $\nabla:= \begin{pmatrix} \partial_{1}, \ldots, \partial_{d} \end{pmatrix}^{\top}$. 
We define $\dn{} := \partial/\partial \nn  = \sum_{i=1}^{d} n_{i} \partial_{i}$, where $\nn = (n_{1}, \ldots, n_{d})^{\top} \in \mathbb{R}^{d}$ is the outward unit normal vector to $\Omega$.
The inner product of column vectors $\vect{a}, \vect{b} \in \mathbb{R}^{d}$ is denoted by $\vect{a} \cdot \vect{b} := \vect{a}^{\top} \vect{b} \equiv \langle \vect{a}, \vect{b} \rangle_{\mathbb{R}^{d}} = \langle \vect{a}, \vect{b} \rangle$, the latter being used when there is no confusion. 
For simplicity, we occasionally use $\varphi_{n}$ instead of $\bphi \cdot \nn$, where $\bphi$ is a vector-valued function.
For a vector-valued function $\uu := (u_{1}, \ldots, u_{d})^{\top} : \Omega \to \mathbb{R}^{d}$, $(\nabla \uu)_{ij} := (\partial_{i} u_{j})_{i,j = 1, \ldots, d}$ represents its gradient, while its Jacobian is denoted by $(D\uu)_{ij} = (\partial_{j} u_{i})_{i,j=1,\ldots,d} \equiv \nabla^{\top} \uu$. Accordingly, we use $\partial_{\nn} \uu := (D\uu)\nn$.

Let $1 \leqslant p \leqslant \infty$ and $m \in \mathbb{N} \cup \{0\}$. 
The function spaces $W^{m,p}(\Omega)$, $W^{0,p}(\Omega)=L^{p}(\Omega)$, and $H^{m}(\Omega) = W^{m,2}(\Omega)$ denote the standard real Sobolev spaces equipped with their natural norms, also expressed in usual notations \cite[Chap. IV]{DautrayLionsv21998} (e.g., $\|\cdot\|_{H^{1}(\Omega)}$ is the standard $H^{1}(\Omega)$-norm).
We define the space $H^{m}(\Omega)^{d} := \{\uu: \Omega \to \mathbb{R}^{d} \mid u_{i} \in H^{m}(\Omega), \text{for $i = 1, \ldots, d$}\}$ with norm $\|\uu\|_{H^{m}(\Omega)^{d}} := ( \sum_{i=1}^{d} |u_{i}|^{2}_{H^{m}(\Omega)} )^{1/2}$.
A similar definition is given when $\Omega$ is replaced by $\partial\Omega$, with $m$ as a rational number.

In this paper, we let $\HH^{m}(\Omega)^{d}$ be the complex version of $H^{m}(\Omega)^{d}$ with the inner product $(\!(\cdot, \cdot)\!)_{m,\Omega,d}$ and norm $\vertiii{\, \cdot\, }_{m,\Omega,d}$ defined respectively as follows: 
for all $\uu, \vv \in \HH^{m}(\Omega)^{d}$, $(\!( \uu, \vv )\!)_{m,\Omega,d} = \sum_{j=1}^{d}(\uu, \cv)_{m,\Omega}$ and $\vertiii{\vv}_{m,\Omega,d} = \sqrt{ (\!( \vv,\vv )\!)_{m,\Omega, d}}$.
Note that, for real valued functions, $\vertiii{\cdot}_{m,\Omega,d} = \|\cdot\|_{W^{m,p}(\Omega)}$.
Here, the operation `$:$' stands for the Frobenius inner product and is defined as $\nabla \vv : \nabla \cw = \sum_{i,j=1}^{d} {\partial_{i} {v}_{j}} {\partial_{i} {\overline{w}}_{j}} \in \mathbb{R}$. 
With these definitions, we sometimes write $\vertiii{\vv}_{X} = \vertiii{\vv}_{1,\Omega,d}$ and $\vertiii{q}_{Q} = \vertiii{q}_{0,\Omega,d}$, dropping $d$ for convenience. 
For ease of writing, we use the function space notations $X := \HH^{1}(\Omega)^{d}$, ${\Vgamma} := \HH^{1}_{\Gamma,\vect{0}}(\Omega)^{d} = \{ \bphi \in \HH^{1}(\Omega)^{d} \mid \bphi = \vect{0} \ \text{on $\Gamma$}\}$, $Q:= \LL^2(\Omega)$, and $\oQ = \left\{ q \in Q \mid \intO{q} = 0\right\}$.

We use the symbol $c$ to represent a positive constant, and its specific value may vary in different contexts. 
The symbol `$\lesssim$' denotes that if $x \lesssim y$, then there exists a constant $c>0$ such that $x \leqslant c y$. 
Similarly, $y \gtrsim x$ is defined as $x \lesssim y$. 
Lastly, we adhere to standard notations, emphasizing them only for clarity.
%
%
%
%
%
%
%
%
%
\subsection{The CCBM formulation}
The CCBM formulation of \eqref{eq:overdetermined_original} is given by the complex PDE system
\begin{equation}
\label{eq:ccbm}
	-{\alpha} \Delta {\uu} + \nabla p 	= {\vect{0}} \ \text{  in } \Omega,\quad
	\dive{\uu}					= {\vect{0}} \ \text{ in } \Omega,\quad
	\sigma(\uu,p)\nn+ i \uu 		= {{\bgg} + i \ff} \ \text{ on }\Sigma, \quad
 	\uu						= {\vect{0}} \ \text{ on }\Gamma,
\end{equation}
where $i=\sqrt{-1}$ is the imaginary unit.
We call the solution of \eqref{eq:ccbm} as \textit{states} or \textit{state solutions}.

Let $(\uu,p) = ({\realu}+i{\imagu}, \realp+i{\imagp})$\footnote{This notation is used hereon without further notice.} be a solution pair to \eqref{eq:ccbm}.
Then, the real-valued functions ${\realu}$, ${\imagu}$, $\pr$, and $\pim$ satisfy the following systems of real PDEs:
\begin{equation}
\label{eq:real_u}
	-{\alpha} \Delta {\realu}+\nabla \realp = \vect{0} \ \text{ in } \Omega,\quad
	\dive{{\realu}} = 0 \   \text{ in } \Omega,\quad
	 \sigma({\realu},\realp)\nn = {\bgg} + {\imagu} \ \text{ on }\Sigma, \quad
 	{\realu} = \vect{0} \   \text{ on }\Gamma;
\end{equation}
\begin{equation}
\label{eq:imaginary_u}
	-{\alpha} \Delta {\imagu}+\nabla {\imagp}  = \vect{0} \ \text{ in }\Omega,\quad
	\dive{{\imagu}}	= 0 \ \text{ in } \Omega,\quad
	\sigma({\imagu},{\imagp})\nn = \ff - {\realu}  \ \text{ on }\Sigma, \quad
	{\imagu} =  \vect{0} \ \text{ on }\Gamma.
\end{equation} 
\begin{remark}\label{rem:equivalent}
If ${\imagu} = \vect{0}$ and ${\imagp}=0$ in $\Omega$, then ${\imagu}=\partial_{n}{\imagu}=\vect{0}$ on $\Sigma$. 
As a result, from the PDE systems \eqref{eq:real_u} and \eqref{eq:imaginary_u}, $(\omega,{\realu},\realp)$ is a solution to the original problem \eqref{eq:overdetermined_original}.
Conversely, if $(\omega,\uu,p)$ is a solution of \eqref{eq:overdetermined_original}, then clearly it satisfies \eqref{eq:ccbm}.
\end{remark}
Remark \ref{rem:equivalent} implies that Problem \ref{problem:gip} can equivalently be formulated as follows:
\begin{problem}{\label{prob:optimization_problem}}
Find  $\omega \in \mathcal{O}_{\delta}$ such that ${\imagu} = \vect{0}$ and ${\imagp}=0$ in $\Omega$, where $(\uu, p)$ satisfy \eqref{eq:ccbm}.
\end{problem}
%
%
%
%
%
%
%
%
%
\subsection{Well-posedness of the CCBM formulation}
Here, we briefly examine the well-posedness of the complex PDE system \eqref{eq:ccbm}. 
To do so, we introduce the following forms:
\begin{equation}\label{eq:forms_for_the_state_problem}
\left\{
	\begin{aligned}
		\aaa({\bphi},{\bpsi}) &= \intO{{\alpha} \nabla {\bphi} : \nabla {\cbpsi}} + i \intS{ {\bphi} \cdot {\cbpsi}}, \quad \text{where}\ {\bphi}, {\bpsi} \in {\Vgamma},\\
		b(\bphi,\lambda) &= -\intO{\lambda ( \nabla \cdot \cbphi) }, \quad \text{where}\ {\bphi}\in {\Vgamma}, \ \lambda \in {Q}, \\
		F(\bpsi) &= \intS{(\bgg + i \ff ) \cdot \cbpsi},  \quad \text{where}\ {\bpsi}\in {\Vgamma}.
 	\end{aligned}
\right.
\end{equation}
With the above forms, we can state the weak formulation of \eqref{eq:ccbm} as follows:
    \begin{equation}\label{eq:ccbm_weak_form}
    \text{Find\quad $({\uu},p) \in \Vgamma \times Q$\quad such that \quad }
    \left\{\arraycolsep=1.4pt\def\arraystretch{1.1}
    \begin{array}{rcll}
    	\aaa({\uu},{\bphi}) + b(\bphi,p)	&=& F(\bphi),		&\ \forall {\bphi}\in {\Vgamma},\\ 
    				b(\uu,\lambda)	&=& 0,			&\ \forall \lambda \in {Q}.\\ 
    \end{array}
    \right.
    \end{equation}
The well-posedness of the above variational problem is issued in the following lemma. 
\begin{lemma}\label{prop:well_posedness_of_CCBM}
	Let $(\ff,{\bgg}) \in H^{1/2}(\Sigma)^{d} \times H^{-1/2}(\Sigma)^{d}$ be given. 
	Then, problem \eqref{eq:ccbm_weak_form} admits a (unique) solution $(\uu,p) \in {\Vgamma} \times Q$.
	This solution depends continuously on the data and we have
	\[
		\vertiii{\uu}_{X} , \ \vertiii{p}_{Q} 
		\lesssim \vertiii{\ff}_{H^{1/2}(\Sigma)^{d}} + \vertiii{\bgg}_{H^{-1/2}(\Sigma)^{d}}
		=:\norm{(\ff,\bgg)}_{1/2,\Sigma}.
	\] 
\end{lemma}
The above claim can be validated through standard steps (cf. \cite[Appx. A]{RabagoNotsu2024}).
In fact, it follows from the continuity of the sesquilinear form $a(\cdot,\cdot)$ on $\Vgamma \times \Vgamma $, its coercivity on $\Vgamma \subset X$ (i.e., there is a constant $c_{a}>0$ such that the relaxed $\Vgamma$-ellipticity assumption $\Re(a(\bphi,\bphi)) \geqslant c_{a} \vertiii{\bphi}_{X}$, for all $\bphi \in {\Vgamma}$), the boundedness of $F$, and on an inf-sup condition; that is, there exists a constant $\beta_{0} > 0$ such that
\begin{equation}\label{eq:inf_sup_nts}
	\inf_{\substack{{\lambda} \in {Q}\\ {\lambda}\neq0}} \sup_{\substack{\bphi \in {\Vgamma}\\ \bphi \neq \vect{0}}} \frac{b(\bphi,{\lambda})}{\vertiii{\bphi}_{X} \vertiii{{\lambda}}_{Q} } \geqslant \beta_{0}.
\end{equation}

\subsection{The shape optimization problem}
We now present the solution method we use to resolve the shape inverse problem \ref{prob:optimization_problem}.
To this end, we introduce the cost functional 
\begin{equation}\label{eq:ccbm_cost_function}
	J(\Omega):=J(D\setminus \overline{\omega})
	=\frac{1}{2} \Big\{\big\Vert {\imagu} \big\Vert_{L^{2}(\Omega)^{d}}^{2}+\big\Vert {\imagp} \big\Vert_{L^{2}(\Omega)}^{2}\Big\}
\end{equation}
and consider the following shape optimization problem:
\begin{problem}\label{prob:shape_optimization_problem}
Find $\omega^{\ast} \in \mathcal{O}_{\delta}$ such that $\omega^{\ast} = \displaystyle \inf_{\omega \in \mathcal{O}_{\delta}} J(\Omega)$.
\end{problem} 
\begin{remark}
CCBM enables us to define the cost function $J$ across the entire domain $\Omega$, offering advantages in robustness for reconstruction, such as the Kohn-Vogelius method, when compared to the cost functionals of boundary-integral-least-squares fitting, namely $J_{N}$ and $J_{D}$.
\end{remark}
To find the shape derivative of the cost function \eqref{eq:ccbm_cost_function} with respect to $\omega$ (or equivalently, to $\Omega$), we can use the classical chain rule approach (see, e.g., \cite{DelfourZolesio2011, HenrotPierre2018}). This method relies on having the expression for the shape derivative of the state problem \eqref{eq:ccbm}. 
However, obtaining the shape derivative of the states requires higher regularity assumptions on the unknown variables, typically derived from the regularities imposed on the domain and on the boundary data. 
Our goal here is to avoid such unnecessary requirements for the first-order shape sensitivity of $J$.
Consequently, we establish the shape gradient of $J$ -- following the Hadamard-Zol\'{e}sio structure theorem \cite[p. 479]{DelfourZolesio2011} -- using only a mild $\mathcal{C}^{1,1}$ regularity assumption on the domain, combined with the assumption that $(\ff, \bgg) \in H^{3/2}(\Sigma)^{d} \times H^{1/2}(\Sigma)^{d}$. 
This task is achieved through the application of an alternative technique in computing shape derivatives of shape functionals known as the \textit{rearrangement method} \cite{IKP2006,IKP2008}. 
This approach bypasses the computation of the material and shape derivative of the states and utilizes only the H\"{o}lder continuity of the state variables to obtain the expression for the shape derivative of the cost.
It is worth noting that, despite its application in computing shape derivatives for various problems, such as shape optimization reformulations of free boundary problems \cite{IKP2006,HIKKP2009,BacaniPeichl2013,RabagoNotsu2024}, shape optimal problems \cite{IKP2008,KasumbaKunisch2011,KasumbaKunisch2014}, and shape design problems \cite{SimonNotsu2022a,SimonNotsu2022b}, there is, to the best of our knowledge, no documented demonstration of the method applied to shape inverse problems. 
This observation serves as a compelling motivation for applying this method to our present problem.
%
%
%
%
%
\section{First-order shape sensitivity analysis}\label{sec:Computation_of_the_shape_derivative}
To iteratively solve Problem \ref{prob:shape_optimization_problem} through a numerical procedure, our goal is to characterize the shape gradient of the cost function $J$, as enunciated in Theorem \ref{thm:shape_derivative}. 
To facilitate this, we will make necessary preparations by recalling some relevant concepts from shape calculus.
\subsection{Some concepts from shape calculus}
Let us define $T_{t}$ as the \textit{perturbation of the identity} $id$ given by the map
\begin{equation}\label{eq:poi}
	T_{t} = T_{t}({\VV}) = id + t \VV, \qquad (T_{t} : D \longmapsto D),
\end{equation}
where $\VV$ is a $t$-independent deformation field belonging to the admissible space
\begin{equation}
\label{eq:space_for_V}
	\sfTheta:=\{\VV = (\theta_{1}, \ldots, \theta_{d})^{\top} \in \mathcal{C}^{1,1}(\overline{D})^{d} \mid \operatorname{supp}{\VV} \subset {\overline{D}}_{\delta}\},
\end{equation}
where $\{x \in {D} \mid d(x,\partial{D}) > \delta/2\} \subset {D}_{\delta} \subset \{x \in {D} \mid d(x,\partial{D}) > \delta/3\}$.
Clearly, by definition, we have that $\Sigma_{t}:=T_{t}(\Sigma) \equiv \Sigma$ and $\Gamma_{t}:=T_{t}(\Gamma)$.
In addition, $\Omega_{0}=\Omega$ and $\Sigma_{0}=\Sigma$.

Hereinafter, we assume that $t > 0$ is sufficiently small such that $T_{t}$ is a diffeomorphism from $\Omega \in \mathcal{C}^{1,1}$ onto its image.
Specifically, we let $\varepsilon \in \mathbb{R}^{+}$ be small enough so that $[t \mapsto T_{t}] \in \mathcal{C}^{1}(\mathcal{I},\mathcal{C}^{1,1}(\overline{D})^{d})$, where $\mathcal{I} := [0,\varepsilon]$. 
Accordingly, we define the set of all admissible perturbations of $\Omega$ denoted by $\mathcal{O}_{ad}$ as follows:
\begin{equation}\label{eq:admissible_domains}
	\mathcal{O}_{ad} = \left\{T_{t}({\VV})(\overline{\Omega}) \subset D \mid \Omega = D\setminus{\overline{\omega}} \in \mathcal{C}^{1,1}, \omega \in \mathcal{O}_{\delta}, t \in \mathcal{I}, \VV \in \sfTheta \right\}. 
\end{equation}
Note that we do not necessarily require the fixed boundary $\Sigma = \partial\Omega$ to be $\mathcal{C}^{1,1}$ regular (it is sufficient for it to be Lipschitz), but for the sake of simplifying the discussion, we assume the stated regularity.

The functional $J : \mathcal{O}_{ad} \to \mathbb{R}$ has a directional \textit{first-order} \textit{Eulerian derivative} at $\Omega$ in the direction of the field $\VV \in \sfTheta$ if the limit $\lim_{t \searrow0} \frac{J(\Omega_{t}) - J(\Omega)}{t} =: {d}J(\Omega)[\VV]$ exists (see, e.g., \cite[Sec. 4.3.2, Eq. (3.6), p. 172]{DelfourZolesio2011}). 
The function $J$ is considered \textit{shape differentiable} at $\Omega$ in the direction of $\VV$ if the mapping $\VV \mapsto {d}J(\Omega)[\VV]$ is both linear and continuous. 
In such cases, we call it the \textit{shape gradient} of $J$.

We introduce a few more notations for ease of writing.
The Jacobian matrix of $T_{t}$, its inverse, and its inverse transpose are denoted respectively by $DT_{t}$, $(DT_{t})^{-1}$, and $(DT_{t})^{-\top} := ((DT_{t})^{\top})^{-1}$.
For convenience, we define
\[
	\dett := \det \, DT_{t},\quad
	A_{t} := \dett(DT_{t}^{-1})(DT_{t})^{-\top},\quad
	\text{and} \quad B_{t} := \dett \abs{\Mt \nn}, \quad \Mt:=(DT_{t})^{-\top},
\]
and we assume that, for all $t \in \mathcal{I}$, we have
\begin{equation}\label{eq:bounds_At_and_Bt}
	0 \leqslant \Lambda_{1} \leqslant \dett \leqslant \Lambda_{2}
	\qquad \text{and} \qquad 0 < \Lambda_{3}\abs{\xi}^{2} \leqslant A_{t}\xi \cdot \xi \leqslant \Lambda_{4}\abs{\xi}^{2},
\end{equation}
for all $\xi \in \mathbb{R}^{d}$, for some constants $\Lambda_{1}$, $\Lambda_{2}$, $\Lambda_{3}$, and $\Lambda_{4}$ ($\Lambda_{1} < \Lambda_{2}$, $\Lambda_{3} < \Lambda_{4}$).

At $t=0$, $I_{0} = 1$, $DT_{0}^{-1} = id$, $(DT_{0})^{-\top} = id$, $A_{0} = id$, and $B_{0} = 1$. 
The maps $t \mapsto (DT_{t})^{\pm 1}$, $t \mapsto \dett$, and $t \mapsto A_{t}$ are differentiable (see \eqref{eq:regular_maps}), and their respective derivative are given as follows:
\begin{equation}\label{eq:limits_of_maps}
\begin{aligned}
	\frac{d}{dt}(DT_{t})^{\pm 1} \big\rvert_{t=0} 
 		&=\lim_{t \to 0} \dfrac{(DT_{t})^{\pm 1} - id}{t} =\pm D\VV,\qquad\qquad
	\frac{d}{dt}\dett \big\rvert_{t=0}
		= \lim_{t\to 0} \frac{\dett - 1}{t} = \dive \VV,\\
	\quad \frac{d}{dt}A_{t} \big\rvert_{t=0}
		& = \lim_{t\to 0} \frac{A_{t} - id}{t}
		= (\dive \VV){id} -  {D} \VV - ({D} \VV)^\top =: A.
\end{aligned}
\end{equation}
%
%
We also introduce the following notations for economy of space:
\begin{equation}\label{eq:special_notations}
	\mathfrak{i}_{t} := I_{t} - 1,
	\qquad \mathfrak{m}_{t} := M_{t}^{\top} - id,
	\qquad \mathfrak{a}_{t} := A_{t} - id,
	\qquad \text{and} \quad \Vn := \VV \cdot \nn. 
\end{equation}
%
%
%

Since we intend to map a function $\bphi_{t} : \Omega_{t} \to D \subset \mathbb{R}^{d}$ to a reference domain using the transformation $T_{t}$, we will utilize the following notation: 
\[
	\bphi^{t}:=\bphi_{t} \circ T_{t}:\Omega \to D \subset \mathbb{R}^{d}.
\]

To conclude this review and preparation section, we present some results regarding the limits of functions composed with the map $T_{t}$. 

For a vector-valued function $\bphi = (\varphi_{1}, \ldots, \varphi_{d})^{\top}\in H^{2}(D)^{d}$, we have
%
%
\begin{equation}\label{eq:convergence_of_vector_valued_functions_1}
	\lim_{t \to 0} \|\bphi \circ T_{t} - \bphi\|_{H^{1}(D)^{d}} 
	= \lim_{t \to 0} \left( \sum_{i=1}^{d} \|\varphi_{i} \circ T_{t} - \varphi_{i} \|_{H^{1}(D)} \right)^{1/2} 
	= 0.
\end{equation}
%
%
Moreover, the maps $t \mapsto \bphi \circ T_{t}$ from $\mathcal{I} \to H^{1}(\Omega)^{d}$ and $t \to \dett \bphi \circ T_{t}$ from $\mathcal{I}$ to $L^{2}(\Omega)^{d}$ are differentiable at $t=0$, and we have
\begin{equation}\label{eq:convergence_of_vector_valued_functions_3}
				\lim_{t \to 0} \frac{1}{t} \left( \bphi\circ T_{t} - \bphi\right) = D\bphi\VV
				\quad\text{and}\quad
				\lim_{t \to 0} \frac{1}{t} \left( \dett\bphi\circ T_{t} - \bphi\right) = \vect{\nabla} \cdot(\bphi \otimes \VV),
\end{equation}
where $\vect{\nabla} \cdot(\bphi \otimes \VV) = ( \operatorname{div}(\varphi_{1}\VV), \operatorname{div}(\varphi_{2}\VV), \ldots, \operatorname{div}(\varphi_{d}\VV) )^{\top}$, and $\otimes$ denotes the outer product; i.e., $\bphi \otimes \VV=\bphi \VV^{\top} = (\varphi_{j} \theta_{k})_{jk}$ where $j,k = 1, \ldots, d$.

The proofs of the above results are omitted since they are standard and follow similar arguments used for the case of scalar functions, see, e.g., \cite{DelfourZolesio2011,IKP2006}.
\subsection{The shape gradient of $J$}
Our main result in this section is the expression for the shape gradient of $J$, which we present in the following theorem.
%
\begin{theorem}
	\label{thm:shape_derivative} 
	Let $(\ff, \bgg) \in H^{3/2}(\Sigma)^{d} \times H^{1/2}(\Sigma)^{d}$ be a given Cauchy pair, 
	$\Omega$ is an admissible domain, and $\VV \in \sfTheta$.
	Then, $J$ is shape differentiable with respect to $\Omega$ in the direction of the deformation field $\VV$. 
	Its shape derivative is given by
	\begin{equation}\label{eq:shape_gradient}
	\begin{aligned}
	{d}J(\Omega)[\VV] 
	&= \lrangle{\ggb\nn}{\VV}_{\Gamma}
	= \intG{\left( \Im\left\{ \overline{\sigma(\vv, q)\nn} \cdot \dn{\uu} \right\} + \frac12 |\pim|^{2} \right)\nn \cdot \VV},
	\end{aligned}
	\end{equation}
	where $\Im$ denotes the imaginary part, $(\uu, p)$ is the unique pair of solution to \eqref{eq:ccbm}, while the pair of adjoints $(\vv,q) = (\vr + i \vi, q_{r} + i q_{i})$ uniquely solves the adjoint system
    \begin{equation}
    \label{eq:adjoint_system}
        	- {\alpha} \Delta \vv + \nabla q	= \ui			\ \text{in $\Omega$},\quad
    	-\nabla \cdot \vv			= \pim 		\ \text{in $\Omega$},\quad
    	\vv	 					= \vect{0}		\ \text{on $\Gamma$},\quad
    	\sigma(\vv,q)\nn - i \vv		= \vect{0}		\ \text{on $\Sigma$},
    \end{equation}	 
    with compatibility condition $\langle \vv, \nn \rangle_{\Sigma} = -(\pim, 1)_{\Omega}$.
\end{theorem}
We emphasize here that if, instead, $\Omega$ is of class $\mathcal{C}^{2,1}$, $(\mathbf{f}, \bgg) \in H^{5/2}(\Sigma)^{d} \times H^{3/2}(\Sigma)^{d}$, and $\VV$ is a $\mathcal{C}^{2,1}$ smooth deformation field, then the existence of the material and shape derivatives of the states is guaranteed.
Because then we will have  $\ur$, $\ui \in H^{3}(\Omega)^{d}$ and $p_{r}$, $\pim \in H^{2}(\Omega)$.
With this regularity, the shape gradient of the cost can be computed with ease using Hadamard's domain differentiation formula
(see, e.g., \cite[Thm. 4.2, p. 483]{DelfourZolesio2011}), \cite[eq. (5.12), Thm. 5.2.2, p. 194]{HenrotPierre2018} or \cite[eq. (2.168), p. 113]{SokolowskiZolesio1992}):
\begin{equation}
\begin{aligned}
	\left. \left\{ \frac{{d}}{{d}t} \intOt{f(t,x)} \right\} \right|_{t=0}
		= \intO{\frac{\partial }{\partial t} f(0,x)}
			+ \intdO{ f(0,\sigma) \Vn} \label{eq:Hadamard_domain_formula}
\end{aligned}
\end{equation} 
Before we provide a rigorous proof of Theorem \ref{thm:shape_derivative} , let us exhibit shortly the shape gradient of $J$ via the chain rule approach.
Let us assume, for now, that $\Omega$ is an admissible domain of class $\mathcal{C}^{2,1}$, and $\VV\in \mathcal{C}^{2,1}(\overline{D})^{d}$ such that $\operatorname{supp}{\VV} \subset {\overline{D}}_{\delta}$, and that $(\ff, \bgg) \in H^{5/2}(\Sigma)^{d} \times H^{3/2}(\Sigma)^{d}$.
Then, we can apply formula \eqref{eq:Hadamard_domain_formula} to obtain -- noting that $\VV= \vect{0}$ on $\Sigma$ and $\uu = \vect{0}$ on $\Gamma$ -- the derivative
	\begin{equation}
	\label{eq:computed_first_derivative_via_Hadamard_formula}
	\begin{aligned}
		{{d}}J(\Omega)[\VV]   
		= \intO{\left( \ui \cdot \ui'  + \pim \pim' \right) } + \frac12 \intG{|\pim|^{2} \Vn}
		=: L_{1} + L_{2}.
	\end{aligned}
	\end{equation}
Here, the shape derivative $(\uup,p^{\prime})$ of the states $(\uu,p)$ is characterized by the PDE system
\begin{equation}
\label{eq:shape_derivative_of_the_state_system}
	- {\alpha} \Delta \uup + \nabla p^{\prime}	= \vect{0}		\ \text{in $\Omega$},\quad
	\nabla \cdot \uup				= \vect{0} 			\ \text{in $\Omega$},\quad
	\sigma(\uup,p^{\prime}) + i \uup		= \vect{0}			\ \text{on $\Sigma$},	\quad
	\uup	 						= - \dn{\uu}\Vn		\ \text{on $\Gamma$}.
\end{equation}
We can confirm this claim using standard techniques from shape calculus \cite{DelfourZolesio2011,HenrotPierre2018}, with arguments similar to the proof in \cite[Thm. C.1]{RabagoNotsu2024}, so we omit it.

We focus on eliminating the derivatives in the first integral $L_{1}$.
To do this, we make use of the adjoint method, utilizing the adjoint system \eqref{eq:adjoint_system}.

Let us multiply by $\overline{\uu}'$ equation \eqref{eq:adjoint_system}, apply integration by parts (IBP), and then use the boundary conditions in \eqref{eq:adjoint_system} to obtain
\[
	\intO{{\alpha} \nabla{\vv} : \nabla{\overline{\uu}'}} - i \intS{\vv\cdot \overline{\uu}'} - \intG{ ({\alpha} \dn{\vv} - q\nn) \cdot \overline{\uu}' } = \intO{\ui \cdot \overline{\uu}'}.
\]
Taking the complex conjugate on both sides of the above equation, we get (using the sesquilinear form in \eqref{eq:forms_for_the_state_problem}) 
$
	a(\uup,\vv) - \intG{ \overline{({\alpha} \dn{\vv} - q\nn)} \cdot \uup } = \intO{\ui \cdot \uup}.
$
Meanwhile, multiplying equation \eqref{eq:shape_derivative_of_the_state_system} by $\overline{\vv}$ and then employing IBP, while taking into account the boundary conditions in \eqref{eq:shape_derivative_of_the_state_system}, we get
$
	a(\uup,\vv) = -\intO{\pim p^{\prime}}.
$
Subtracting this equation from the previous one, we see that
\begin{align*}
	L_{1} = \Im\left\{ \intO{\left( \uu \cdot \ui'  + p \pim' \right) } \right\} 
	&= \Im\left\{ \intG{ \overline{\sigma(\vv,q)\nn} \cdot \dn{\uu}\Vn } \right\},
\end{align*}
from which, after combining with $L_{2}$, we recover the expression in \eqref{eq:shape_gradient}.

Our next objective is to demonstrate that applying the rearrangement method yields the same expression, with the only requirement being mild $\mathcal{C}^{1,1}$ regularity assumptions on $\Omega$, with perturbations in $\mathcal{O}_{ad}$, and regularity conditions $(\mathbf{f}, \bgg) \in H^{3/2}(\Sigma)^{d} \times H^{1/2}(\Sigma)^{d}$.
\begin{remark}
The conclusion in Theorem \ref{thm:shape_derivative} holds even when $(\mathbf{f}, \bgg) \in H^{1/2}(\Sigma)^{d} \times H^{-1/2}(\Sigma)^{d}$. 
In this case, the expression of $\lrangle{\ggb}{\Vn}_{\Gamma}$ has to be seen as a duality product $H^{1/2} \times H^{-1/2}$.
\end{remark}
\subsection{Computation of the shape gradient: Proof of Theorem \ref{thm:shape_derivative} }
\label{subsec:shape_derivative_of_the_cost_using_{t}he_Eulerian_derivatives}
We are now in the position to compute the shape derivative of $J$ rigorously via rearrangement method given that $(\ff, \bgg) \in H^{3/2}(\Sigma)^{d} \times H^{1/2}(\Sigma)^{d}$, $\Omega = D\setminus\overline{\omega}$ of class $\mathcal{C}^{1,1}$ such that $\omega \in \mathcal{O}_{\delta}$, and $\VV \in \sfTheta$.
The discussion then lays out the proof of Theorem \ref{thm:shape_derivative}, closely following \cite{RabagoNotsu2024}.

We start with the weak formulation of \eqref{eq:adjoint_system} which is given as follows:	
        \begin{equation}\label{eq:adjoint_system_weak_form}
	\text{find \quad$({\vv},q) \in \Vgamma \times Q$\quad such that \quad}
        \left\{\arraycolsep=1.4pt\def\arraystretch{1.1}
        \begin{array}{rcll}
        	\tilde{\aaa}({\vv},{\bpsi}) + b(\bpsi,q)	&=& \tilde{F}(\bpsi),	& \ \forall {\bpsi}\in {\Vgamma},\\ 
        				b(\vv,\mu)	&=& (\mu, \pim),			& \ \forall \mu \in {Q},
        \end{array}
        \right.
        \end{equation}
    where
    \begin{equation}\label{eq:forms_for_the_adjoint_problems}
    \left\{
    	\begin{aligned}
    		\tilde{\aaa}({\bphi},{\bpsi}) &= \intO{{\alpha} \nabla {\bphi} : \nabla {\cbpsi}} - i \intS{{\bphi} \cdot {\cbpsi}}, \quad \text{with}\ {\bphi}, {\bpsi} \in {\Vgamma},\\ 
    		\tilde{F}(\bpsi) &= \intO{\ui\cdot \cbpsi},  \quad \text{with $\ui \in H^{1}(\Omega)^{d}$ and ${\bpsi}\in {\Vgamma}$}.
     	\end{aligned}
    \right.
    \end{equation}
    %
    %
The demonstration of the well-posedness of \eqref{eq:adjoint_system_weak_form} follows a standard procedure. It mirrors the proof outlined in Proposition \ref{prop:well_posedness_of_CCBM}, anchoring on the distinct solvability of \eqref{eq:ccbm_weak_form} within the space $\Vgamma \times Q$. 
Notably, it can be established that the variational problem satisfies an inf-sup condition.

%
%
%
%

At this juncture, we note that for $\Omega$ of class $\mathcal{C}^{k+1,1}$, where $k \in \mathbb{N}_{0}$, and assuming $(\ff, \bgg) \in H^{k+3/2}(\Sigma)^{d} \times H^{k+1/2}(\Sigma)^{d}$, it can be demonstrated that the weak solution $(\uu, p) \in \Vgamma \times Q$ of the variational problem \eqref{eq:ccbm_weak_form} also belongs to $\HH^{k+2}(\Omega)^{d} \times \HH^{k+1}(\Omega)$ -- specifically, $\ur, \ui \in H^{k+2}(\Omega)^{d} \cap H^{1}_{\Gamma,0}(\Omega)^{d}$ and $\pr, \pim \in H^{k+1}(\Omega)$. Consequently, the weak solution $(\vv,q)$ of problem \eqref{eq:adjoint_system_weak_form} is not only in $\Vgamma \times Q$ but also lies within $\HH^{k+2}(\Omega)^{d} \times \HH^{k+1}(\Omega)$.

Next, we introduce the sesquilinear forms $\aat, \taat \in \Vgamma \times \Vgamma \to \mathbb{R}$ and linear forms, $b^{t} : \Vgamma \times Q \to \mathbb{R}$ and $F^{t}, \tilde{F}^{t} : \Vgamma \to \mathbb{R}$:
%
%
\begin{equation}\label{eq:transformed_forms}
\left\{ 
\begin{aligned}
		\aat({\bphi},{\bpsi}) &= \intO{{\alpha} A_{t} \nabla {\bphi} : \nabla {\cbpsi}} + i \intS{ {\bphi} \cdot  {\cbpsi}},\qquad
		b^{t}(\bphi,\lambda) = -\intO{ \dett \bar{\lambda} ( \Mt^{\top}: \nabla \bphi) }, \\
		F^{t}(\bphi) &= \intS{B_{t} (\bgg^{t} + i \ff^{t} ) \cdot \cbphi} \equiv F(\bphi) = \intS{(\bgg + i \ff ) \cdot \cbphi},\\%
		\taat({\bphi},{\bpsi}) &= \intO{{\alpha} A_{t} \nabla {\bphi} : \nabla {\cbpsi}} - i \intS{ {\bphi} \cdot  {\cbpsi}},\qquad
		\tilde{F}^{t}(\bphi) = \intO{\dett \ui \cdot \cbphi},
\end{aligned}
\right.
\end{equation}
where $\ui$ is the imaginary part of the velocity solution to the state system \eqref{eq:ccbm}.

In the above, notice that $F^{t}(\bphi) = F(\bphi)$ holds true for any ${\bphi}\in {\Vgamma}$ because $\Sigma_{t} \equiv \Sigma$ for all $t > 0$, and thus $T_{t}(x) = x$ for all $x \in \Sigma_{t}$, for all $t > 0$.
This is clear from the fact that $B_{t} = 1$ on $\Sigma$.
As a consequence, we have $\ff^{t} = \ff$ and $\bgg^{t} = \bgg$ on $\Sigma$.
Further below, we will use the given notations to express the respective variational formulations of the transformed versions of the state and adjoint state systems \eqref{eq:ccbm} and \eqref{eq:adjoint_system}.

Before deriving the shape gradient of $J$, we need to establish some key lemmas.
\begin{lemma}\label{lem:boundedness_of_sesquiliner_and_linear_forms}
	For $t \in \mathcal{I}$, the sesquilinear forms $\aat$ and $\taat$ defined on $X \times X$ are bounded and coercive on $X \times X$.
	In addition, the linear forms $F^{t}(\bphi)$ and $\tilde{F}^{t}(\bphi)$ are also bounded.
	Moreover, the bilinear form $b$ satisfies the condition that there is a constant $\beta_{1} > 0$ such that
	\begin{equation}
	\label{eq:inf_sup_condition_for_transformed_problem}
		\inf_{\substack{{\lambda} \in {Q}\\ {\lambda}\neq0}} \sup_{\substack{\bphi \in {\Vgamma}\\ \bphi \neq \vect{0}}} \frac{b^{t}(\bphi,{\lambda})}	{\vertiii{\bphi}_{X} \vertiii{{\lambda}}_{Q} } \geqslant \beta_{1}.
	\end{equation}

\end{lemma}
\begin{proof}
The proof follows a parallel line of reasoning as presented in \cite[Appx. A]{RabagoNotsu2024}, so we omit it. 
However, it is important to note that one must consider the properties of $A_{t}$ and $B_{t}$ given in \eqref{eq:regular_maps}, along with the bounds specified in \eqref{eq:bounds_At_and_Bt}, to derive the desired inequality.
\end{proof}
The following two lemmas address the well-posedness of the transported perturbed version of \eqref{eq:ccbm_weak_form} and the (uniform) boundedness of its solution, respectively.
Regarding the result, it is important to note that the compatibility condition $\intSt{\uut \cdot \nnt} = 0$, assumed hereafter, is equivalent to \eqref{eq:compatibility_condition} because $\Sigma$ is invariant.
\begin{lemma}\label{lem:transported_problem}
	The pair $(\uu^{t},p^{t}) = (\ur^{t} + i \ui^{t},\pr^{t} + i \pim^{t})$ uniquely solves in $\Vgamma \times Q$ the variational problem
        \begin{equation}\label{eq:transformed_ccbm_weak_form}
        \left\{\arraycolsep=1.4pt\def\arraystretch{1.1}
        \begin{array}{rcll}
        	\aat({\uu^{t}},{\bphi}) + b^{t}(\bphi,p^{t})	&=& F^{t}(\bphi),		& \ \forall {\bphi}\in {\Vgamma},\\ 
        					b^{t}(\uu^{t},\lambda)	&=& 0,			& \ \forall \lambda \in {Q}.
        \end{array}
        \right.
        \end{equation}
\end{lemma}	
\begin{lemma}\label{lem:boundedness_of_the_transformed_state}
	For $t \in \mathcal{I}$, the solution pair $(\uu^{t},p^{t})$ of \eqref{eq:transformed_ccbm_weak_form} are uniformly bounded in $X \times Q$.
	More precisely, for all $t \in \mathcal{I}$, we have
	$
		\vertiii{\uu^{t}}_{X}, \
		\vertiii{p^{t}}_{Q} 
		\lesssim \norm{(\ff,\bgg)}_{1/2,\Sigma}$.
\end{lemma}
To complete our preparations, we next issue the H\"{o}lder continuity of the state variables $(\uu^{t}, p^{t})$ with respect to $t$.
\begin{lemma}\label{lem:holder_continuity}
	Let $(\uu,p) \in \Vgamma \times Q$ be the solution of \eqref{eq:ccbm_weak_form}.
	Then, the following limit holds
	\[
		\lim_{t\to0^{+}} \frac{1}{\sqrt{t}} \left( \vertiii{\uu^{t} - \uu}_{X} + \vertiii{p^{t} - p}_{Q} \right) = 0,
	\] 
	where $({\uu^{t}}, p^{t}) \in \Vgamma \times Q$ solves \eqref{eq:transformed_ccbm_weak_form}, for $t \in \mathcal{I}$.
\end{lemma}
{The proofs of Lemma \ref{lem:transported_problem}, Lemma \ref{lem:boundedness_of_the_transformed_state}, and Lemma \ref{lem:holder_continuity}, which closely follow the techniques in \cite{RabagoNotsu2024}, are deferred to Appendix \ref{appendix:lemmata_proofs}.}

Additionally, by the higher regularity property of $\uu^{t}$, $\uu$, $p^{t}$, and $p$, we easily obtain the following corollary of Lemma \ref{lem:holder_continuity}.
\begin{corollary}\label{cor:holder_continuity}
	The state $(\uu,p)$ belongs to $( \HH^{2}(\Omega)^{d} \times \HH^{1}(\Omega) ) \cap ( \Vgamma \cap Q)$.
	Moreover, the following limit holds
	\[
		\lim_{t\to0^{+}} \frac{1}{\sqrt{t}} \left( \vertiii{\uu^{t} - \uu}_{\HH^{2}(\Omega)^{d}} + \vertiii{p^{t} - p}_{\HH^{1}(\Omega)} \right) = 0,
	\] 
	where $({\uu^{t}}, p^{t}) \in ( \HH^{2}(\Omega)^{d} \times \HH^{1}(\Omega) ) \cap ( \Vgamma \cap Q)$ solves \eqref{eq:transformed_ccbm_weak_form}, for $t \in \mathcal{I}$.
\end{corollary}
We are almost ready to present the proof of Theorem \ref{thm:shape_derivative} without relying on the shape derivative of the state. 
We only need to establish additional groundwork by introducing a set of auxiliary results, which were previously established in \cite{DelfourZolesio2011} and \cite{RabagoNotsu2024}. 
The lemmas derived from these results encompass various identities crucial for our purpose and serve as key components in deriving the boundary expression for the shape derivative of $J$, aligning with the Hadamard-Zol\'{e}sio structure theorem \cite[p. 479]{DelfourZolesio2011}.
\begin{lemma}[{\cite{DelfourZolesio2011}}]\label{lem:expansion}
For $\varphi$, $\psi \in H^{2}(\Omega)$ such that $\varphi = 0$ and $\psi = 0$ on $\Sigma$ and $\VV \in \sfTheta$, we have
\[
\begin{aligned}
	- \intO{A\nabla \varphi \cdot \nabla \overline{\psi} }
	&= - \intO{ (\Delta \varphi) \VV \cdot \nabla \overline{\psi} }
		- \intO{ (\Delta \overline{\psi}) \VV \cdot \nabla \varphi }
		+ \intG{\dn{\varphi}(\VV \cdot \nabla \overline{\psi})} \\
	& \  \qquad
		 + \intG{\dn{\overline{\psi}}(\VV \cdot \nabla \varphi)}
		- \intG{(\nabla \overline{\psi} \cdot \nabla \varphi)\Vn}.
\end{aligned}
\]
\end{lemma}
\begin{lemma}[{\cite[Lem. 3.11]{RabagoNotsu2024}}]\label{lem:divergence_identity_limit}
	The state solution $\uu \in X$ of \eqref{eq:ccbm_weak_form} satisfies the equation
	\[
		\lim_{t \to 0} b(\ut - \uu, \lambda) 
		= - \lim_{t \to 0} \intO{\overline{\lambda} \nabla \cdot \left(\frac{\uu^{t} - \uu}{t}\right)}
		= - \intO{\overline{\lambda} (D\VV : \nabla \uu)}, \quad \forall \lambda \in {Q}.
	\]
\end{lemma}
\begin{lemma}[{\cite[Lem. 3.13]{RabagoNotsu2024}}]\label{lem:divergence_expansion}
	For sufficiently smooth $\bphi$ and $\VV$, it holds that
	\begin{equation}\label{eq:divergence_expansion}
		\dive (\nabla\bphi^{\top} \VV) = (\VV \cdot \nabla) (\nabla \cdot \bphi) + D\VV : \nabla \bphi
			= (\VV \cdot \nabla) (\nabla \cdot \bphi) + \nabla\bphi^{\top} : D\VV^{\top}.
	\end{equation}
	%
	%
\end{lemma}	
The next lemma is derived from the arguments presented in the proof of Theorem 3.4 in \cite{RabagoNotsu2024}, in conjunction with the fact that $\vv = \vect{0}$ on $\Gamma$.
\begin{lemma}\label{lem:crucial_identity}
	Let $p$ be the pressure solution of the state system \eqref{eq:ccbm} and $\vv$ be the velocity solution of the adjoint problem \eqref{eq:adjoint_system}.
	Then, for any $\VV \in \sfTheta$, the following identity holds 
	\[
	\intO{\left[ p (\nabla \cdot \VV) (\nabla \cdot \overline{\vv}) + p (\VV \cdot \nabla)(\nabla \cdot \overline{\vv}) \right]}
			 = \intG{ \dn{\overline{\vv} \cdot p\nn \Vn } }
				+ \intO{(\VV \cdot \nabla p ) \pim}.
	\] 
\end{lemma}

We now provide the proof of our main result, Theorem \ref{thm:shape_derivative}.
\begin{proof}[Proof of Theorem \ref{thm:shape_derivative} ]
	The proof essentially proceeds in two steps.
	First, we evaluate the limit $\lim_{t\to0} [ J(\Omega_{t}) - J(\Omega) ]/t$.
	Then, using the regularity of the domain as well as the state and adjoint variables, we characterized the boundary integral expression for the computed limit.

	We begin by applying a change-of-variables formula (see \cite[eq. (4.2), p. 482]{DelfourZolesio2011}):
	\begin{equation}\label{eq:domain_transformation_formula}
		\intOt{\varphi_{t}} = \intO{\varphi_{t} \circ T_{t} \dett} = \intO{\varphi^{t} \dett},
	\end{equation}
	which holds for $\varphi_{t} \in L^{1}(\Omega_{t})$
combined with the identity $\eta^{2} - \zeta^{2} = (\eta - \zeta)^{2} + 2 \zeta (\eta - \zeta)$ to obtain the following sequence of calculations:
\begin{align*}
	J(\Omega_{t}) - J(\Omega)  
	&= \frac{1}{2} \intO{ {\mathfrak{i}}_{t} (\abs{\uu_{i}^{t}}^{2} - \abs{\uu_{i}}^{2}) }	
		+ \frac{1}{2} \intO{ {\mathfrak{i}}_{t} \abs{\uu_{i}}^{2} }	
		+ \frac{1}{2} \intO{ \abs{\yti}^{2} }	
		+ \intO{ \yti \cdot \uu_{i} }	\\
	& \quad  + \frac{1}{2} \intO{ {\mathfrak{i}}_{t} (\abs{\pim^{t}}^{2} - \abs{\pim}^{2}) }	
		+ \frac{1}{2} \intO{ {\mathfrak{i}}_{t} \abs{\pim}^{2} }
		+ \frac{1}{2} \intO{ \abs{\rti}^{2} }	
		+ \intO{ \rti \pim }\\
	&=: \sum_{i=1}^{8}{J}_{i}(t),		
\end{align*}
where $\yti = \uu_{i}^{t} - \uu_{i}$ and $\rti = \pim^{t} - \pim$.
In view of \eqref{eq:limits_of_maps}$_{1}$ and using Lemma \ref{lem:holder_continuity}, we easily infer that $\dot{J}_{1}(0) = \dot{J}_{3}(0) = \dot{J}_{5}(0) = \dot{J}_{7}(0) = 0$.
	The limits in \eqref{eq:limits_of_maps}$_{1}$ also reveal that
	\[
		\dot{J}_{2}(0) + \dot{J}_{6}(0) 
			= \frac{1}{2}  \intO{\operatorname{div}\VV ( \abs{\uu_{i}}^{2} + \abs{\pim}^{2}) }.
	\]
	We expand the above integral expression through the identity $\dive(\psi \bphi) = \psi \dive \bphi + \bphi \cdot \nabla \psi$, and then apply Green's theorem to obtain
	\begin{equation}\label{eq:distributed_gradient_first_part}
	\begin{aligned}
		\dot{J}_{2}(0) + \dot{J}_{6}(0)  		
			&= - \intO{ (\uu_{i} \cdot \nabla \uu_{i}^{\top} \VV + \pim \VV \cdot \nabla \pim) }
				+ \frac{1}{2}  \intG{ \abs{\pim}^{2} \Vn }.							
	\end{aligned}
	\end{equation}
	The calculations for the remaining two expressions, $\dot{J}{4}(0)$ and $\dot{J}{8}(0)$, demand more effort and necessitate the application of the adjoint system \eqref{eq:adjoint_system_weak_form}. 
	Considering that $\yt = \uu^{t} - \uu \in \Vgamma$ and $r^{t} = p^{t} - p \in Q$, we can express
	\begin{align*}
		{J}_{4}(t) + {J}_{8}(t)
		&= \Im\left\{ \intO{{\alpha} \nabla \overline{\vv} \cdot \nabla \yt } + i \intS{ \overline{\vv} \cdot \yt }
			- \intO{ \overline{q} \operatorname{div} \yt}
			- \intO{ r^{t} \operatorname{div} \vv} \right\}\\
		& \equiv \Im\left\{ a(\uu^{t} - \uu, \vv) + b(\vv, p_{t} - p) + b(\uu^{t} - \uu, q) \right\}\\
		&= \Im\{ \Phi^{t}(\vv) + b(\uu^{t} - \uu, q) \},
	\end{align*}
	where the latter equation follows from \eqref{eq:difference_equation} 
	and $\Phi^{t}(\vv)$ is given by (see \eqref{eq:special_notations} for notations)
	\begin{equation}\label{eq:big_Phi_sup_t}
		\Phi^{t}(\bphi)
		= - \intO{ {\alpha} \mathfrak{a}_{t} \nabla {\uu^{t}} : \nabla {\cbphi}}
			+ \intO{ {\mathfrak{i}}_{t} p^{t} ( \Mt^{\top} : \nabla \cbphi) } 
			+ \intO{ p^{t} ( \mathfrak{m}_{t} : \nabla \cbphi ) },
	\end{equation}
	with $\bphi = \vv \in \Vgamma$.
	Using \eqref{eq:limits_of_maps}, \eqref{eq:convergence_of_vector_valued_functions_3}, Lemma \ref{lem:divergence_identity_limit}, and identity \eqref{eq:divergence_expansion}, we obtain the following limit
\begin{equation}\label{eq:distributed_gradient}
\begin{aligned}
		\dot{J}_{4}(0) + \dot{J}_{8}(0)
		&= \lim_{t \to 0} \frac{1}{t}\left[ \Im\{ \Phi^{t}(\vv) + b(\uu^{t} - \uu, q) \} \right]\\
		& = \Im\Big\{ - \intO{ {\alpha} A \nabla {\uu} : \nabla {\overline{\vv}}} 
			+ \intO{ (\nabla \cdot \VV) p ( \nabla \cdot \overline{\vv}) } 
			+ \intO{ p [ (-D\VV) : \nabla \overline{\vv} ] } \\
		&\quad\qquad - \intO{\overline{q} (D\VV : \nabla \uu)} \Big\}\\
		&\quad  =:  \Im\left\{ K_{1} + K_{2} + K_{3} + K_{4} \right\}.
	\end{aligned}
\end{equation}
	%
	%
	%
To write this expression as a boundary integral, we invoke the regularity assumptions on $\ff$, $\bgg$, and $\Omega$.
Let us note that $\uu = \vv = \vect{0}$ on $\Gamma$ which implies that $\nabla{\uu} = \dn{\uu}\nn$ and $\nabla{\overline{\vv}} = \dn{\overline{\vv}}\nn$ on $\Gamma$.
Thus, through Lemma \ref{lem:expansion}, the integral $K_{1}$ can be expanded as follows
\begin{align*}
            &- \intO{ {\alpha} A \nabla {\uu} : \nabla {\overline{\vv}}} \\
            %
            %
            %
            & \quad  = \sum_{j=1}^{d} \Bigg\{ \intO{ (- {\alpha} \Delta {u}_{j}) \VV \cdot \nabla \overline{{v}}_{j} }
            	- \intO{ ({\alpha} \Delta \overline{{v}}_{j}) \VV \cdot \nabla {u}_{j} }
            	+ \intG{\mu\dn{{u}_{j}}(\VV \cdot \nabla \overline{{v}}_{j})} \\
            & \quad   \quad \qquad
            	 + \intG{\mu\dn{\overline{{v}}_{j}}(\VV \cdot \nabla {u}_{j})}
            	- \intG{\mu(\nabla \overline{{v}}_{j} \cdot \nabla {u}_{j})\Vn} \Bigg\}\\
            & \quad  = \sum_{j=1}^{d} \Bigg\{ \intO{ (- {\alpha} \Delta {u}_{j}) \VV \cdot \nabla \overline{{v}}_{j} }
            	- \intO{ ({\alpha} \Delta \overline{{v}}_{j}) \VV \cdot \nabla {u}_{j} }
            	+ \intG{\mu\dn{\overline{{v}}_{j}}(\VV \cdot \nabla {u}_{j})} \Bigg\}\\	
            %
            %
            %
            %
            %
            %
            & \quad  =  \intO{ p \dive{(\nabla \overline{\vv}^{\top} \VV)}} 
            			- \intG{ p\nn \cdot \nabla \overline{\vv}^{\top} \VV }
           			+ \intO{\uu_{i} \cdot \nabla \uu^{\top} \VV}
            			+ \intO{ \overline{q} \dive{(\nabla {\uu}^{\top} \VV)}} \\
            &\qquad \qquad		 
            	+ \intG{\overline{\sigma(\vv,q)}\nn \cdot \dn{\uu} \Vn } 
           \\  & \quad  =: K_{11} + K_{12} + K_{13} + K_{14} + K_{15}.	
\end{align*}
In light of Lemma \ref{lem:divergence_expansion} (considering $\nabla \cdot{\uu} = 0$ in $\Omega$) and subsequent application of Lemma \ref{lem:crucial_identity}, we determine that
$K_{14} + K_{4} = 0$ and $K_{11}+K_{12}+K_{2}+K_{3} = \intO{ (\VV \cdot \nabla{p}) \pim }$.
Thus, from equation \eqref{eq:distributed_gradient}, we have the limit
\[
	\dot{J}_{4}(0) + \dot{J}_{8}(0)
		= \Im\left\{ \intO{ \left[ \uu_{i} \cdot \nabla \uu^{\top} \VV + (\VV \cdot \nabla{p}) \pim \right] } \right\}
            			+  \Im\left\{ \intG{\overline{\sigma(\vv,q)}\nn \cdot \dn{\uu} \Vn } \right\}.
\]
Combining this expression with $\dot{J}_{2}(0) + \dot{J}_{6}(0)$ from \eqref{eq:distributed_gradient_first_part} yields the desired expression \eqref{eq:shape_gradient}, concluding the proof of the theorem.
\end{proof}
\begin{remark}
In the proof of Theorem \ref{thm:shape_derivative}, we only require $\Omega$ to be $\mathcal{C}^{1,1}$ regular, in addition to the condition that $(\ff, \bgg) \in H^{3/2}(\Sigma)^{d} \times H^{1/2}(\Sigma)^{d}$.
These assumptions enable us to attain higher regularity for the state and adjoint variables, facilitating the application of Lemma \ref{lem:expansion} to express the shape derivative of $J$ in a boundary integral form.
It actually suffices to assume that $\Sigma$ is only Lipschitz smooth to obtain the shape derivative of $J$. 
However, for the sake of simplicity in our arguments, we have adopted the stated regularity assumptions.
\end{remark}
The following result can be drawn easily from \eqref{eq:shape_gradient}, \eqref{eq:adjoint_system}, and the equivalence between the shape optimization problem and the original inverse problem.
\begin{corollary}[Necessary optimality condition]\label{cor:necessary_condition}
	Let the domain $\Omega^\ast$ be such that $(\uu, p)=(\uu(\Omega^\ast),p(\Omega^\ast))$ satisfies \eqref{eq:overdetermined_original}, i.e., we have $-p\nn + {\alpha} \dn{\uu}= \bgg	$ and $\uu = \ff$ on $\Sigma^{\ast}$, or equivalently,
	\begin{equation}
	\label{eq:imaginary_is_zero}
		\ui = \vect{0} \quad \text{and} \quad \pim = 0 \quad \text{on $\Omega^\ast$},
	\end{equation}
	with $(\uu, p)$ satisfying \eqref{eq:ccbm}.
	Then, the domain $\Omega^\ast$ is a stationary solution for the shape optimization problem $J(\Omega) \to \inf$ which is subject to the state problem \eqref{eq:ccbm}; i.e., it fulfills the necessary optimality condition
	\begin{equation}
	\label{eq:optimality_condition}
		{d}J(\Omega^\ast)[\VV] = 0, \quad \text{for all $\VV \in \sfTheta$}.
	\end{equation}
\end{corollary}
\begin{proof}
	Using the assumption that $\ui = \vect{0}$ and $\pim = 0$ on $\Omega^\ast$, one obtains -- in view of \eqref{eq:adjoint_system} -- that $\vect{v}_{i} = \vect{0}$ and $q_{i}=0$ on $\Omega^\ast$.
	Thus, it follows that $\ggb = 0$ on $\Gamma^\ast$ which therefore gives us the conclusion ${d}J(\Omega^\ast)[\VV] = 0$, for any $\VV \in \sfTheta$.
\end{proof}
{With Theorem \ref{thm:shape_derivative} and Corollary \ref{cor:necessary_condition} at hand, we can devise a numerical procedure to solve Problem \ref{prob:shape_optimization_problem}. We will present this procedure in the next section, along with various numerical examples to test the proposed method.}
\section{Numerical experiments}
\label{sec:numerical_experiments}
To numerically solve the optimization problem $J(\Omega) \to \inf$, we will apply a preconditioned shape-gradient-based descent method based on FEM.
%
%
%
\begin{remark}\label{rem:using_shape_Hessian}
Newton-type methods, as shown in \cite{EpplerHarbrecht2005,AfraitesDambrineEpplerKateb2007,RabagoAzegami2020}, could be developed for the numerical procedure using shape Hessian information. 
However, we refrain from applying the method for two main reasons.
Firstly, computing second-order shape derivatives is challenging, involving substantial effort, preparations, and technical intricacies; see, e.g., \cite{AfraitesDambrineKateb2008}. 
In the computation of the shape Hessian of $J$, two additional systems of complex PDEs, which correspond to the new adjoints, must be introduced.%
The second reason is primarily numerical. 
In general, the implementation and computational effort required for computing the Hessian are not justified by a significant reduction in the number of iterations. 
Indeed, it is important to note that evaluating the shape Hessian requires the solution of (at least) four variational problems (two for the states and another two for the adjoints), making it prohibitively costly.
Despite these drawbacks, we defer the examination of the shape Hessian of $J$ for future work. 
\end{remark}
%
%
%
\subsection{Numerical algorithm}
\label{subsec:Numerical_Algorithm}
Our main algorithm follows a standard approximation procedure (see, e.g., \cite{CaubetDambrineKatebTimimoun2013,RabagoNotsu2024}), the important details of which we provide as follows.

\textit{Choice of descent direction.}
We will employ a Riesz representation of the shape gradient $G$ to avoid undesired oscillations on the unknown boundary, which could result in instabilities in the approximation. 
To do this, we fix $\eta \in (0,1]$ and generate a smooth descent direction for the cost function $J$ by seeking a vector $\VV \in H_{\Sigma,0}^1(\Omega)^{d}$ that solves the variational equation
\begin{equation}\label{eq:extension_regularization}
	 \eta \intO{ \nabla \VV : \nabla \vect{\varphi}} + (1-\eta) \intG{ \nabla_{\Gamma} \VV : \nabla_{\Gamma} \vect{\varphi} }
		= - \intG{{\ggb} \nn \cdot \vect{\varphi}}, \ \text{for all $\vect{\varphi} \in H_{\Sigma,0}^1(\Omega)^{d}$}.
\end{equation}
Here, we choose $\eta = 0.5$. In this way we obtain a \textit{Sobolev gradient} \cite{Neuberger1997} representation $\VV$ of $G$.
More importantly, this approach produces a smoothed preconditioned extension of $-{G}\nn$ over the entire domain $\Omega$.
Such extensions allow us to deform the discretized computational domain by moving the (movable) nodes of the computational mesh and not only the boundary nodes that are subject for perturbation.
The boundary integral appearing on the right side of \eqref{eq:extension_regularization} corresponds to the Laplace-Beltrami operator on $\Gamma$.
Its purpose is to ensure an additional level of smoothness for the descent direction $\VV$ within $\Gamma$.
Further discussion about discrete gradient flows for shape optimization problems are issued in \cite{Doganetal2007}.
\begin{remark}
We reiterate here that our approach differs completely from that applied in \cite{CaubetDambrineKatebTimimoun2013}. 
In the cited work, the authors employed parametric regularization to address the ill-posedness of the problem, utilizing a parametric model of shape variations for numerical realization. 
This process involved restricting the domains to star-shaped configurations and using polar coordinates for parametrization, resulting in an explicit computation of $\VV$.
\end{remark}
To compute the $k$th domain $\Omega^{k}$, we carry out the following procedures:
\begin{description}
\setlength{\itemsep}{2pt}
	\item[1. \it{Initilization}] Choose an initial guess $\Omega^{0}$.
	\item[2. \it{Iteration}] For $k = 0, 1, 2, \ldots$, do the following:
		\begin{enumerate}
			\item[2.1] Solve the state's and adjoint's variational equations on the current domain $\Omega^{k}$.
			\item[2.2] Choose $t^{k}>0$, and compute the descent vector $\VV^{k}$ using \eqref{eq:extension_regularization} in $\Omega^{k}$.
			\item[2.3] Update the current domain by setting $\Omega^{k+1} := \{ x + t^{k} \VV^{k}(x) \in \mathbb{R}^{d} \mid  x \in \Omega^{k}\}$.
		\end{enumerate}
	\item[3. \it{Stop Test}] Repeat \textit{Iteration} until convergence.
\end{description}

\textit{Step-size computation and stopping condition.} In Step 2.2, $t^{k}$ is computed using a backtracking line search procedure inspired by \cite[p. 281]{RabagoAzegami2020} with the formula $t^{k} = \mu J(\Omega^{k})/|\VV^{k}|^2_{\mathbf{H}^{1}(\Omega^{k})}$ at each iteration, where $\mu > 0$ is a scaling factor. 
This step size is further reduced to avoid reversed triangles within the mesh after the update.
Meanwhile, we terminate the algorithm as soon as the condition that $\abs{J^{k-1}-J^{k}} < \varepsilon_{J}$ or $t^{k} < \varepsilon_{T}$ is satisfied where $\varepsilon_{J}, \varepsilon_{T} > 0$ are sufficiently small real numbers, or the algorithm reached the maximum allowable number of iterations.
In all experiments, we set $\mu = 0.5$, $\varepsilon_{J} = \varepsilon_{T} = 10^{-12}$, and the maximum number of iterations to $300$.

\subsection{Forward problem}
We let $\alpha = 1.0$ and consider the unit circle as the shape of the medium.
For the input data, we will consider $\bgg = (\sin\theta, -\cos\theta)^{\top}$, where $\theta \in [0, 2\pi)$ represents the prescribed flux. 
To obtain additional boundary measurements on the accessible part $\Sigma$, we will use synthetic data. 
Specifically, given a prescribed Neumann flux $\bgg$, we generate the Dirichlet data $\ff$ on the accessible boundary $\Sigma$ by numerically solving the well-posed PDE problem \eqref{eq:state_un} using the finite element method. 
To avoid `inverse crimes' (as discussed by Colton and Kress \cite[p. 179]{ColtonKress2019}) in producing the measurements, we generate the synthetic data with a different numerical scheme. 
This involves using a larger number of discretization points (with finer meshes of uniform mesh width $h^{\ast} = 1/256$) and applying ${P}{3}/{P}{2}$ finite element basis functions in the \textsc{FreeFem++} \cite{Hecht2012} code than in the inversion process. 
For experiments with noise-contaminated data with a specified noise level $\delta$, we perturb the data with Gaussian noise.
\subsection{Inversion procedure}
In the inversion procedure, all variational problems are solved using Hood-Taylor (${P}{2}/{P}{1}$) finite elements, and we discretize the domain with coarse meshes.
Specifically, in all experiments, the initial domain is discretized with uniform mesh size having $70$ nodes on the interior boundary and $100$ nodes on the exterior boundary.
\begin{remark}\label{rem:lower_degree_finite_elements}
We emphasize here that we are, in fact, using lower-degree finite elements in our numerical simulations, in contrast to the computational setup in \cite{CaubetDambrineKatebTimimoun2013}, where the authors used ${P}{4}/{P}{3}$ finite elements for the forward problem while employing ${P}{3}/{P}{2}$ finite elements in the inversion process.
\end{remark}
{We will now test our proposed method by solving concrete 2D and 3D examples of the inverse obstacle problem \eqref{problem:gip}.}
\subsection{Numerical tests in 2D}
In all plotted figures illustrating the computed shapes, the surface of the medium is represented by the thicker black solid lines, while the initial guesses are depicted with black dotted lines. 
The exact obstacles are illustrated by red solid lines, whereas the intermediate and final approximate shapes are represented by black dashed lines (or grey dashed lines), usually accompanied by markers. 
We initiate the study with numerical experiments using exact data before progressing to cases involving noisy data.
	
Before we proceed with our numerical examples, we recall that, by assumption, $\Omega$ is of class $\mathcal{C}^{1,1}$.
However, in our experiments, we will also consider cases of unknown obstacles with Lipschitz boundaries (e.g., a square or an \textsf{L}-shape obstacle).
Although these test examples violate the regularity assumption, we will see below that the algorithm was still able to produce fair reconstructions of the considered obstacle types.
%
\subsubsection{Large convex obstacles}
The method performs exceptionally well with large convex obstacles exhibiting smooth shapes (see the left plot in Figure \ref{fig:large_convex}) and functions quite effectively even with large convex obstacles that have corners (see the right plot in Figure \ref{fig:large_convex}) -- as expected.
\begin{figure}[htp!]
\centering
\resizebox{0.23\textwidth}{!}{\includegraphics{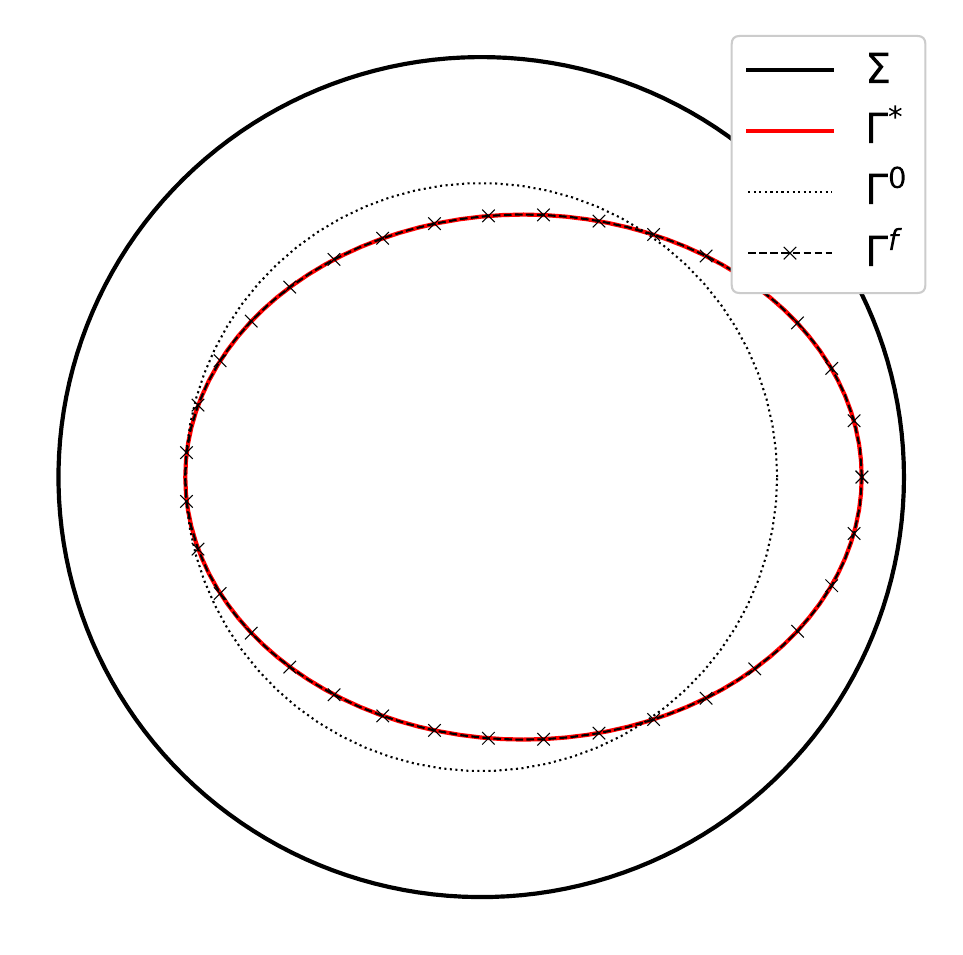}}
\resizebox{0.23\textwidth}{!}{\includegraphics{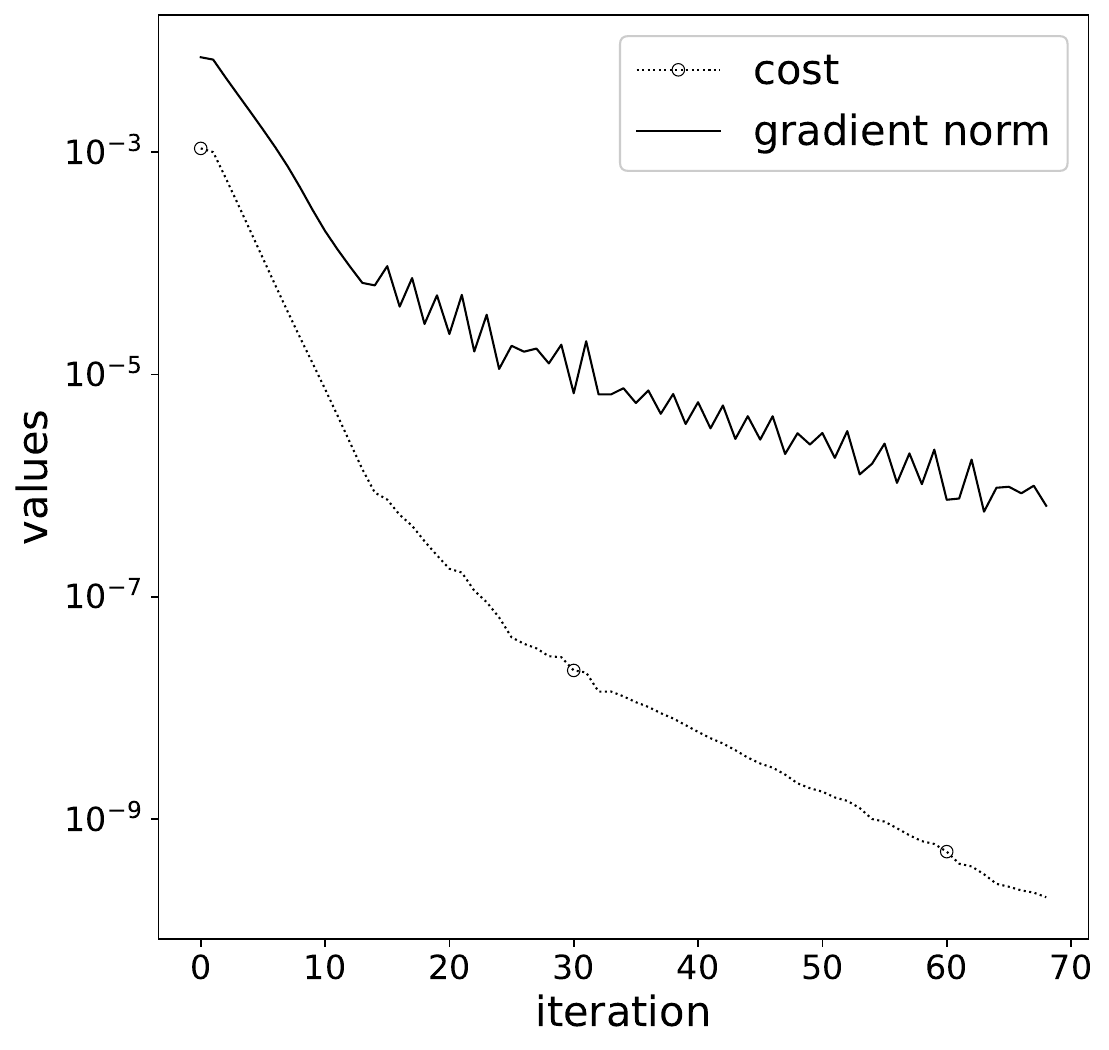}} \hfill 
\resizebox{0.23\textwidth}{!}{\includegraphics{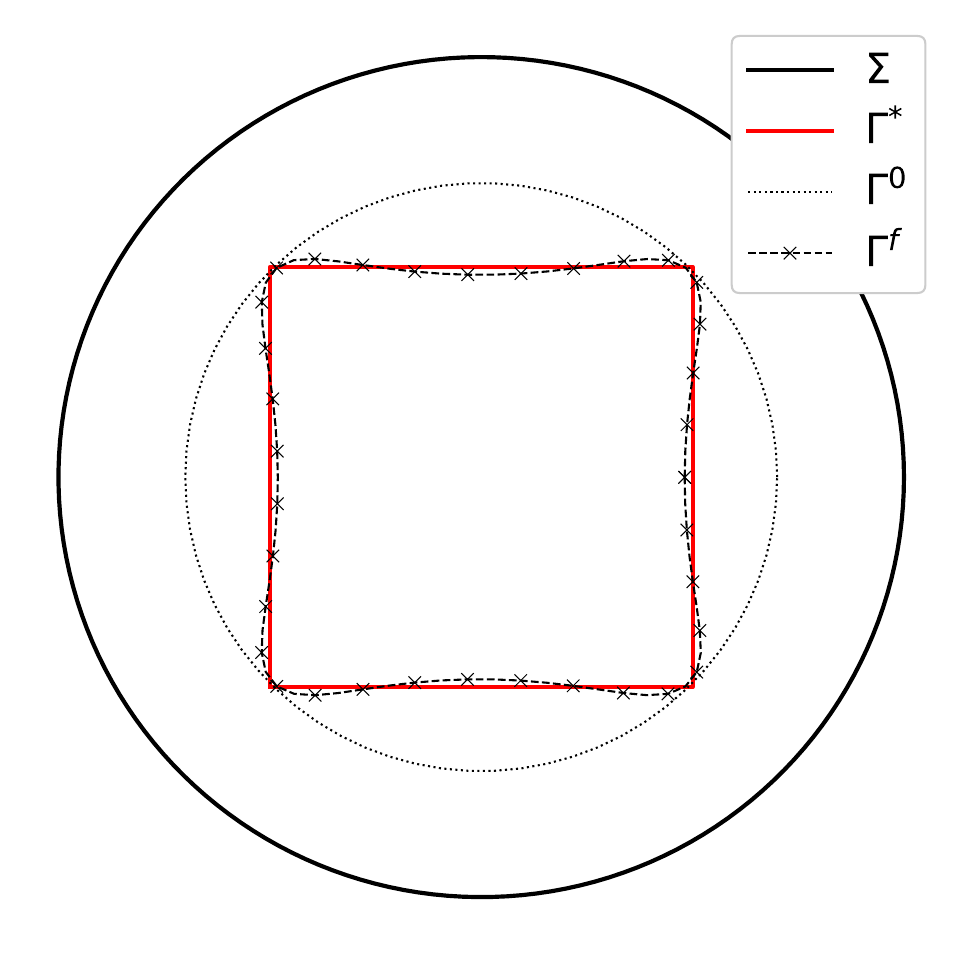}}
\resizebox{0.23\textwidth}{!}{\includegraphics{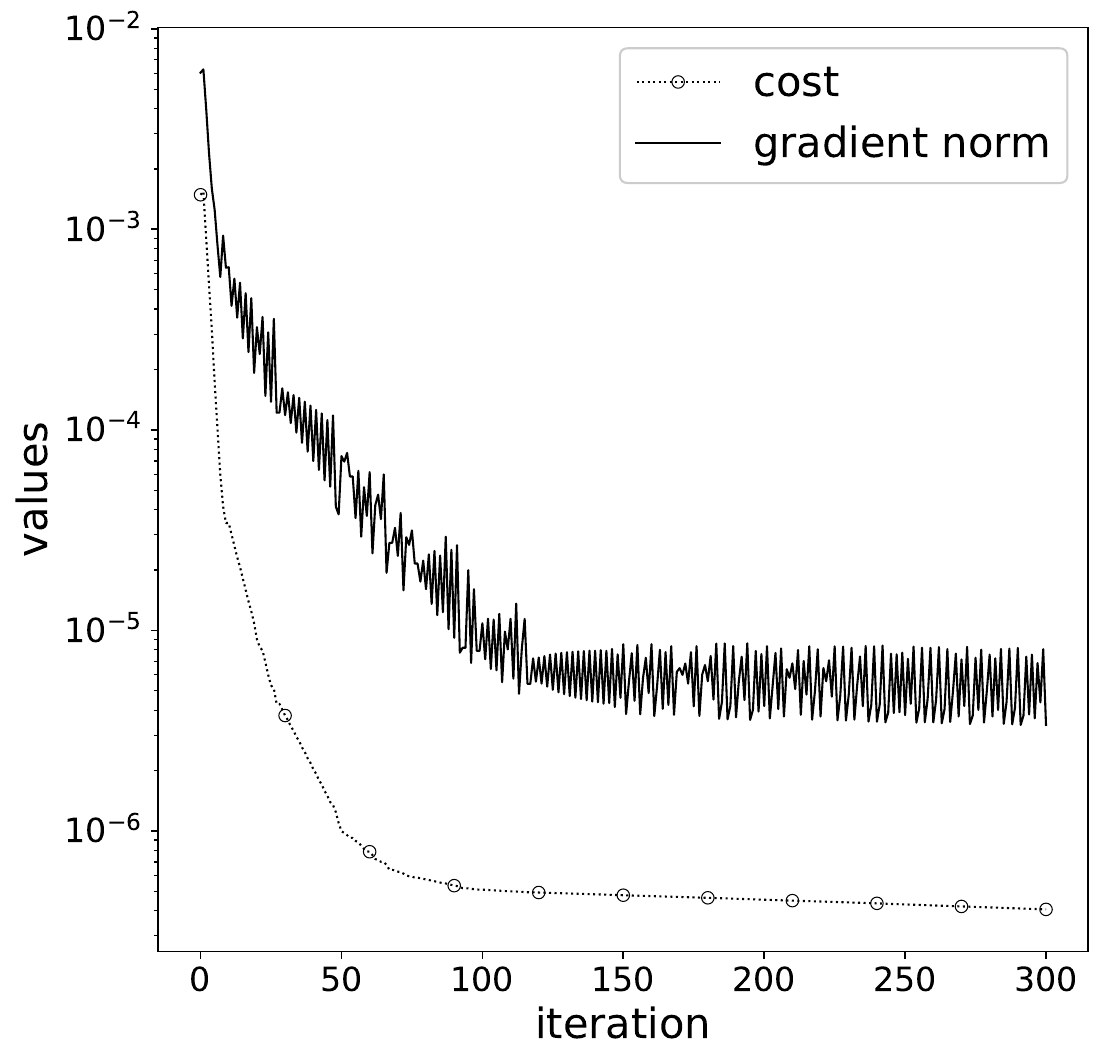}}
\caption{Shape reconstructions with large convex obstacles}
\label{fig:large_convex}
\end{figure}
\subsubsection{Large non-convex obstacle}
Meanwhile, when dealing with more complicated shapes featuring pronounced concavities, the method performs only fairly. 
For instance, for obstacles parametrized as follows
\begin{align}
	\Gamma^{\ast}_{1} &= \left\{\begin{pmatrix} 0.195+0.4(\cos{t}+0.65\cos{2t}), 0.55\sin{t}\end{pmatrix}^{\top},\ \forall t \in [0, 2\pi) \right\},\label{ex:case1}\tag{C.I}\\
	\Gamma_{2}^{\ast} &= \left\{\begin{pmatrix} 0.64 \cos{t}, 0.48 \sin{t} (1.8 + \cos{(2t)}) \end{pmatrix}^{\top},\ \forall t \in [0, 2\pi) \right\},\label{ex:case2}\tag{C.II}\\
	\Gamma^{\ast}_{3} 
		&= \left\{\begin{pmatrix} -0.25 + \displaystyle\frac{0.6+0.54\cos{t}+0.06\sin{2t}}{1+0.75\cos{t}}\cos{t} \\[0.75em] 0.05 + \displaystyle\frac{0.6+0.54\cos{t}+0.06\sin{2t}}{1+0.75\cos{t}}\sin{t}\end{pmatrix},\ \forall t \in [0, 2\pi) \right\},\label{ex:case3}\tag{C.III}\\
	\Gamma^{\ast}_{4} &= \left\{0.4(1+0.75\cos(5t+\pi))\begin{pmatrix}\cos{t}, \sin{t}\end{pmatrix}^{\top},\ \forall t \in [0, 2\pi) \right\},\label{ex:case4}\tag{C.IV}	
\end{align}
the reconstructed shapes are only reasonably accurate, as depicted in Figure \ref{fig:large_complex}.
It seems that obstacles with deep concavities are more challenging to reconstruct with good accuracy -- as expected (see the rightmost plot in Figure \ref{fig:large_complex}). 
Nevertheless, the proposed method accurately predicts the positions of concavities within the obstacles, as evident in Figure \ref{fig:large_complex}.
\begin{figure}[htp!]
\centering
 \hfill 
\resizebox{0.24\textwidth}{!}{\includegraphics{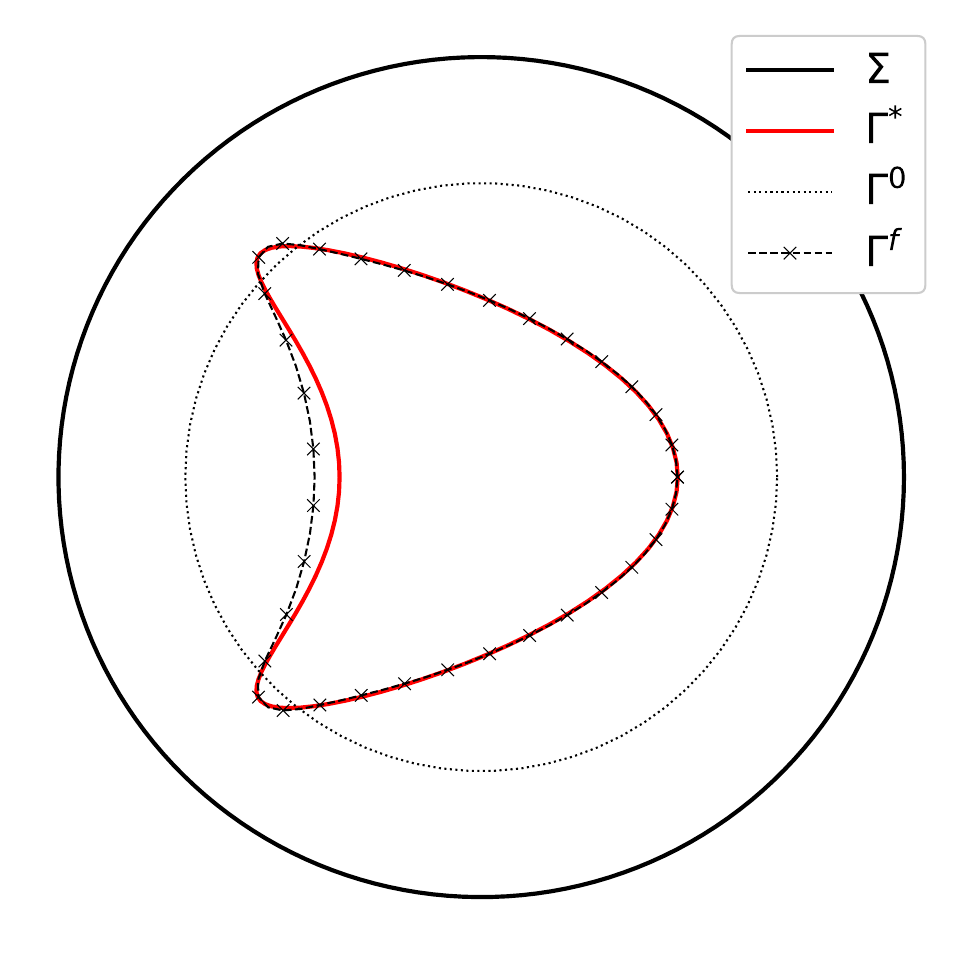}} \hfill 
\resizebox{0.24\textwidth}{!}{\includegraphics{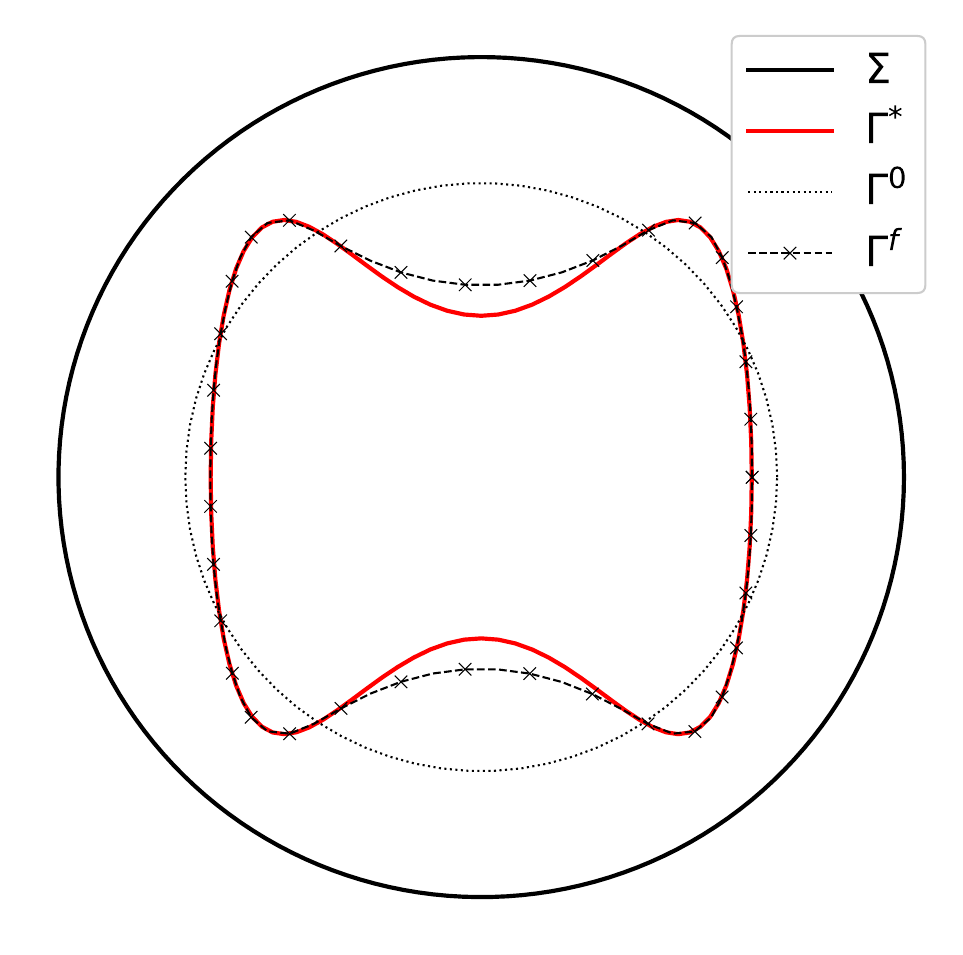}} \hfill 
\resizebox{0.24\textwidth}{!}{\includegraphics{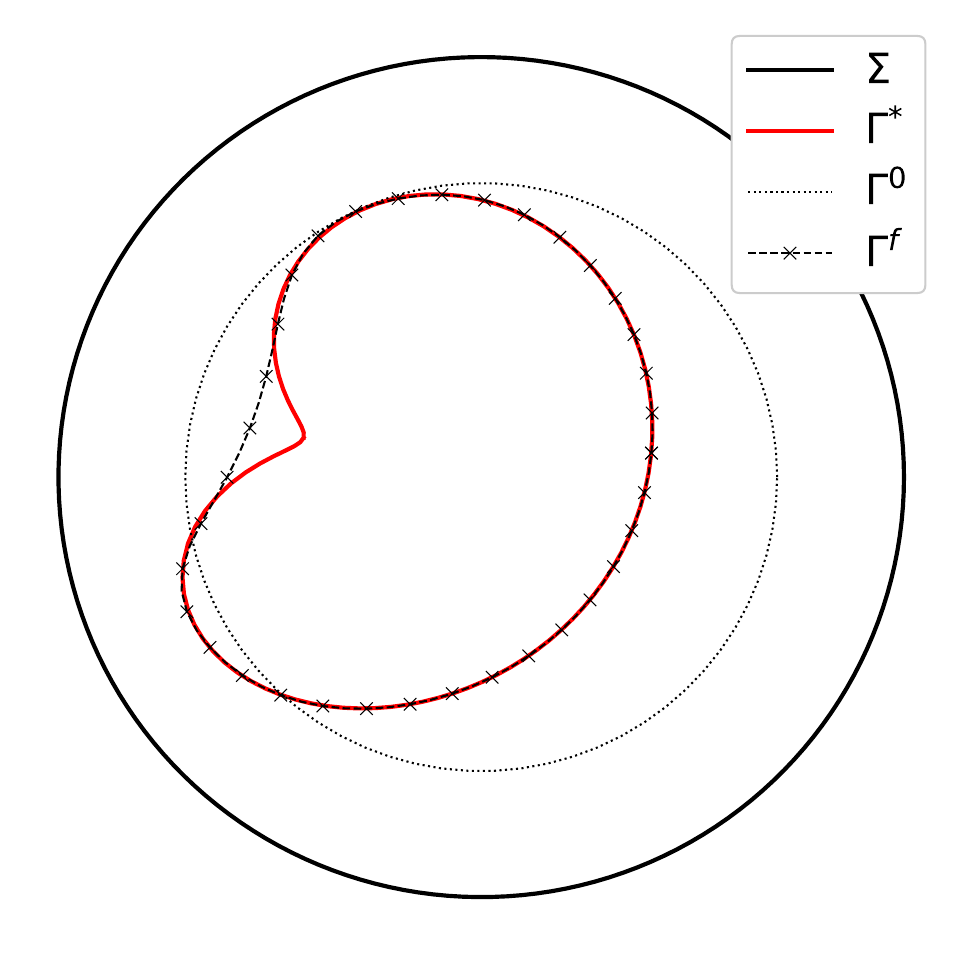}} \hfill 
\resizebox{0.24\textwidth}{!}{\includegraphics{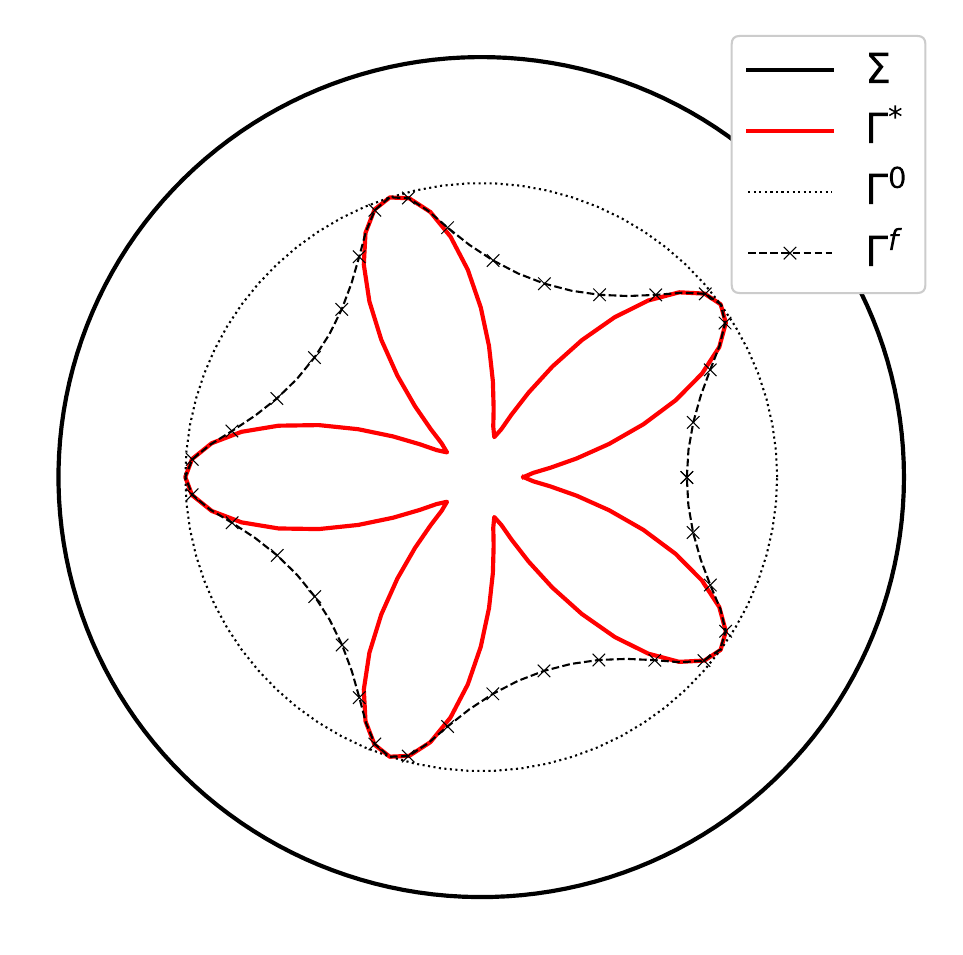}}  \hfill 
\caption{Shape reconstructions with complicated-shaped obstacles}
\label{fig:large_complex}
\end{figure}
\subsubsection{Effect of the choice of initial guess}
We also tested various initial guesses to assess their impact on the reconstruction. 
Based on our findings, the effect on the reconstruction is negligible (see Figure \ref{fig:initial_guesses}), unless the initial guess is excessively large or small. 
In such cases, it becomes imperative to implement remeshing, involving the use of coarser meshes, after a certain number of iterations. 
This necessity arises because, with continued iterations, the mesh quality deteriorates, leading to the flattening of some triangles, especially when employing large step sizes. 
Addressing this issue is feasible by opting for smaller step sizes in each iteration.
\begin{figure}[htp!]
\centering 
\resizebox{0.32\textwidth}{!}{\includegraphics{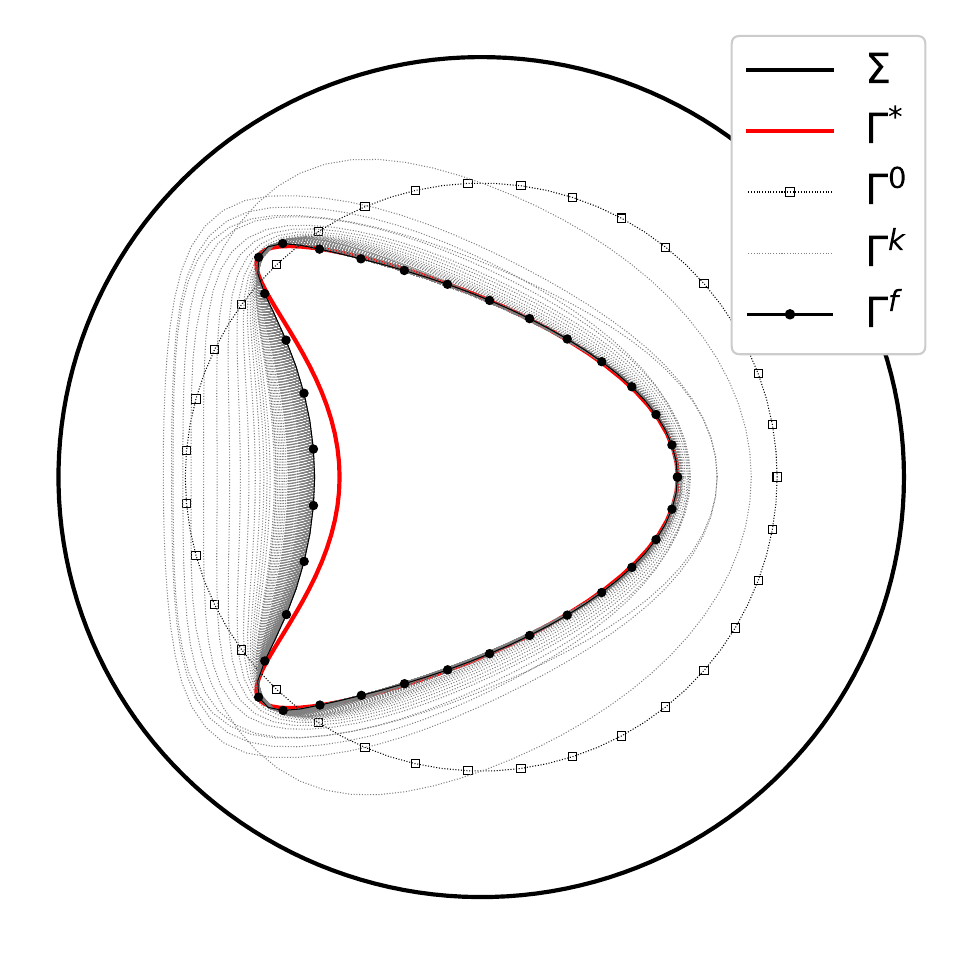}} \hfill
\resizebox{0.32\textwidth}{!}{\includegraphics{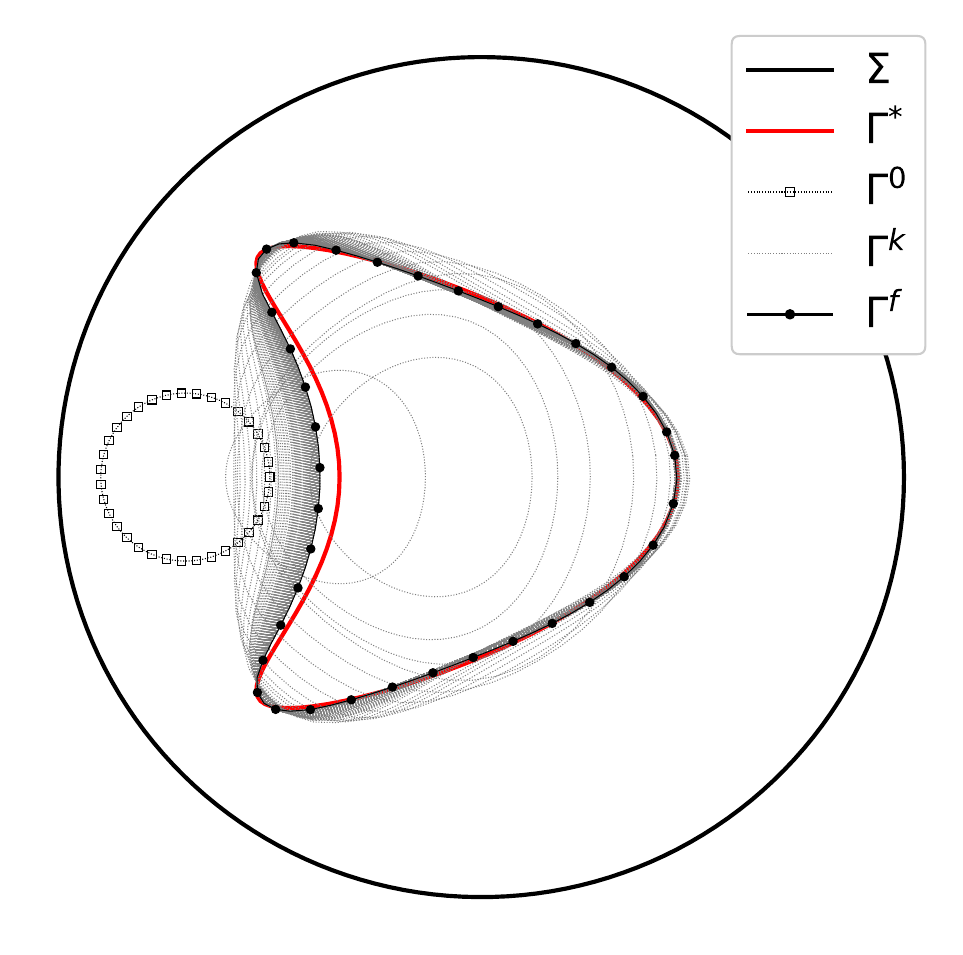}} \hfill
\resizebox{0.32\textwidth}{!}{\includegraphics{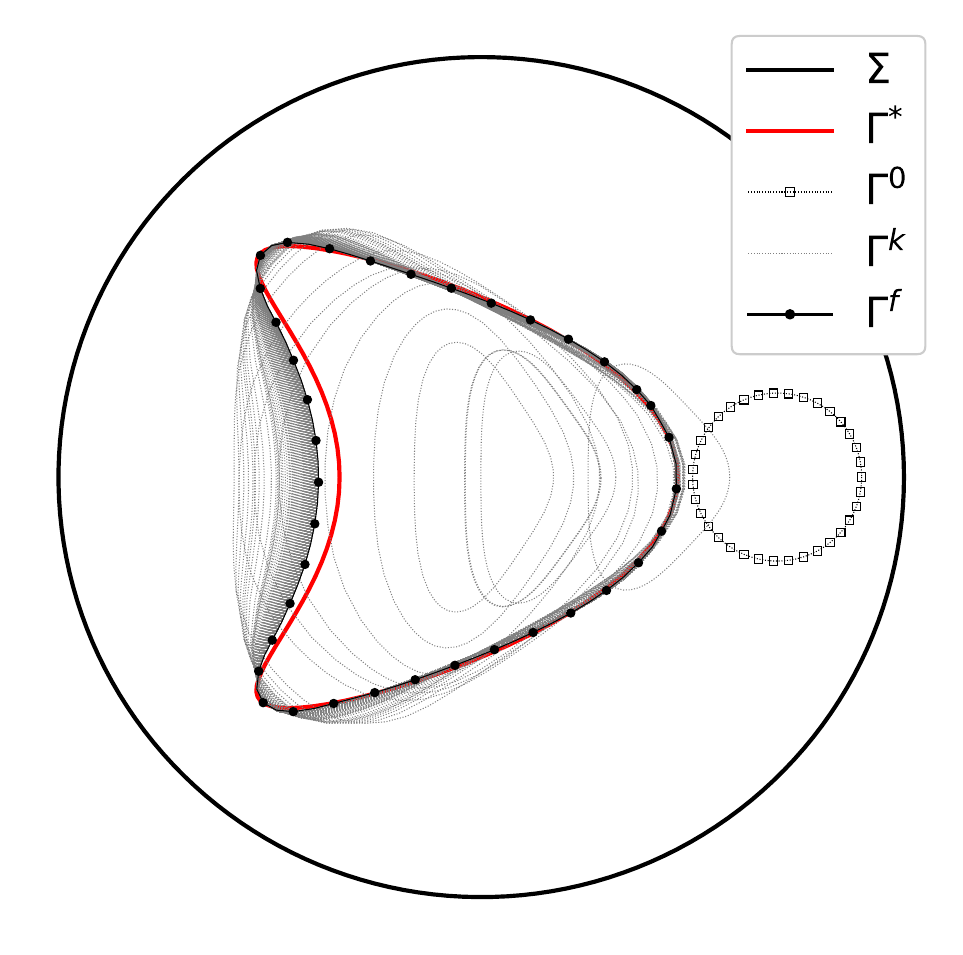}} 
\caption{Effect of the choice of initial guess}
\label{fig:initial_guesses}
\end{figure}

\subsubsection{Noisy data}
We present reconstructions using noisy data and emphasize that, in the presence of noise, reconstructions typically involve regularization. 
One common approach is perimeter penalization (see, e.g., \cite{CaubetDambrineKatebTimimoun2013,RabagoAzegami2018}). 
The problem is inherently ill-posed, necessitating the application of regularization methods for numerical solutions. 
Introducing a perimeter functional term, for instance, results in well-posed problems (see, e.g., \cite{Dambrine2002,BucurButtazzo2005}).
However, in our numerical investigation with noisy data, we refrain from adding any additional regularization term to our method. 
This decision is based on our observations from conducted tests, where the proposed CCBM formulation exhibited a regularizing effect similar to the Kohn-Vogelius method; see, e.g., \cite{AfraitesRabago2022}.
The results of our numerical experiments, depicted in Figures \ref{fig:noisy_data} (also referring to Figures \ref{fig:influence_of_size} and \ref{fig:L-shape}), demonstrate reasonable reconstructions even in the presence of high noise levels. 
Additionally, the method generally exhibited the capability to accurately detect concavities within obstacles. 
It is worth noting, however, that the presence of noise compels the approximation to converge towards larger cost values, indicating a greater deviation of the reconstructed shape from the exact shape, as expected.
\begin{figure}[htp!]
\centering
\resizebox{0.32\textwidth}{!}{\includegraphics{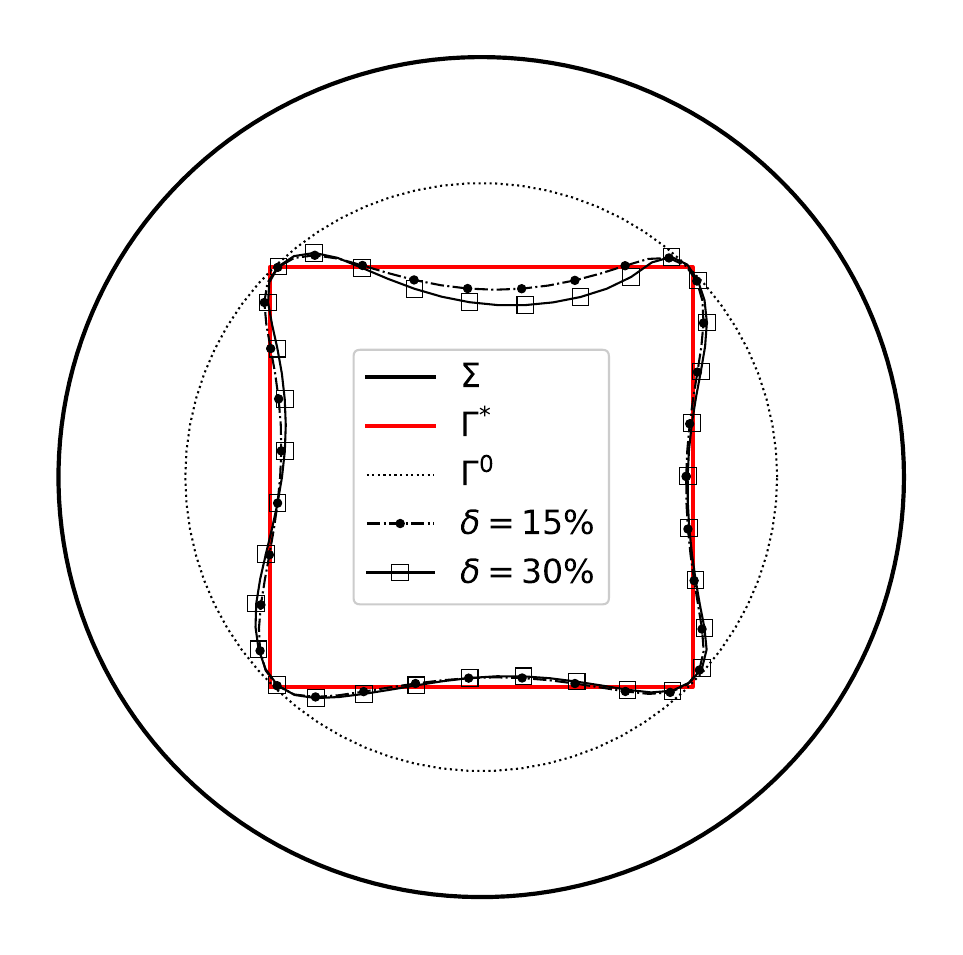}}\hfill
\resizebox{0.32\textwidth}{!}{\includegraphics{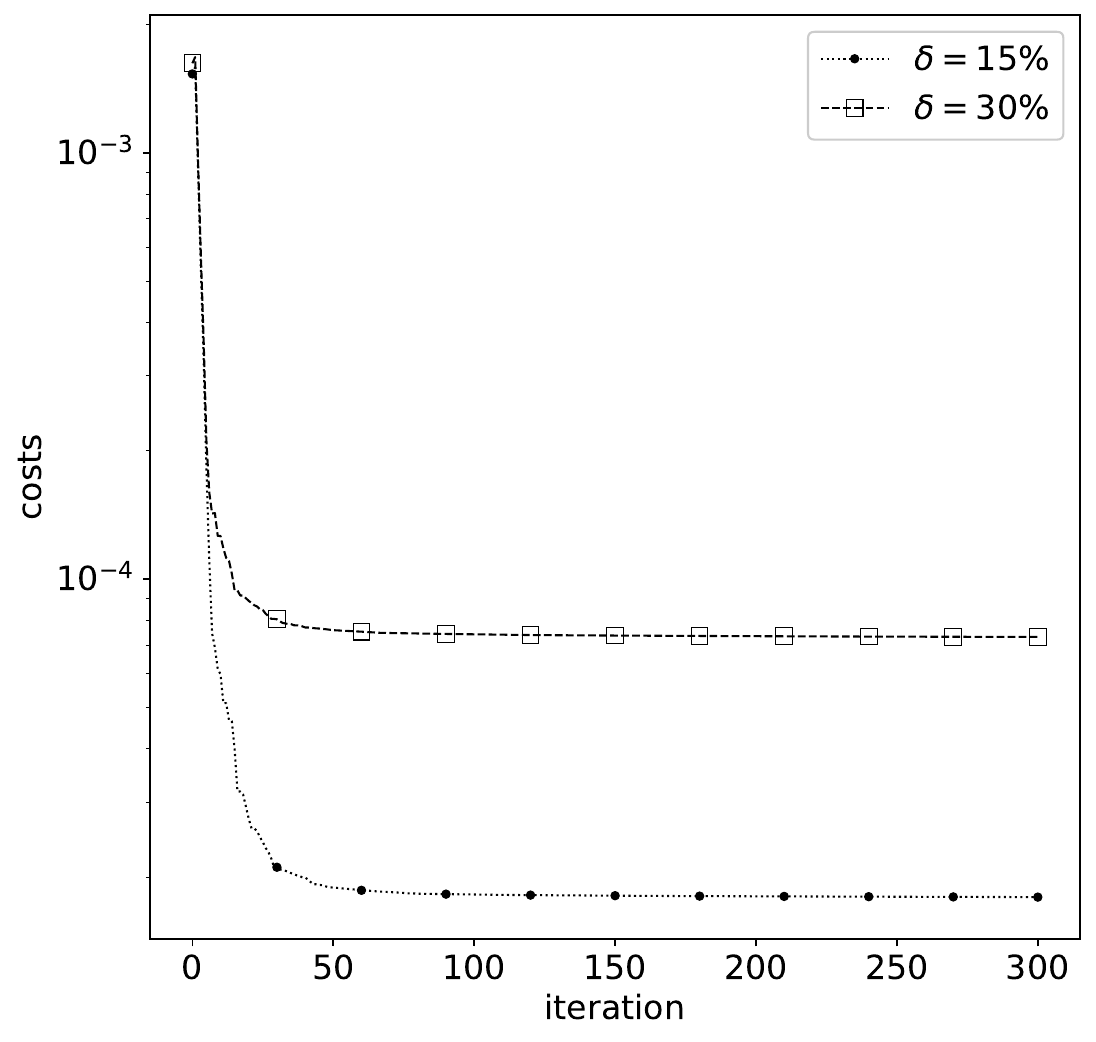}}\hfill
\resizebox{0.32\textwidth}{!}{\includegraphics{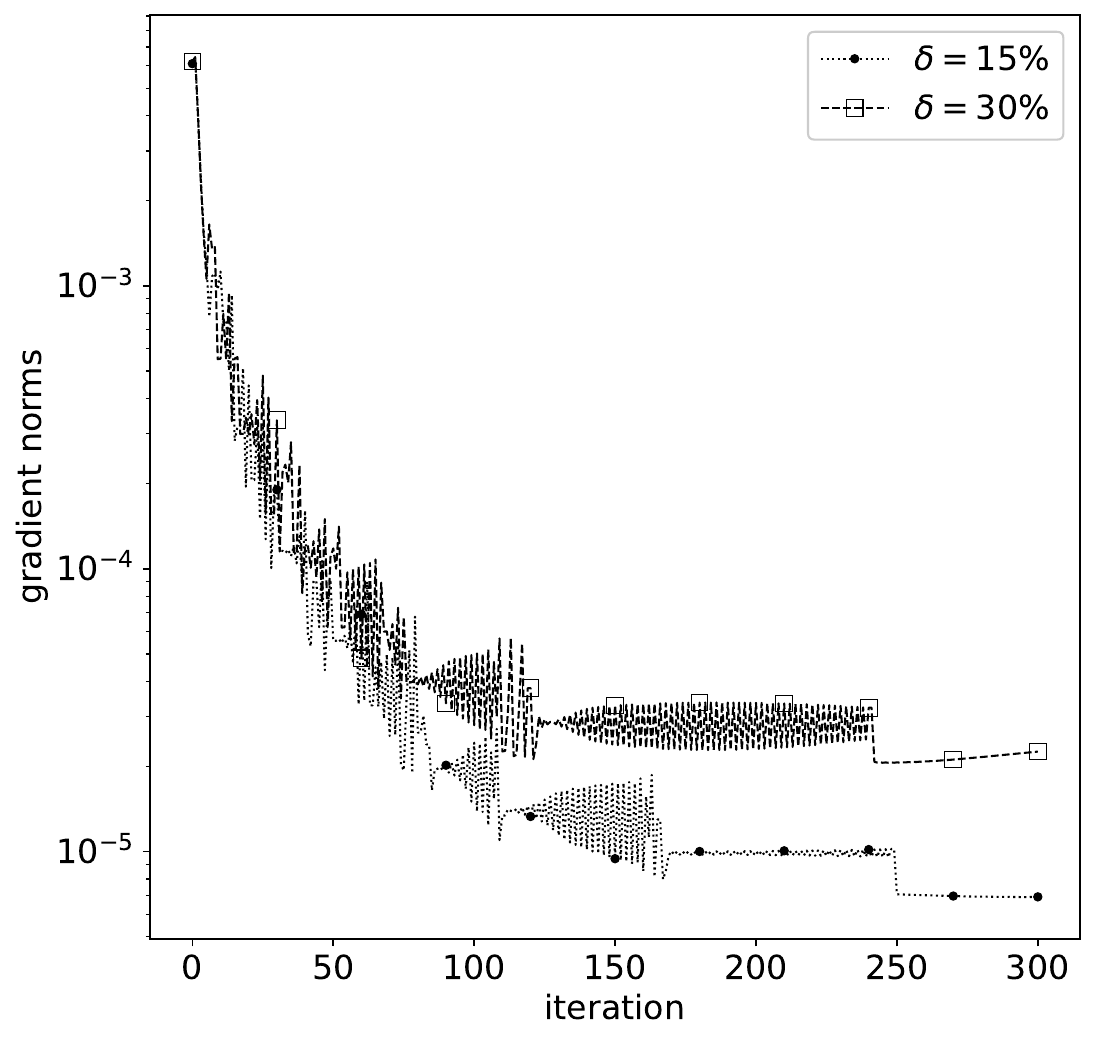}}\hfill
\\[0.7em]
\resizebox{0.32\textwidth}{!}{\includegraphics{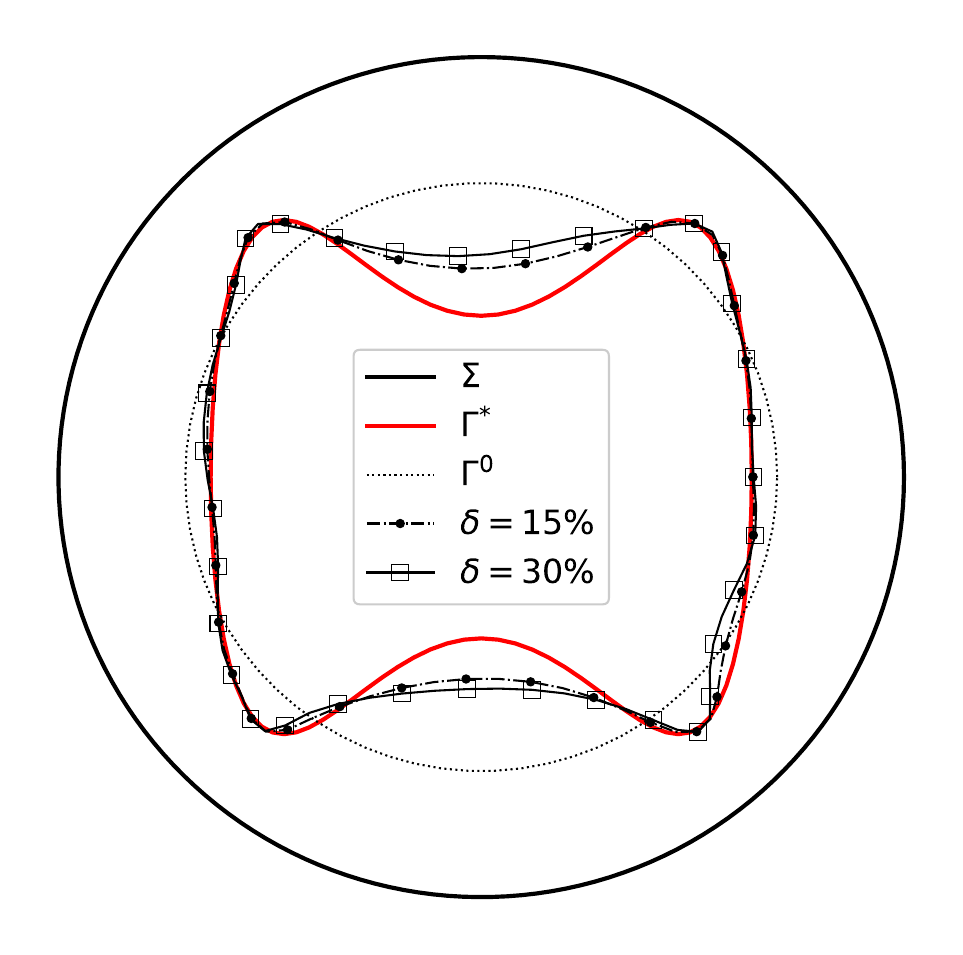}}\hfill
\resizebox{0.32\textwidth}{!}{\includegraphics{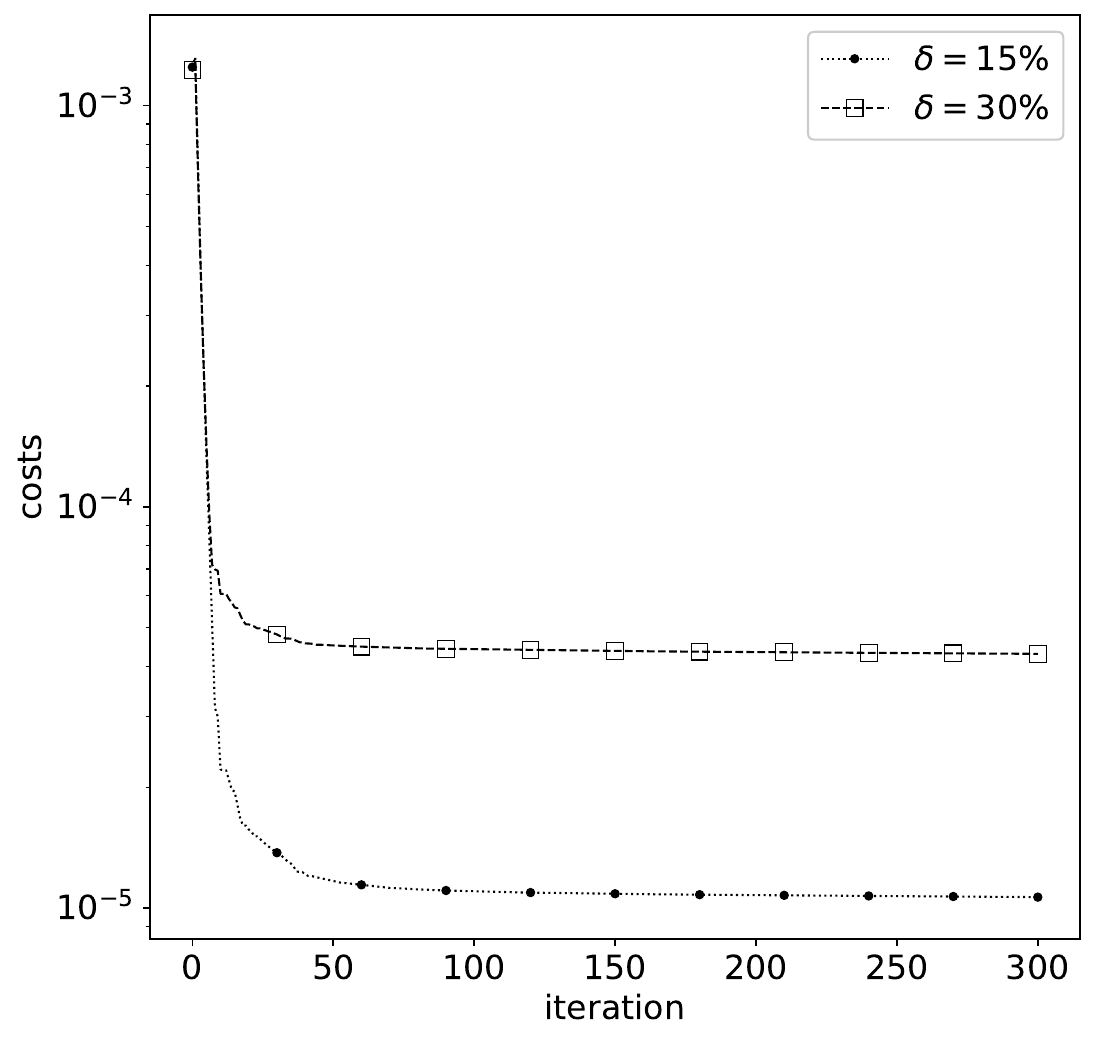}}\hfill
\resizebox{0.32\textwidth}{!}{\includegraphics{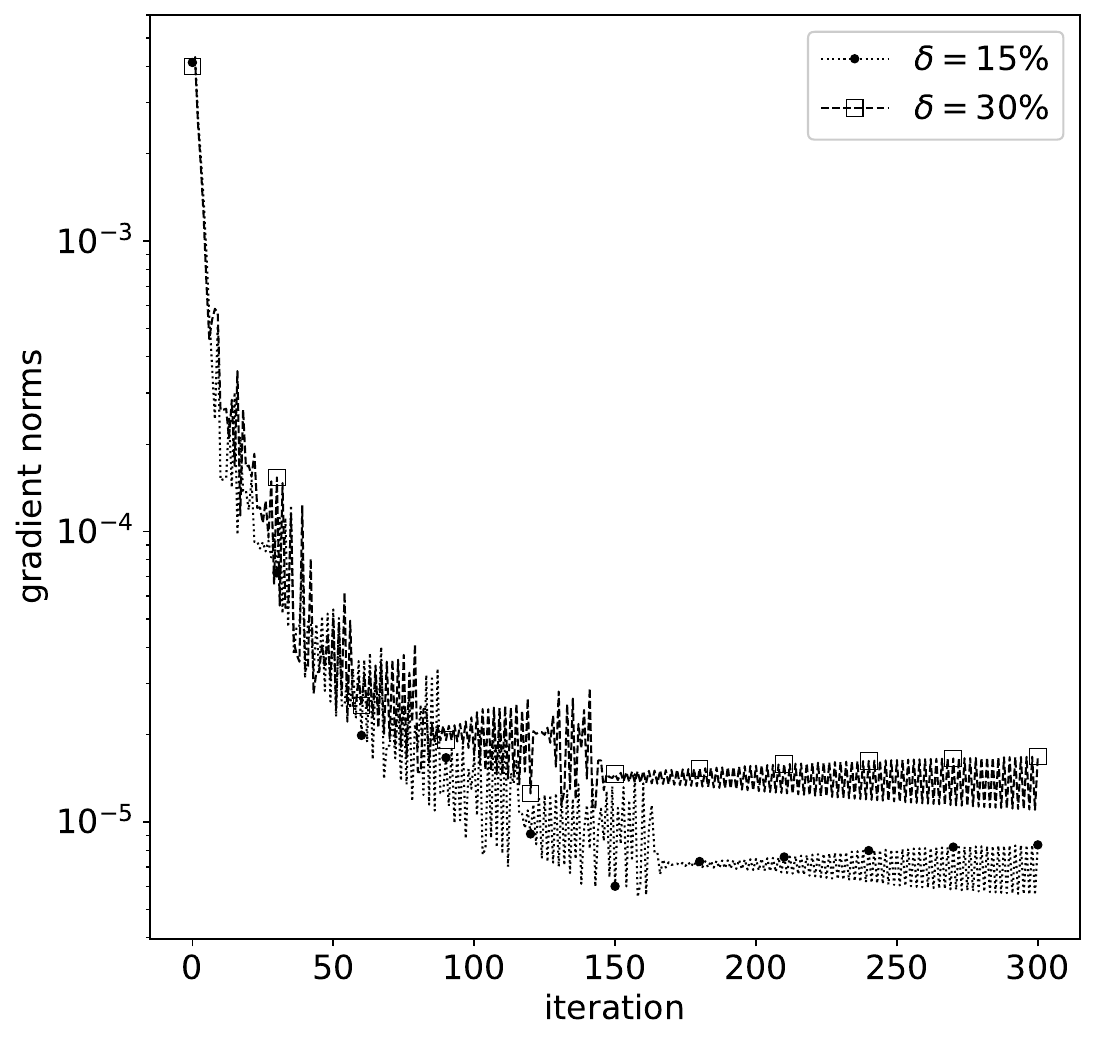}}\hfill
\caption{Shape reconstructions with noisy data}
\label{fig:noisy_data}
\end{figure}

\subsubsection{Influence of the size of the obstacle}
We next look at the influence of the obstacle's size on the reconstruction. 
We observed that when attempting to reconstruct a smaller object, the results are worse, as evident in Figure \ref{fig:influence_of_size} -- as expected. 
This phenomenon, related to the obstacle's size, could be further explained using the explicit calculus of the \textit{shape Hessian}, as done, for example, in \cite{AfraitesDambrineEpplerKateb2007,CaubetDambrineKatebTimimoun2013}. 
We defer this aspect for future examination, as the first-order method already yields satisfactory reconstructions of shapes, even in conditions characterized by a high level of noise ($\delta = 15\%$). See also Remark \ref{rem:using_shape_Hessian}.
\begin{figure}[htp!]
\centering
\resizebox{0.32\textwidth}{!}{\includegraphics{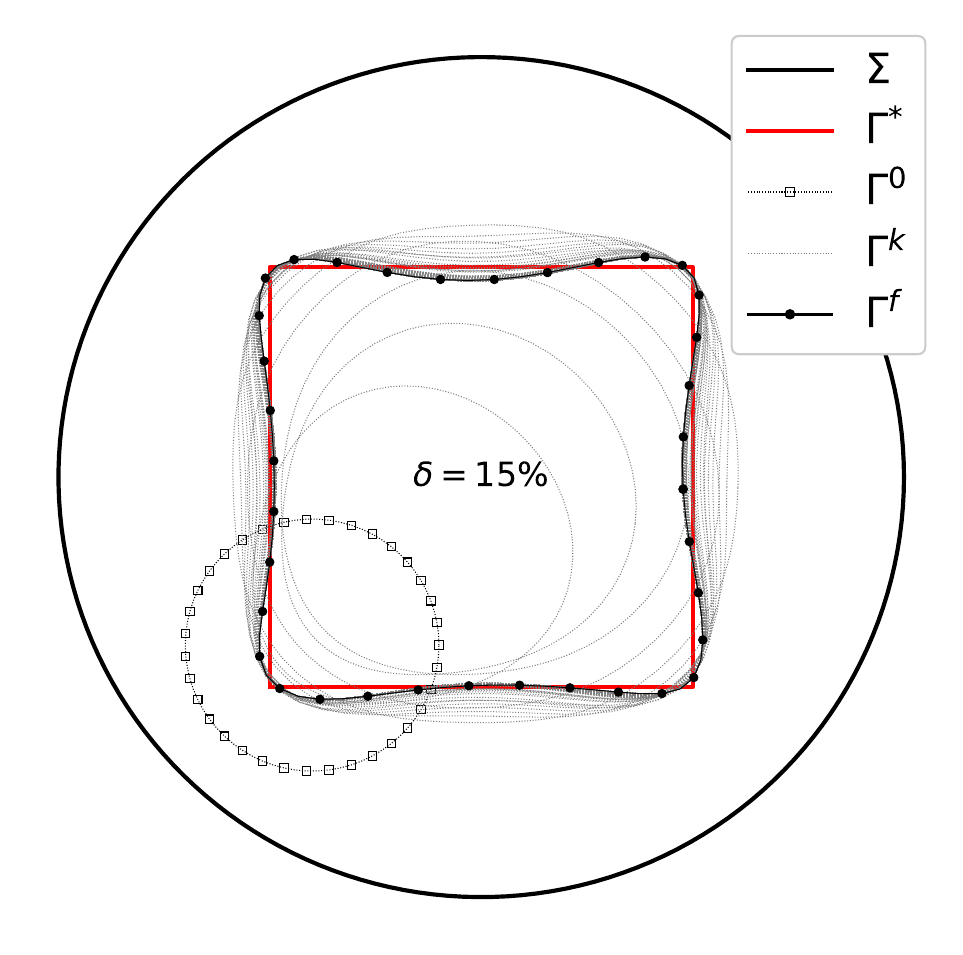}}\hfill
\resizebox{0.32\textwidth}{!}{\includegraphics{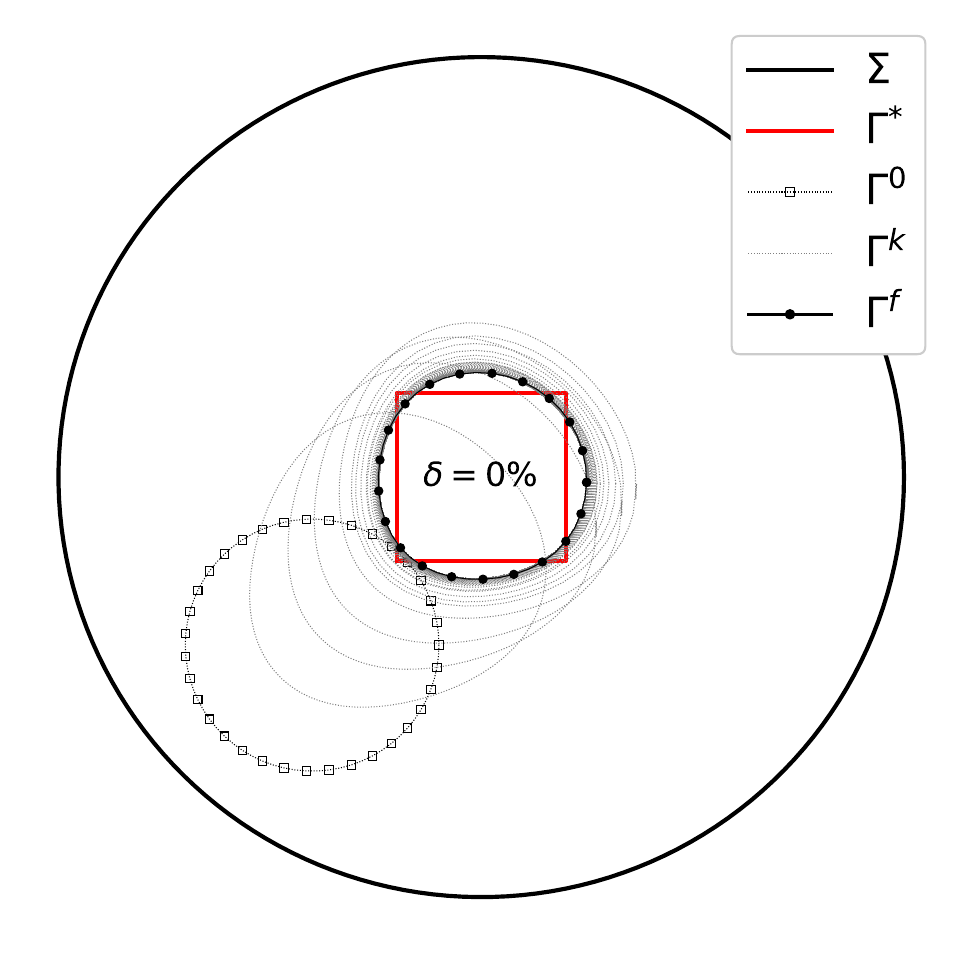}}\hfill
\resizebox{0.32\textwidth}{!}{\includegraphics{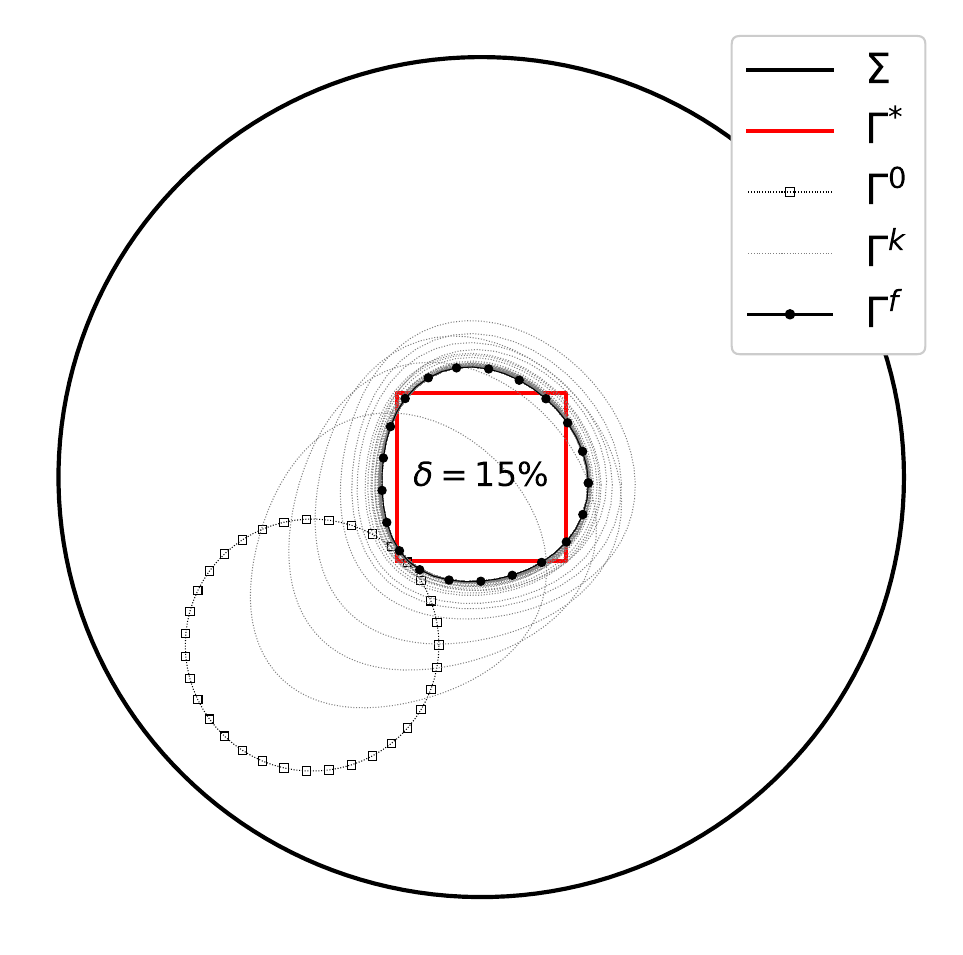}}
\caption{Influence of the size of the obstacle}
\label{fig:influence_of_size}
\end{figure}
\subsubsection{An obstacle with reentrant corner}
We also tested our method for reconstructing an obstacle with a reentrant corner, such as an \textsf{L}-block (domain see Figure \ref{fig:L-shape}). 
As anticipated, the reconstruction is much worse compared to obstacles with smooth shapes and concavities. 
However, it is noteworthy that the reconstruction accuracy can be improved by choosing a good initial guess, as shown in Figure \ref{fig:L-shape}.
{We emphasize that, although the test examples violate the regularity assumption, it appears that the algorithm was still able to generate fair reconstruction of the type of inclusion considered.}
\begin{figure}[htp!]
\centering
\resizebox{0.32\textwidth}{!}{\includegraphics{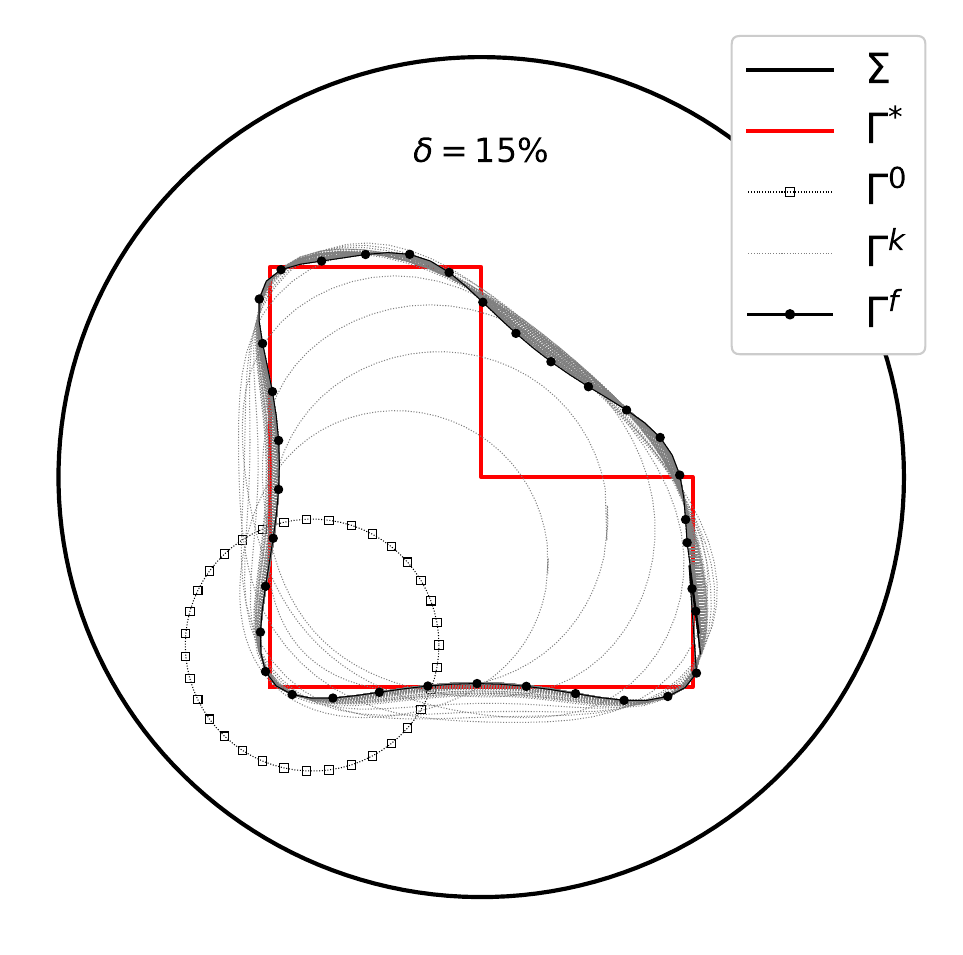}}\hfill
\resizebox{0.32\textwidth}{!}{\includegraphics{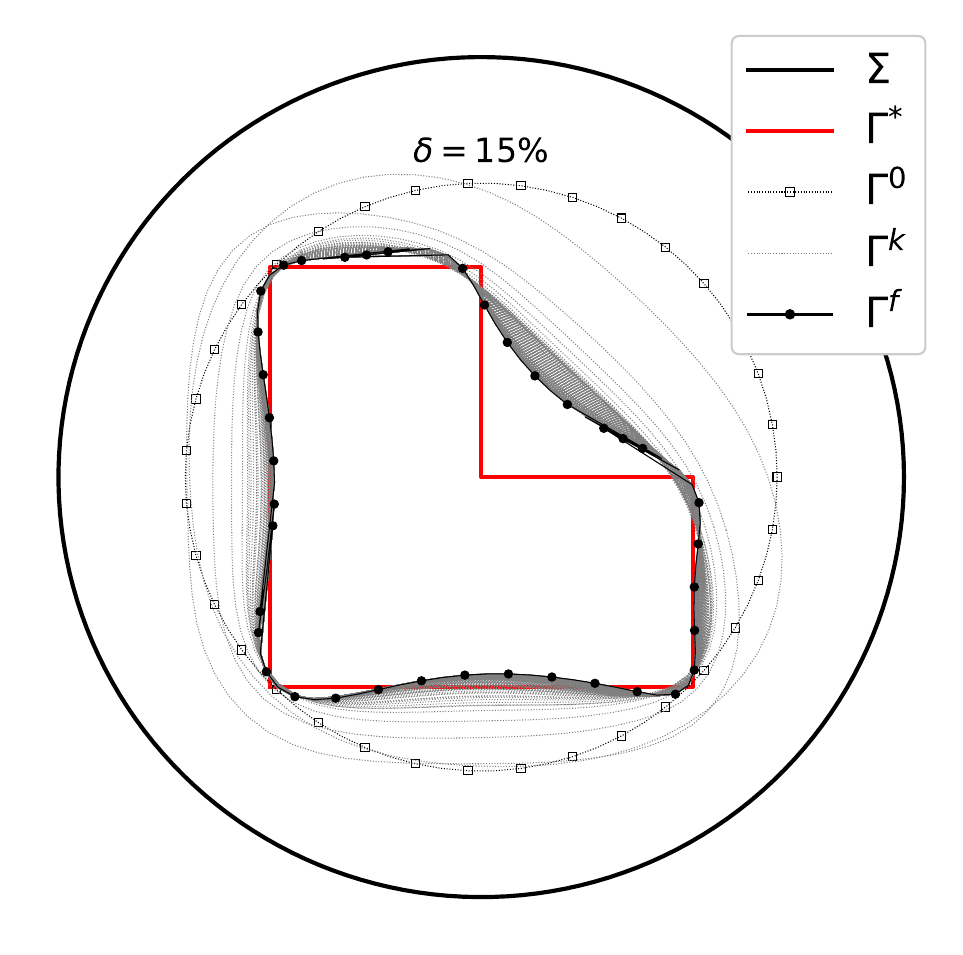}}\hfill
\resizebox{0.32\textwidth}{!}{\includegraphics{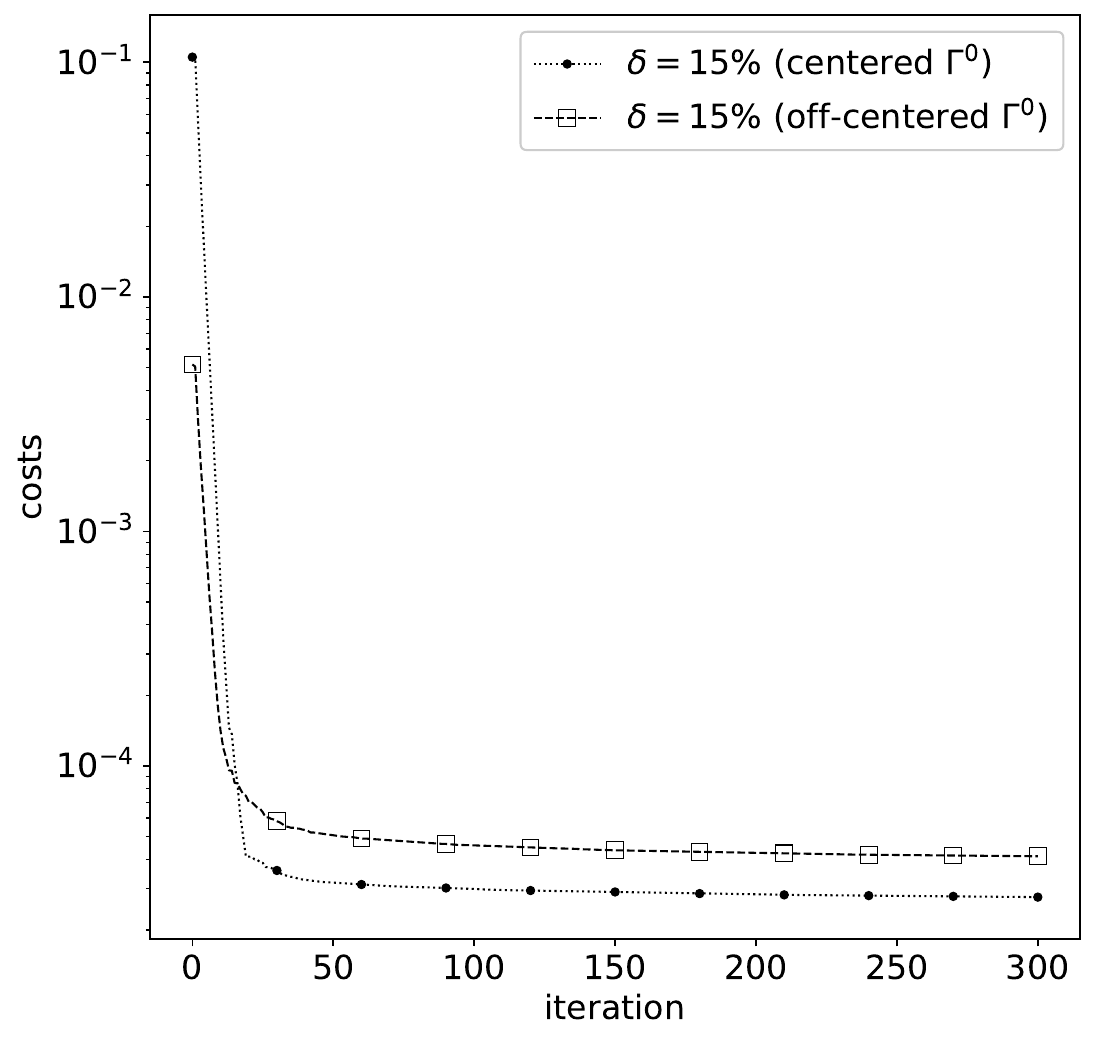}}\hfill
\caption{\textsf{L}-shape obstacle with noisy data ($\delta = 15\%$)}
\label{fig:L-shape}
\end{figure}

\subsubsection{Detecting more than one object}
Let us note here that the main assumption about the domain is that $\Omega = D \setminus\overline{\omega}$ is connected. 
This assumption does not exclude the case of two or more inclusions in $D$. 
Thus, we aim to numerically detect two objects: a square and a circle with relatively small sizes. 
Here, we apply remeshing every ten iterations to further avoid instabilities during mesh deformations. 
As evident in Figure \ref{fig:two_obstacles}, the reconstructions with data contaminated with a $30\%$ noise level are reasonable, and we point out that the computational times for the iteration processes in detecting multiple obstacles are comparable to the case of just one inclusion.
For instance, in the case of a small square centered at the origin shown in Figure \ref{fig:influence_of_size}, it only requires 291 seconds to complete 300 iterations, while in the present test case, the average computational time is 300 seconds.
\begin{remark}
In \cite{CaubetDambrineKatebTimimoun2013}, the authors noted longer computational times for detecting multiple obstacles compared to a single object -- a contrast to our method. Through remeshing at selected iterates, we achieve fast and fairly accurate reconstructions of obstacles, even in high-noise conditions, thanks to our use of coarser meshes and lower-degree finite element methods.
\end{remark}
\begin{figure}[htp!]
\centering\hfill
\resizebox{0.32\textwidth}{!}{\includegraphics{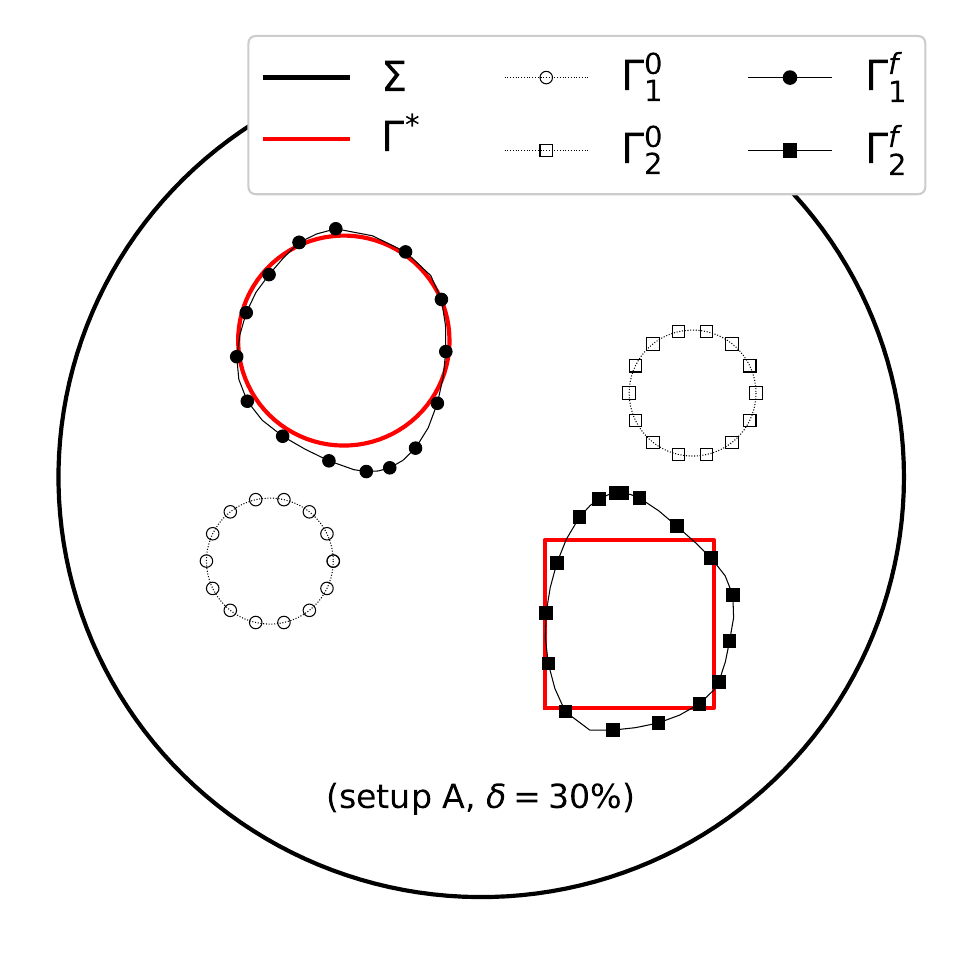}}\hfill
\resizebox{0.32\textwidth}{!}{\includegraphics{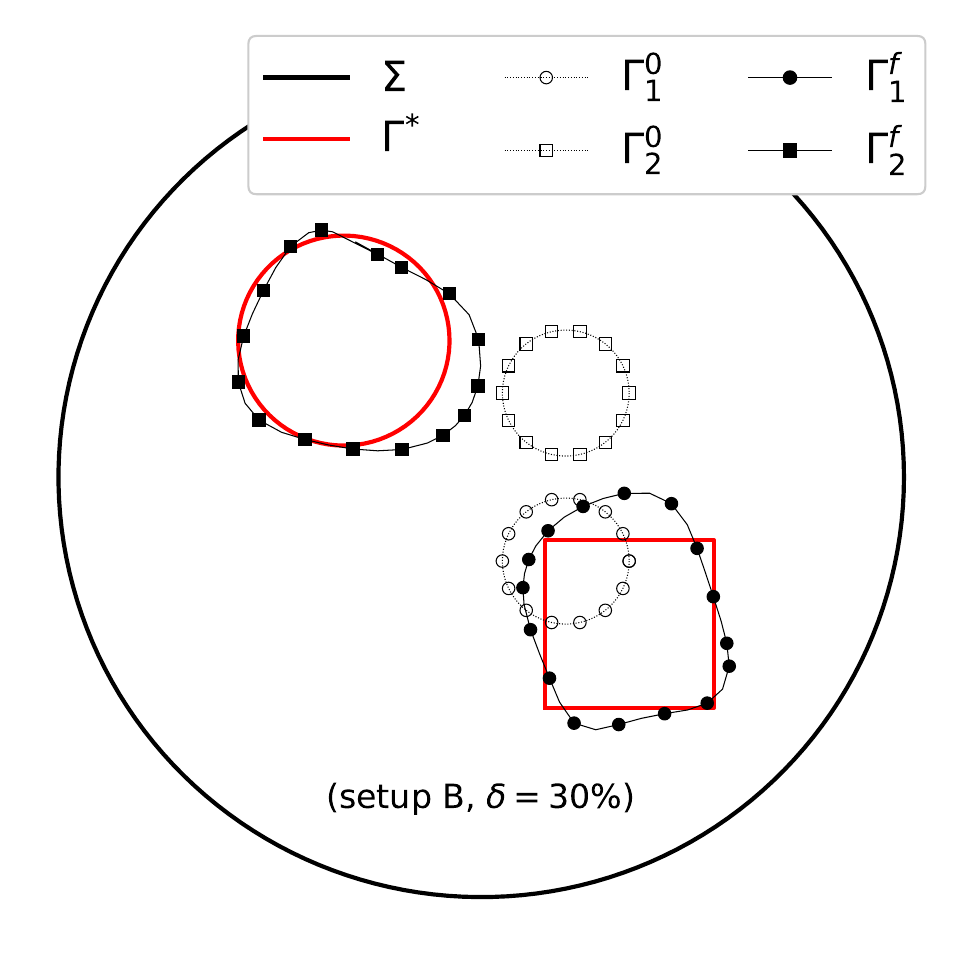}}\hfill
\resizebox{0.32\textwidth}{!}{\includegraphics{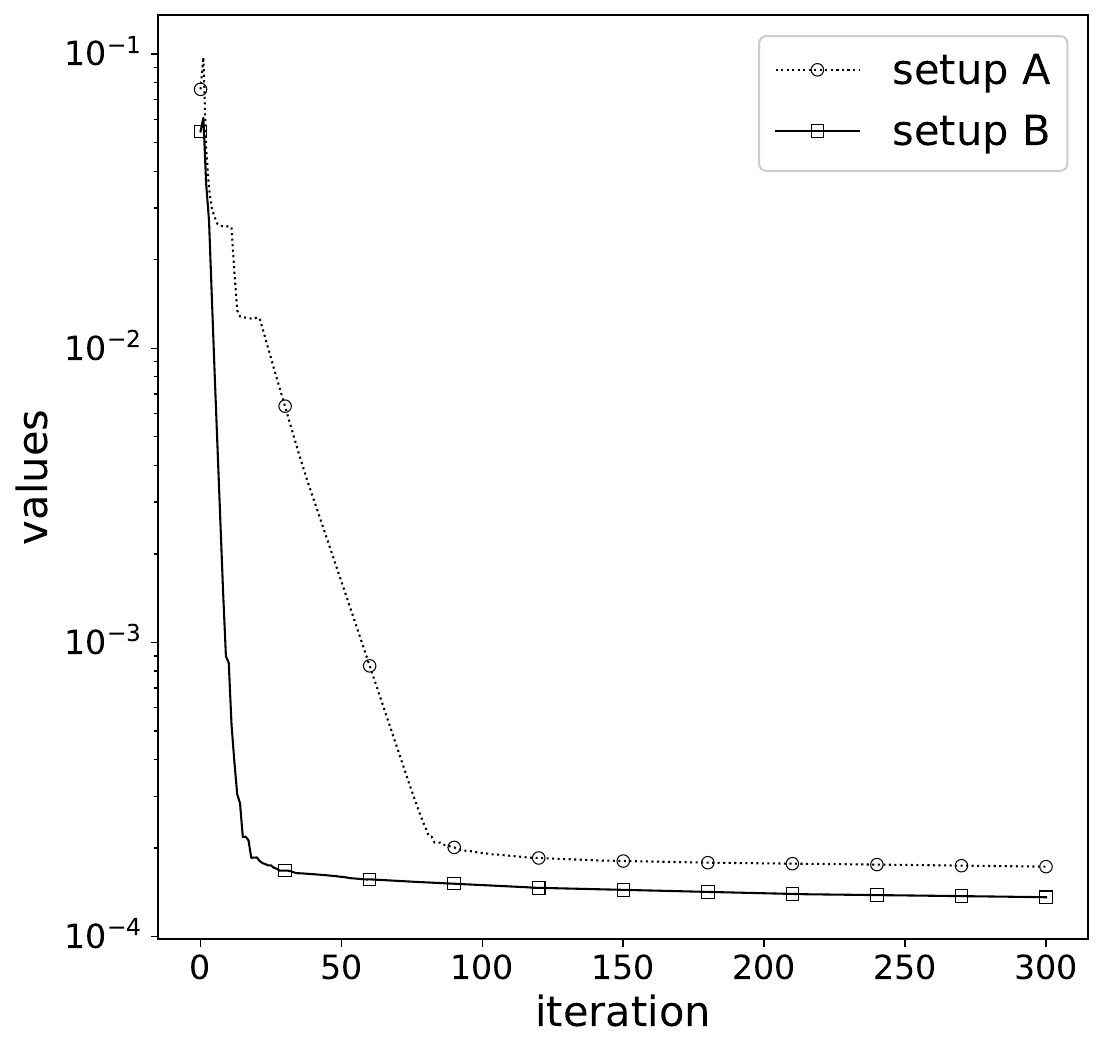}}\hfill
\caption{Detecting more than one object with noisy data ($\delta = 30\%$)}
\label{fig:two_obstacles}
\end{figure}
%
%
\begin{remark}
Obviously, the Lagrangian method adopted here to detect objects does not permit the modification of the topology of the object $\overline{\omega}$. 
Thus, in order to detect two or more inclusions, we need to know how many objects are included in $D$. 
One solution could be to initialize the algorithm using the notion of topological gradient (see, for example, \cite{BenAbdaetal2009,GarreauGuillaumeMasmoudi2001,SokolowskiZochowski1999}), which could provide us with the number of inclusions and their rough locations, offering initial shapes for our optimization method.
\end{remark}
\subsection{Numerical tests in 3D}
In our investigation, we extend the application of our proposed method to 3D cases. 
The computational setup closely mirrors that employed in 2D settings. 
Specifically, the forward problems are solved using tetrahedra characterized by smaller mesh widths and volumes, while larger ones are used in the inversion process.
Furthermore, the finite elements utilized differ between the forward and inverse problems. 
Specifically, $P3/P2$ finite elements are employed in the context of the forward problems, whereas $P2/P1$ finite elements are used for the inverse problems.
In the subsequent numerical experiments, we prescribe the input data as $\bgg = [1, 1, 1]$ and set $\mu = 1$ during the step size computation. 
The maximum number of iterations is set to $30$ for the first two examples, while it is set to $100$ for the third example. 
We do not employ adaptive remeshing to reduce the computational cost associated with the approximation.
\subsubsection{Large convex obstacle} For large convex shapes, such as an elliptical obstacle shown in Figure \ref{fig:large_ellipsoid}, the method works perfectly even in the presence of a high level of noise (e.g., $\delta = 30\%$). 
As an illustration, the plot shown in Figure \ref{fig:large_ellipse_approximation} is a reconstruction of the exact shape in Figure \ref{fig:large_ellipsoid} with an initial guess given by a sphere of radius $0.65$.
\begin{figure}[htp!]
\centering
\resizebox{0.235\textwidth}{!}{\includegraphics{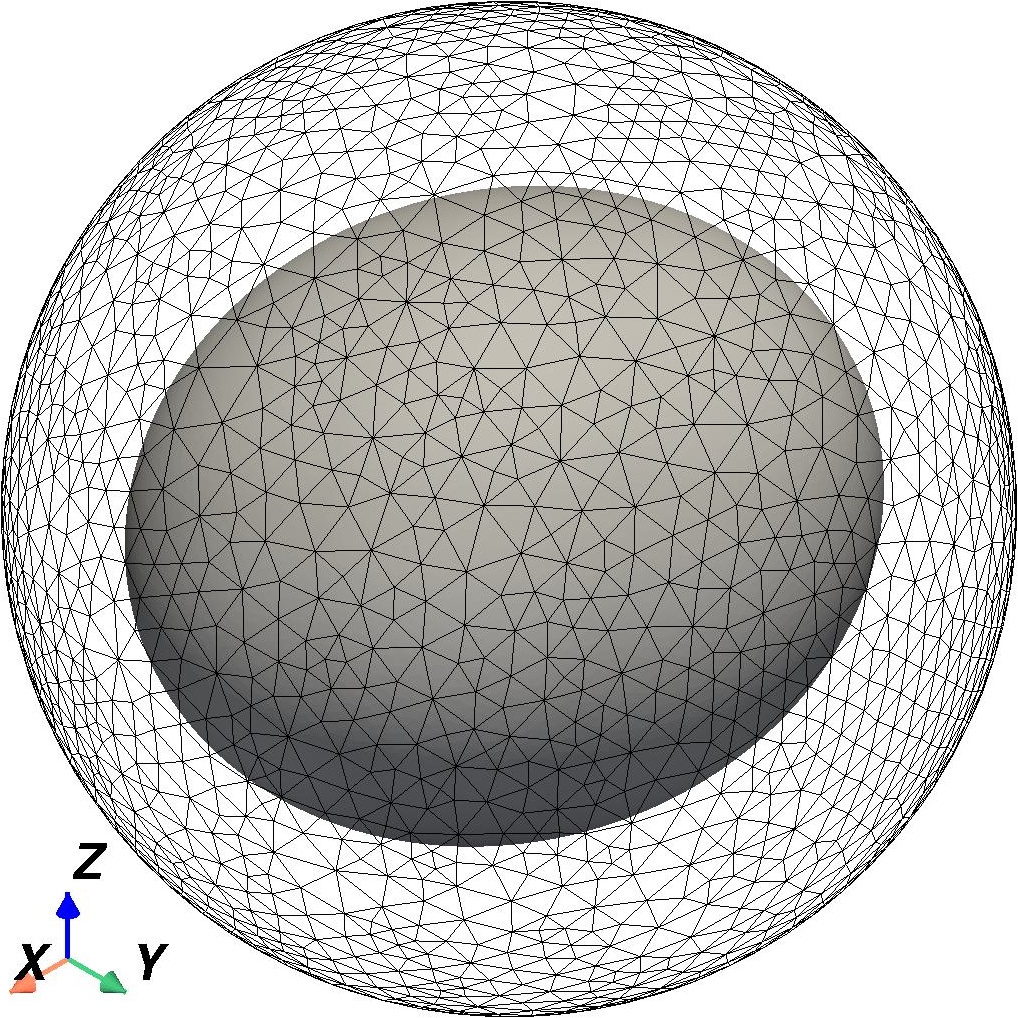}} \hfill  
\resizebox{0.235\textwidth}{!}{\includegraphics{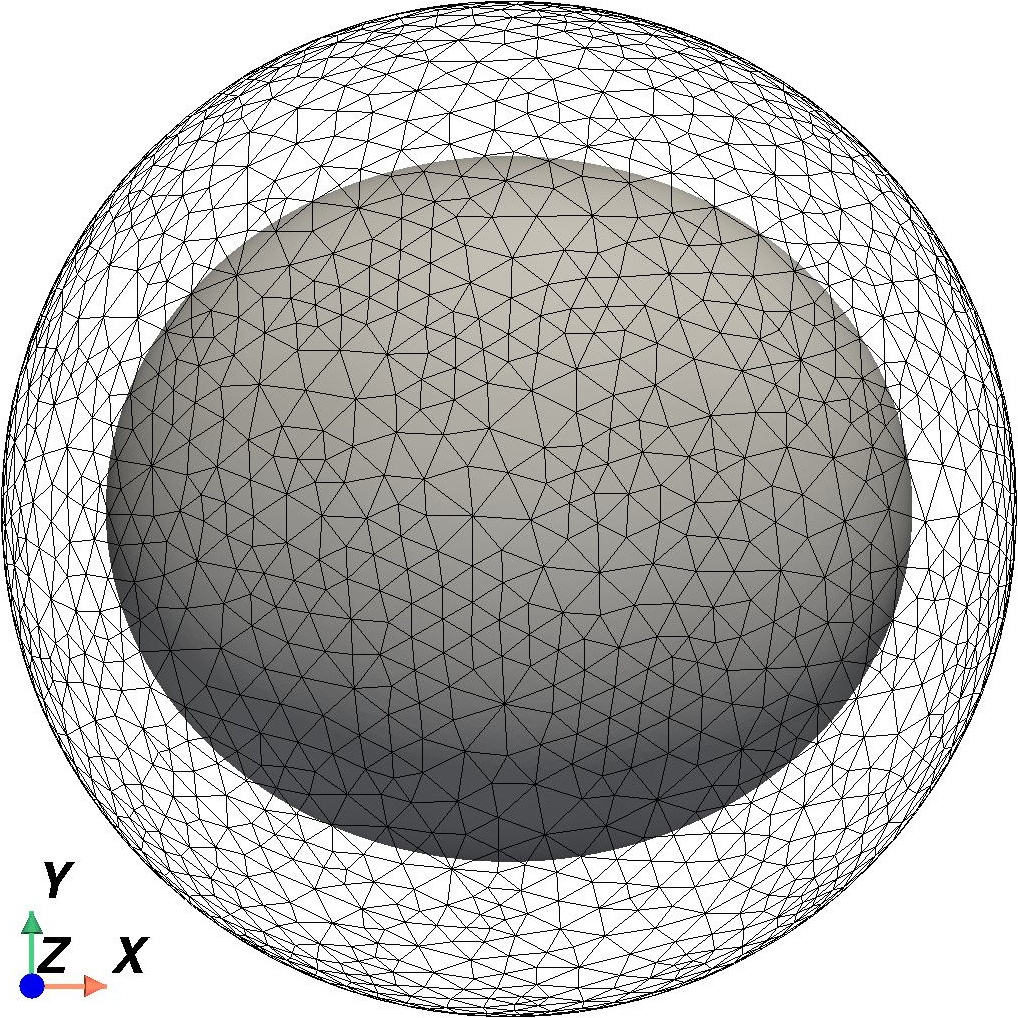}} \hfill  
\resizebox{0.235\textwidth}{!}{\includegraphics{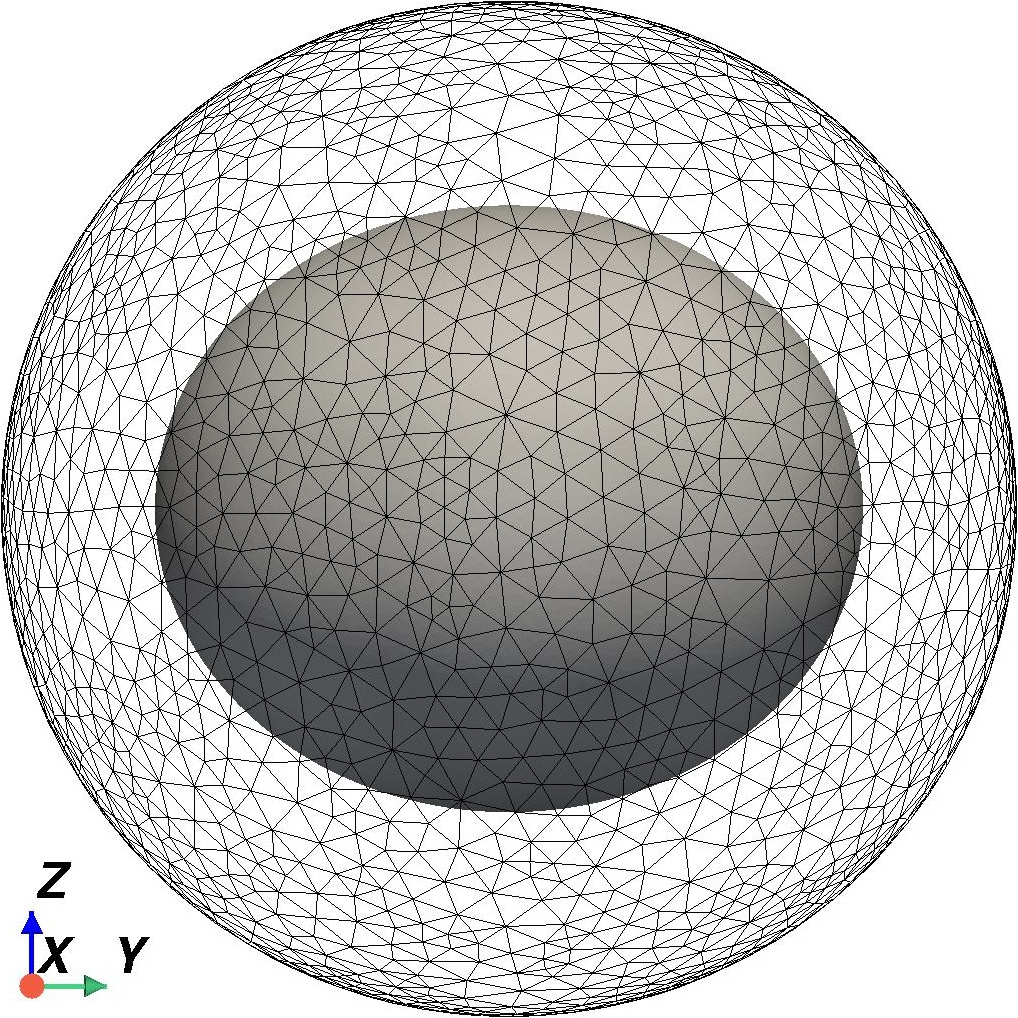}} \hfill  
\resizebox{0.235\textwidth}{!}{\includegraphics{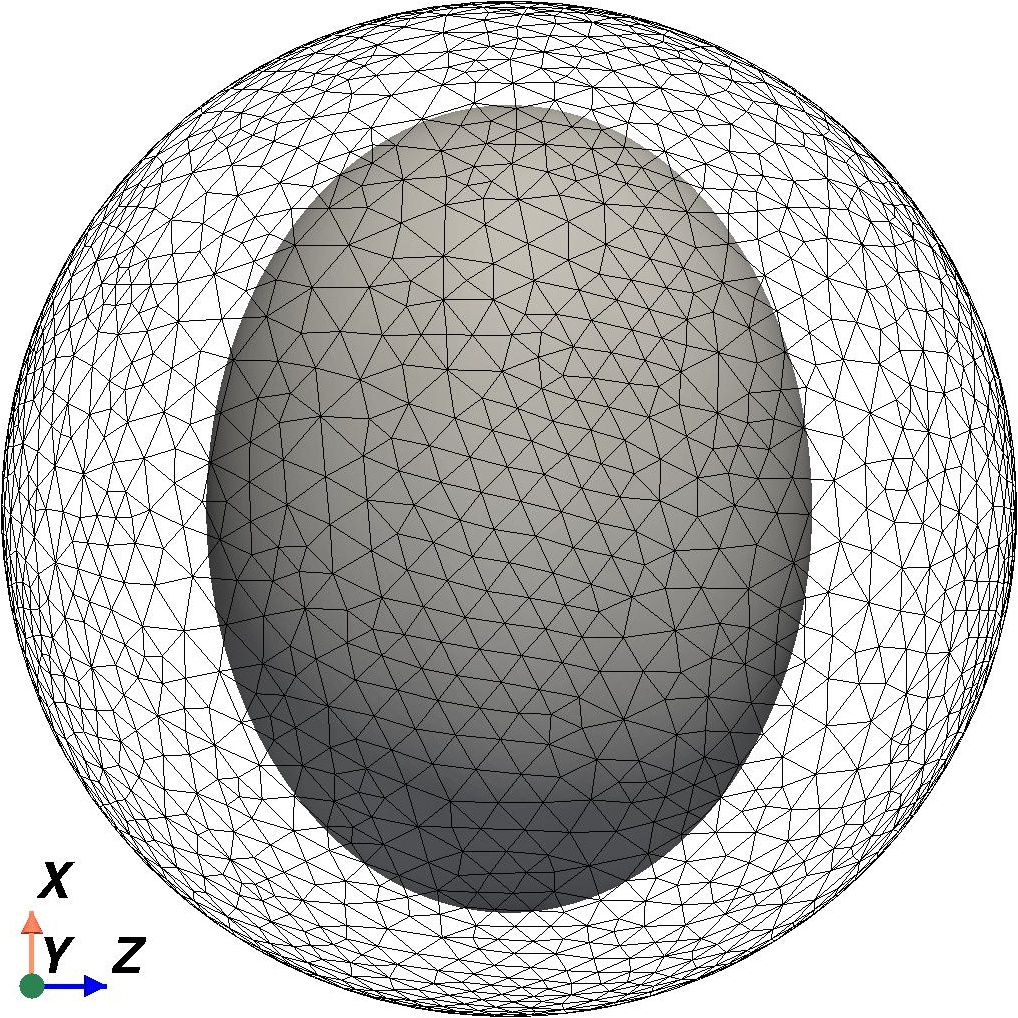}} \hfill  
\caption{Exact ellipsoidal shape obstacle}
\label{fig:large_ellipsoid}
\end{figure}
\begin{figure}[htp!]
\centering
\resizebox{0.235\textwidth}{!}{\includegraphics{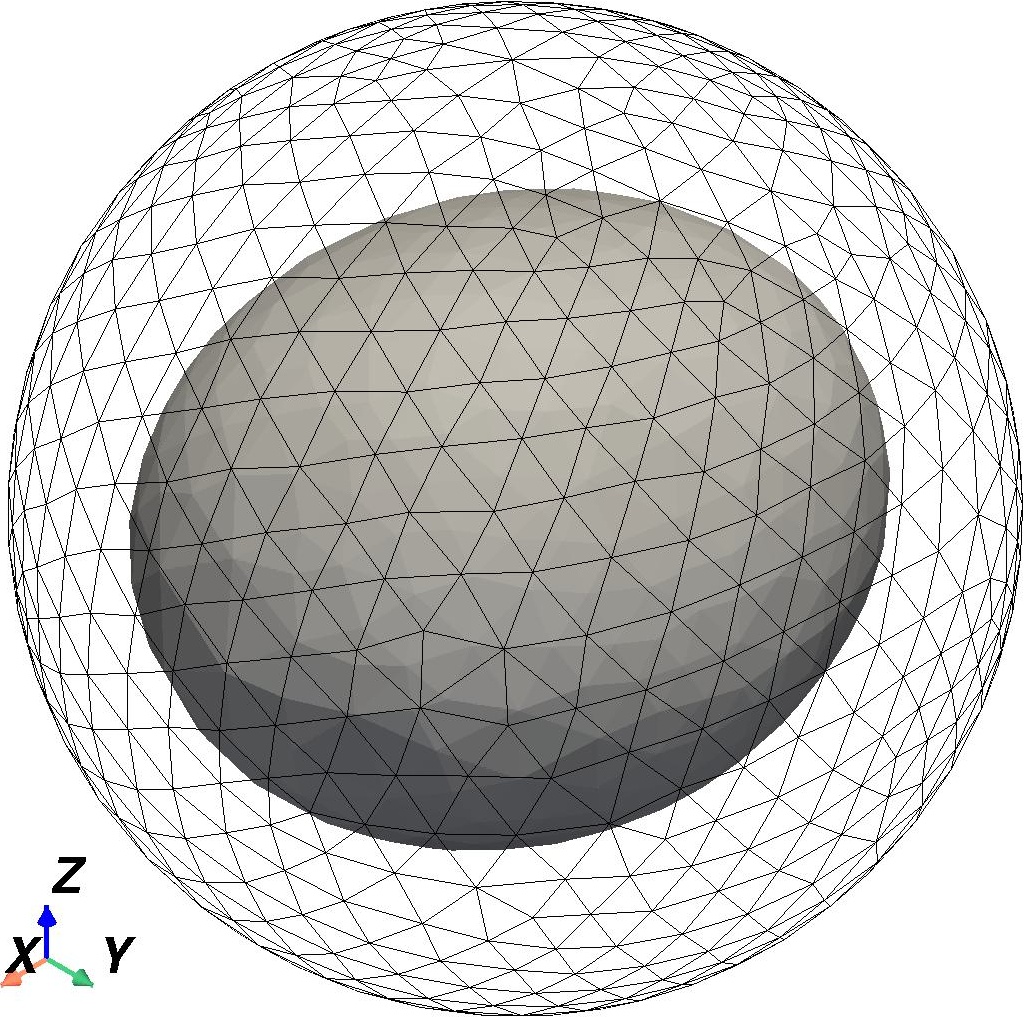}} \hfill  
\resizebox{0.235\textwidth}{!}{\includegraphics{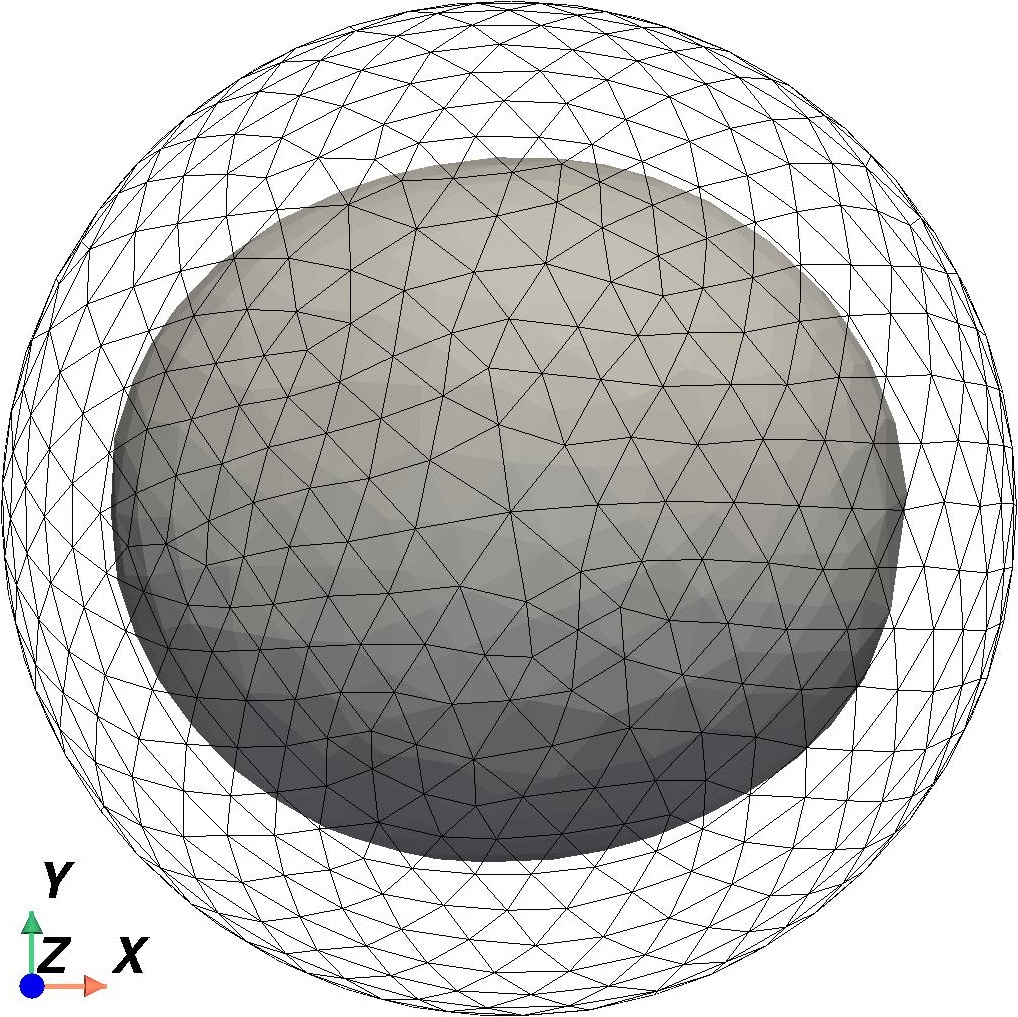}} \hfill  
\resizebox{0.235\textwidth}{!}{\includegraphics{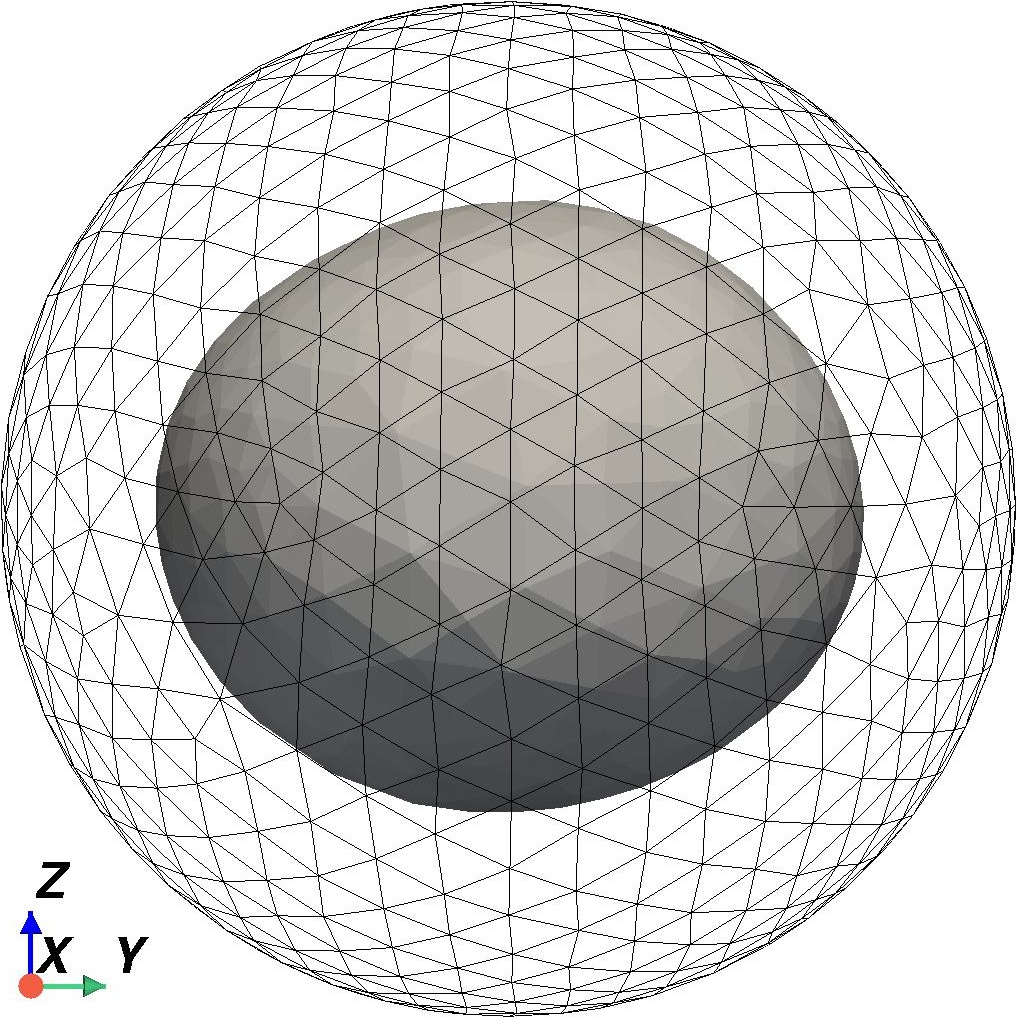}} \hfill  
\resizebox{0.235\textwidth}{!}{\includegraphics{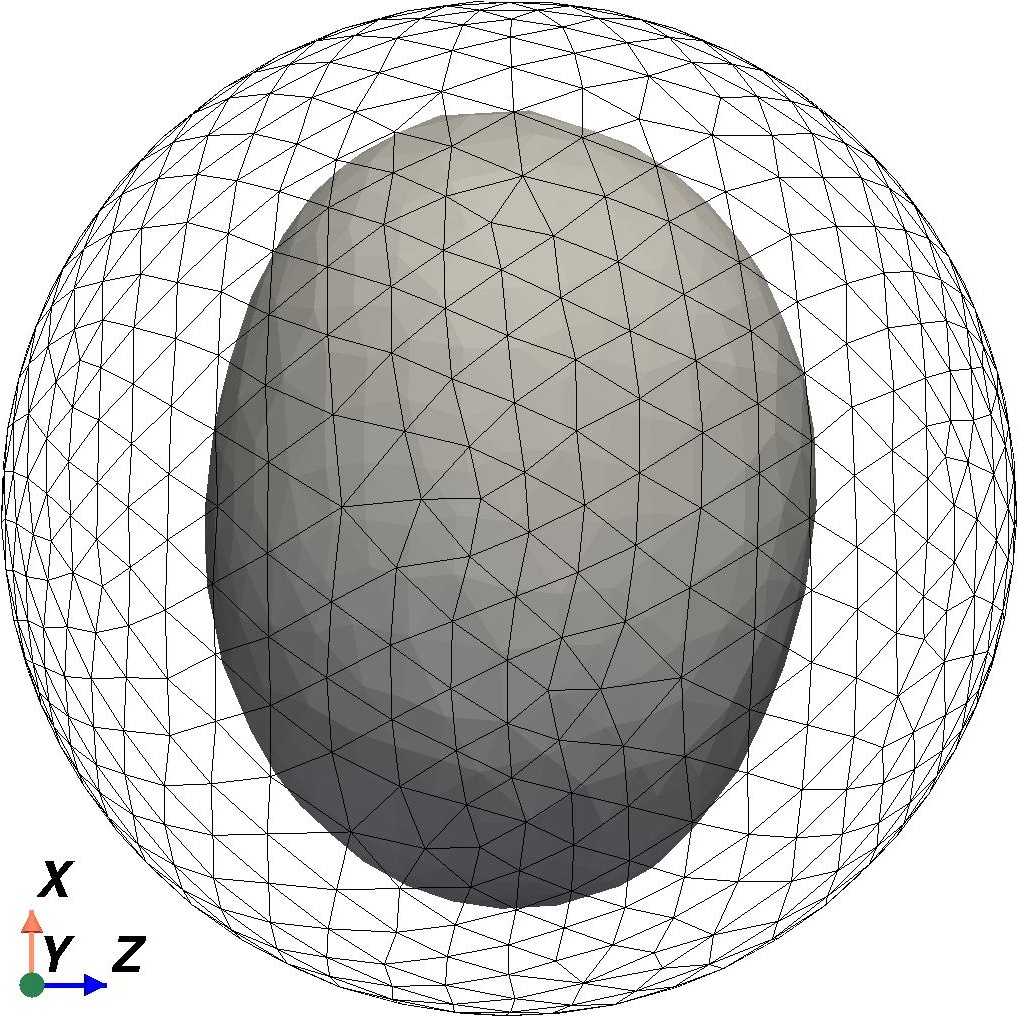}} \hfill  
\caption{Identified shape with noisy data ($\delta = 30\%$)}
\label{fig:large_ellipse_approximation}
\end{figure}
\subsubsection{Large non-convex obstacle} In the case of non-convex obstacles, such as a star-shaped obstacle depicted in Figure \ref{fig:large_star}, achieving accurate reconstructions appears to be more challenging, as expected. 
For instance, in Figure \ref{fig:large_star_approximation}, we can only achieve a reasonably accurate reconstruction of the precise geometry of the obstacle depicted in Figure \ref{fig:large_star} at high noise levels with $\delta = 30\%$ (and an initial guess given by a sphere of radius 0.65.). 
Nevertheless, we are able to identify some concavities within the shape. 
Based on our experience, it seems that regions with pronounced concavities are easier to reconstruct, while sharper or less pronounced concavities, especially those distant from the measurement region, prove difficult to detect.

\begin{figure}[htp!]
\centering
\resizebox{0.235\textwidth}{!}{\includegraphics{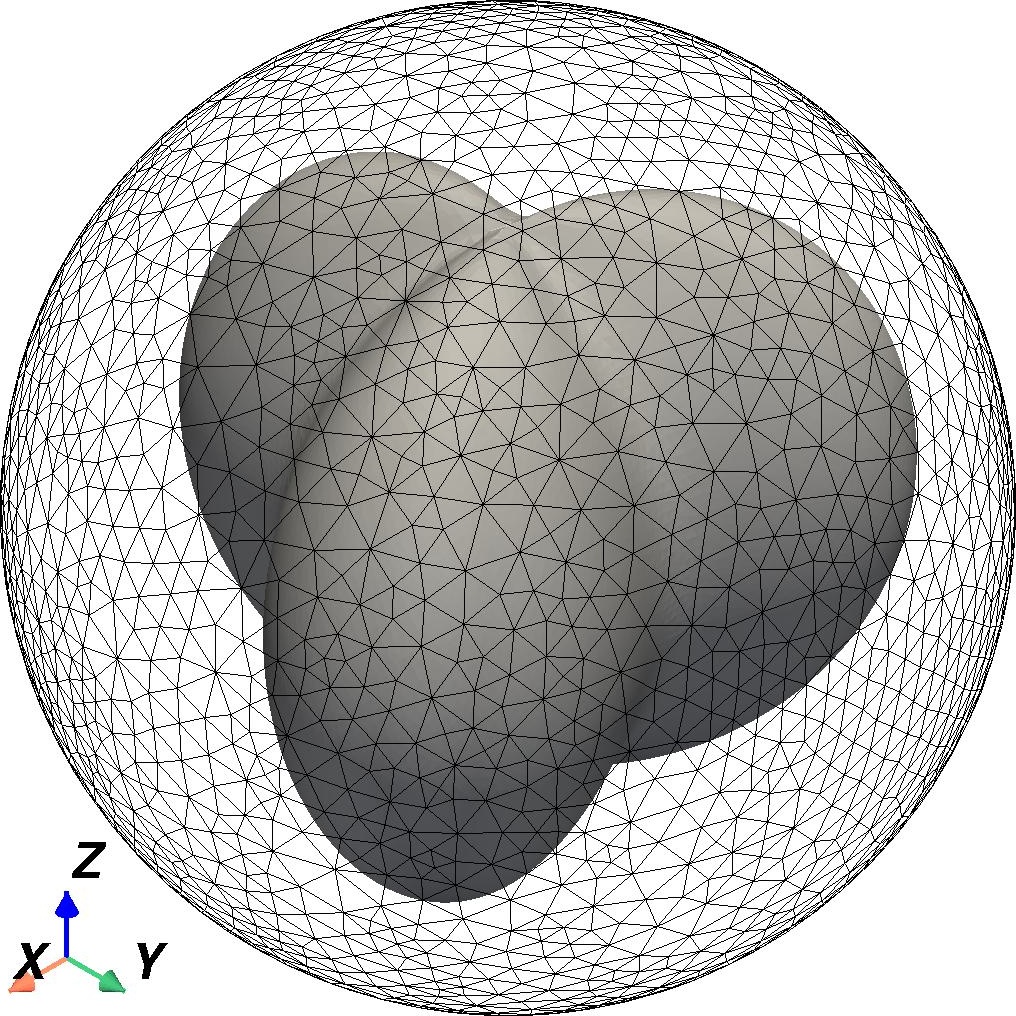}} \hfill  
\resizebox{0.235\textwidth}{!}{\includegraphics{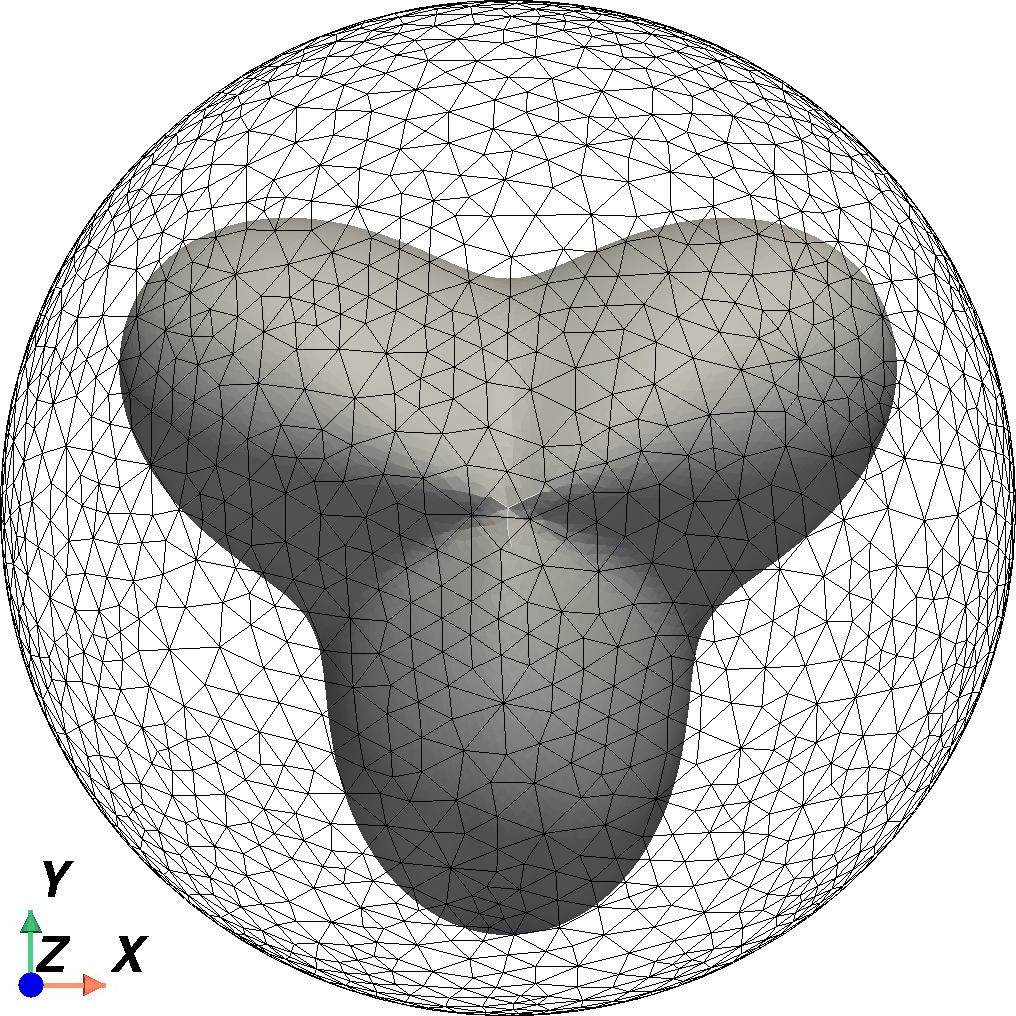}} \hfill  
\resizebox{0.235\textwidth}{!}{\includegraphics{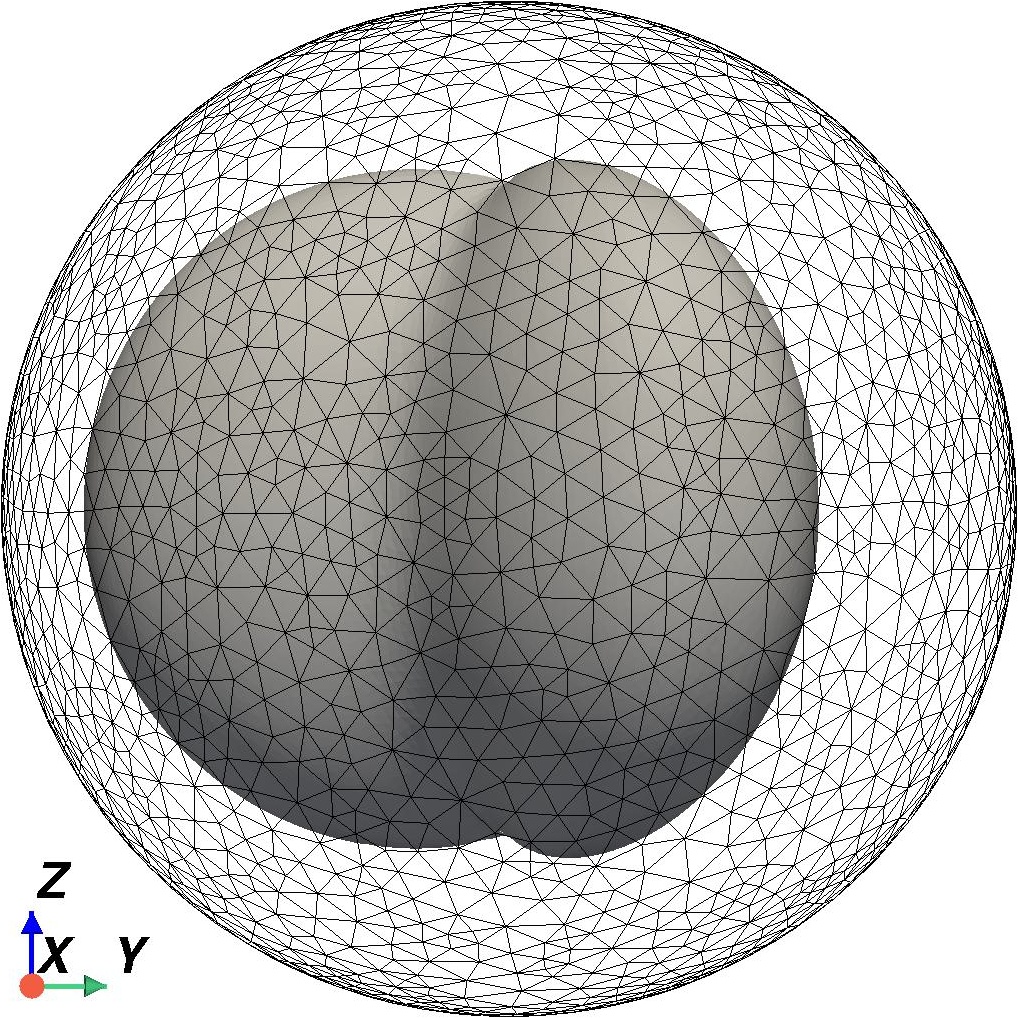}} \hfill  
\resizebox{0.235\textwidth}{!}{\includegraphics{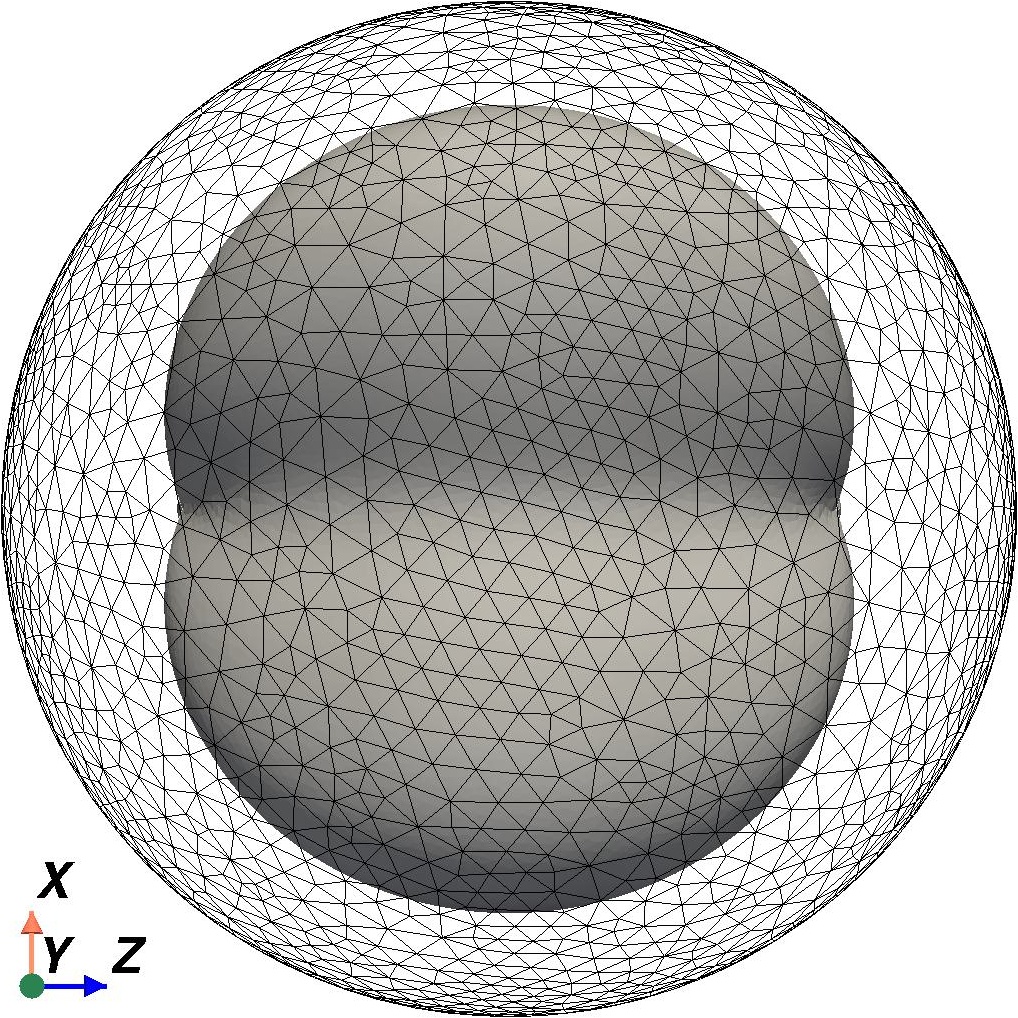}} \hfill  
\caption{Exact star-shape obstacle}
\label{fig:large_star}
\end{figure}
\begin{figure}[htp!]
\centering
\resizebox{0.235\textwidth}{!}{\includegraphics{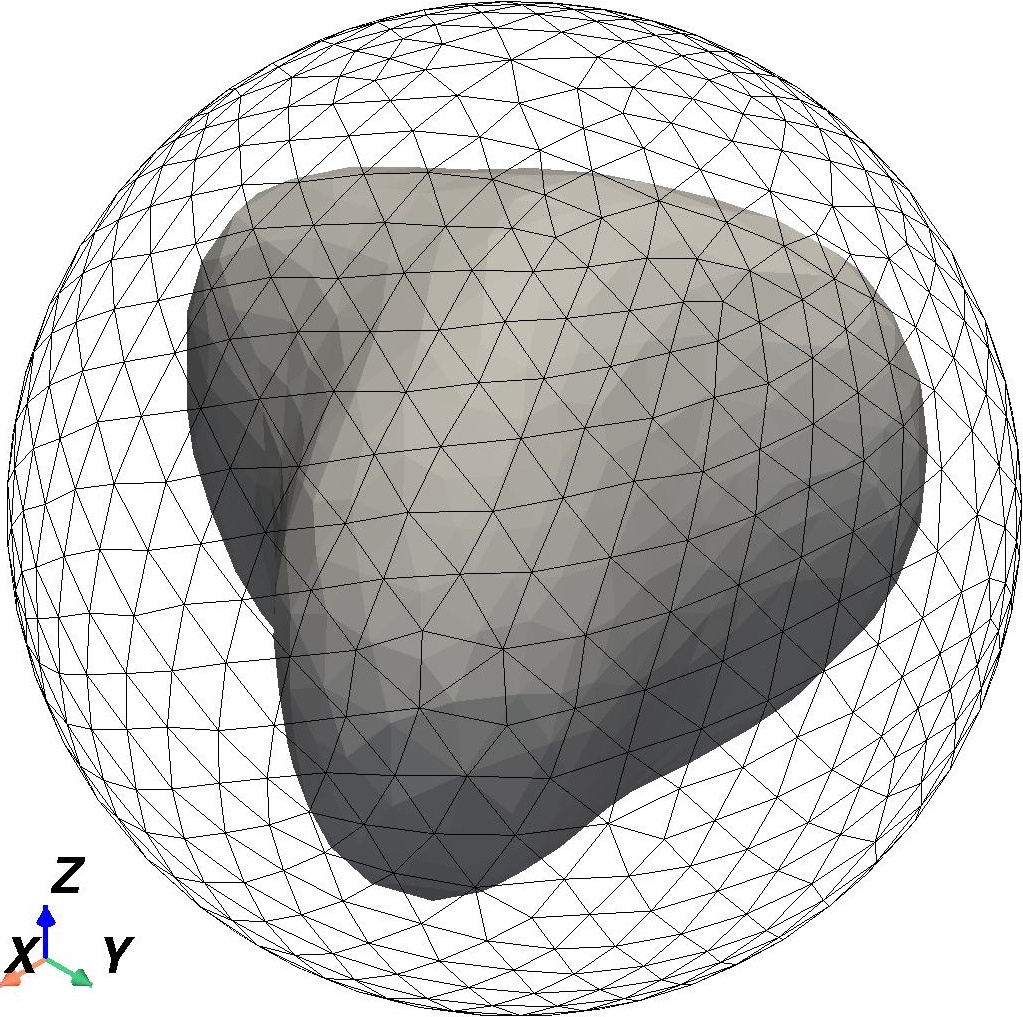}} \hfill  
\resizebox{0.235\textwidth}{!}{\includegraphics{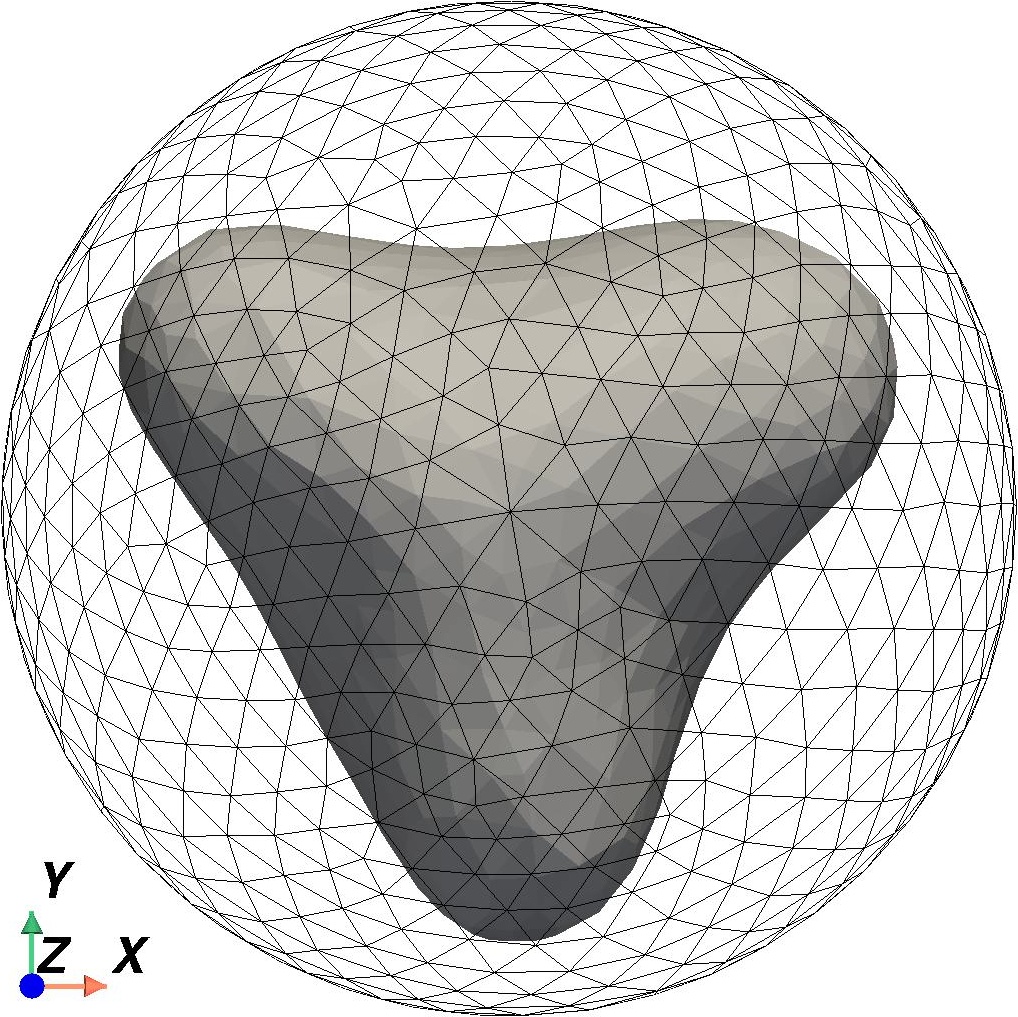}} \hfill  
\resizebox{0.235\textwidth}{!}{\includegraphics{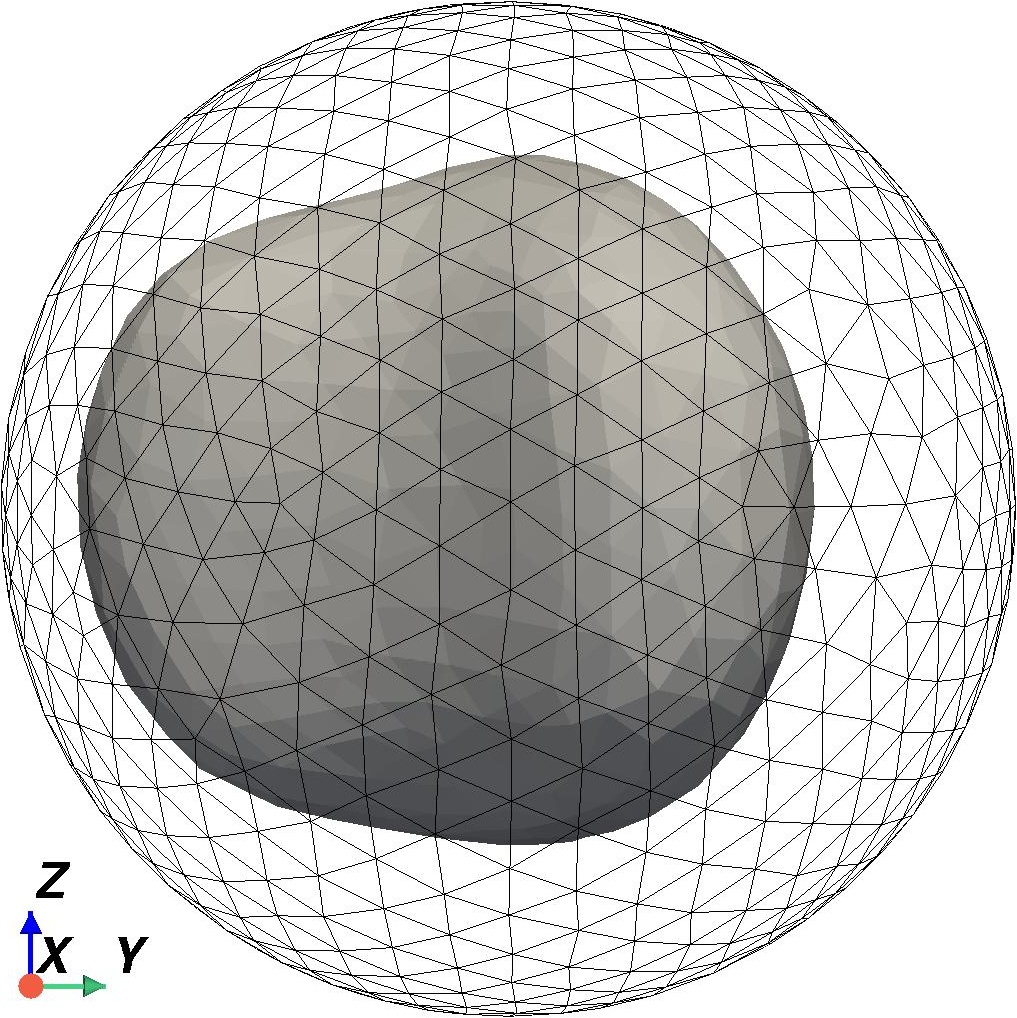}} \hfill  
\resizebox{0.235\textwidth}{!}{\includegraphics{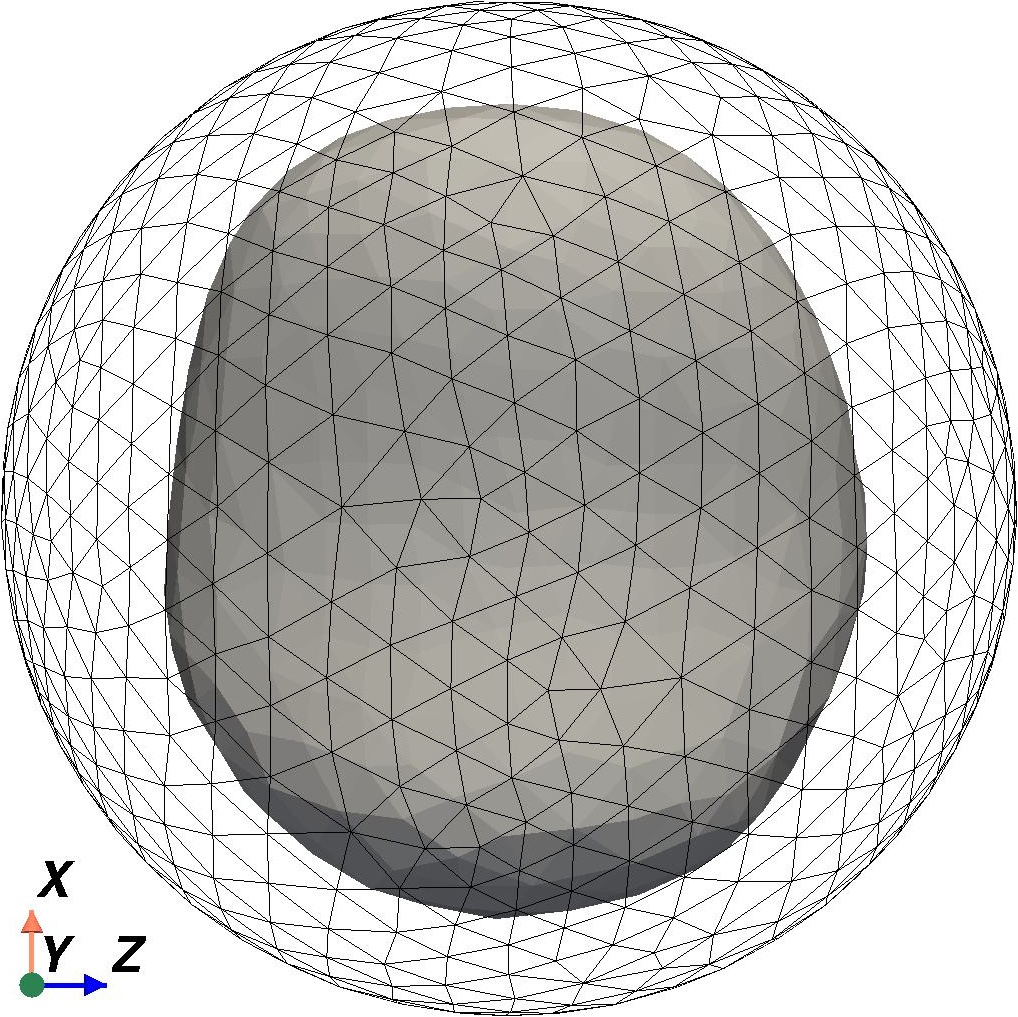}} \hfill  
\caption{Identified shape with noisy data ($\delta = 30\%$)}
\label{fig:large_star_approximation}
\end{figure}
\subsubsection{Two small cubic-shaped obstacles} Finally, we evaluate our algorithm using two cubic-shaped obstacles, as illustrated in Figure \ref{fig:two_cubes}. 
Despite violating the regularity assumption on the unknown interior boundary, this test example demonstrates the algorithm's ability to generate a reasonable reconstruction of the inclusion type, even in the presence of high noise levels ($\delta = 30\%$), as depicted in Figure \ref{fig:two_cubes_approximation}. 
As expected, it is difficult to accurately reconstruct the edges of the cubic obstacles, see Figure \ref{fig:two_cubes_comparison}. 
Nevertheless, it correctly predicts the location of the two obstacles, reaching convergence in the cost function after approximately 60 iterations.
\begin{figure}[htp!]
\centering
\resizebox{0.235\textwidth}{!}{\includegraphics{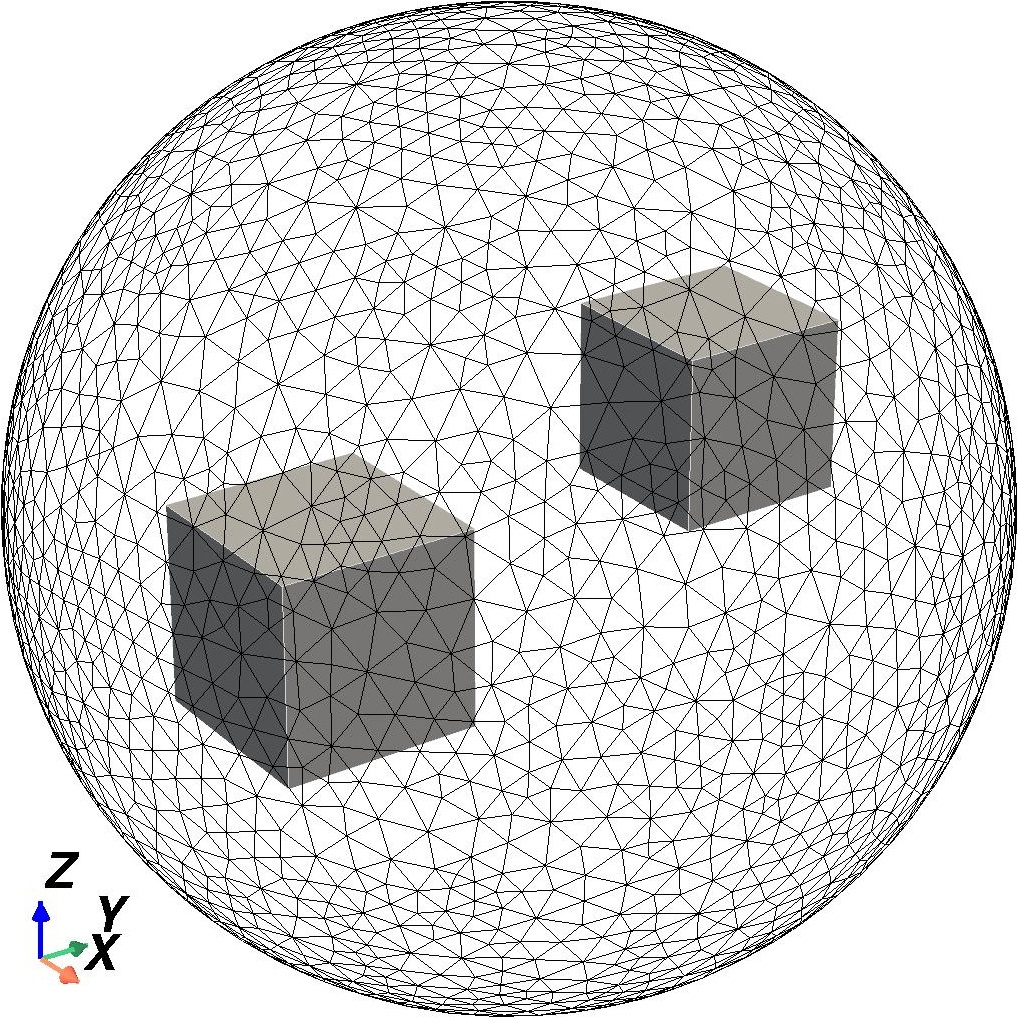}} \hfill  
\resizebox{0.235\textwidth}{!}{\includegraphics{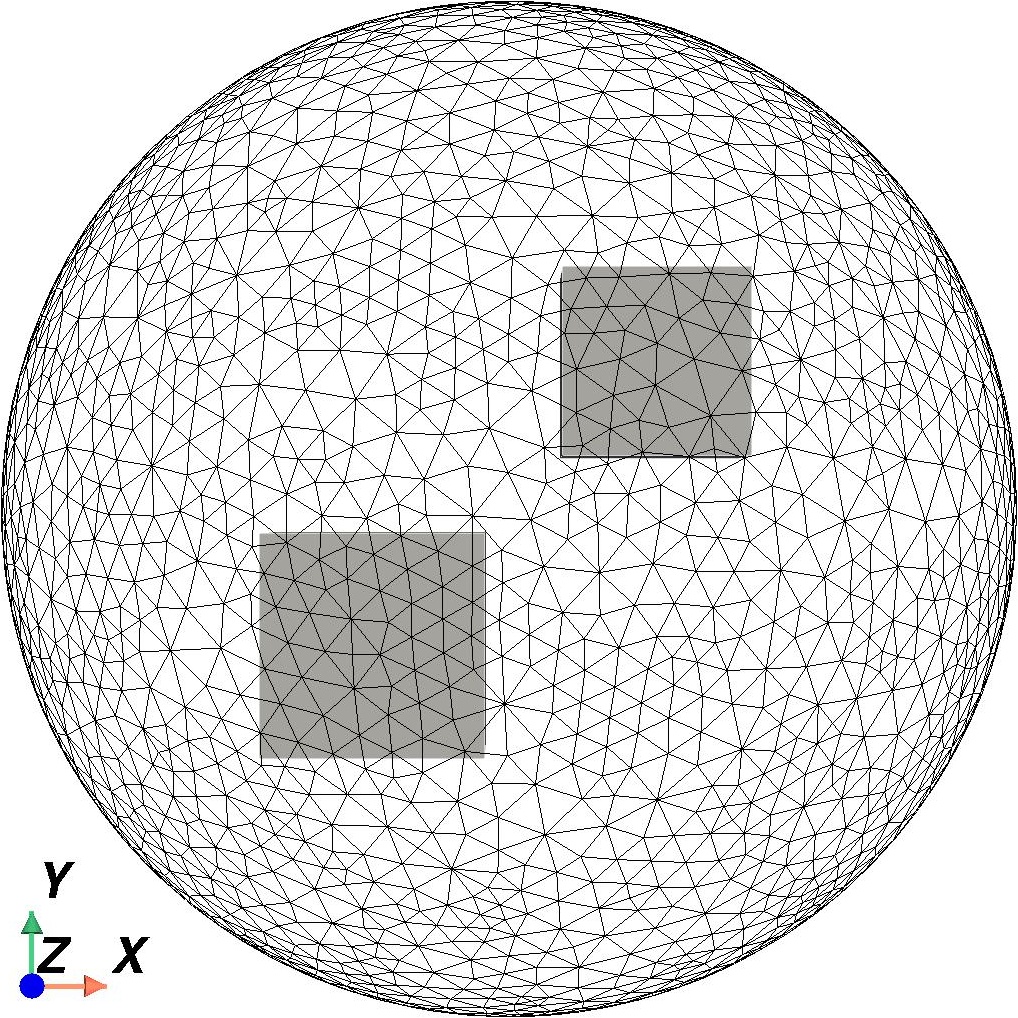}} \hfill  
\resizebox{0.235\textwidth}{!}{\includegraphics{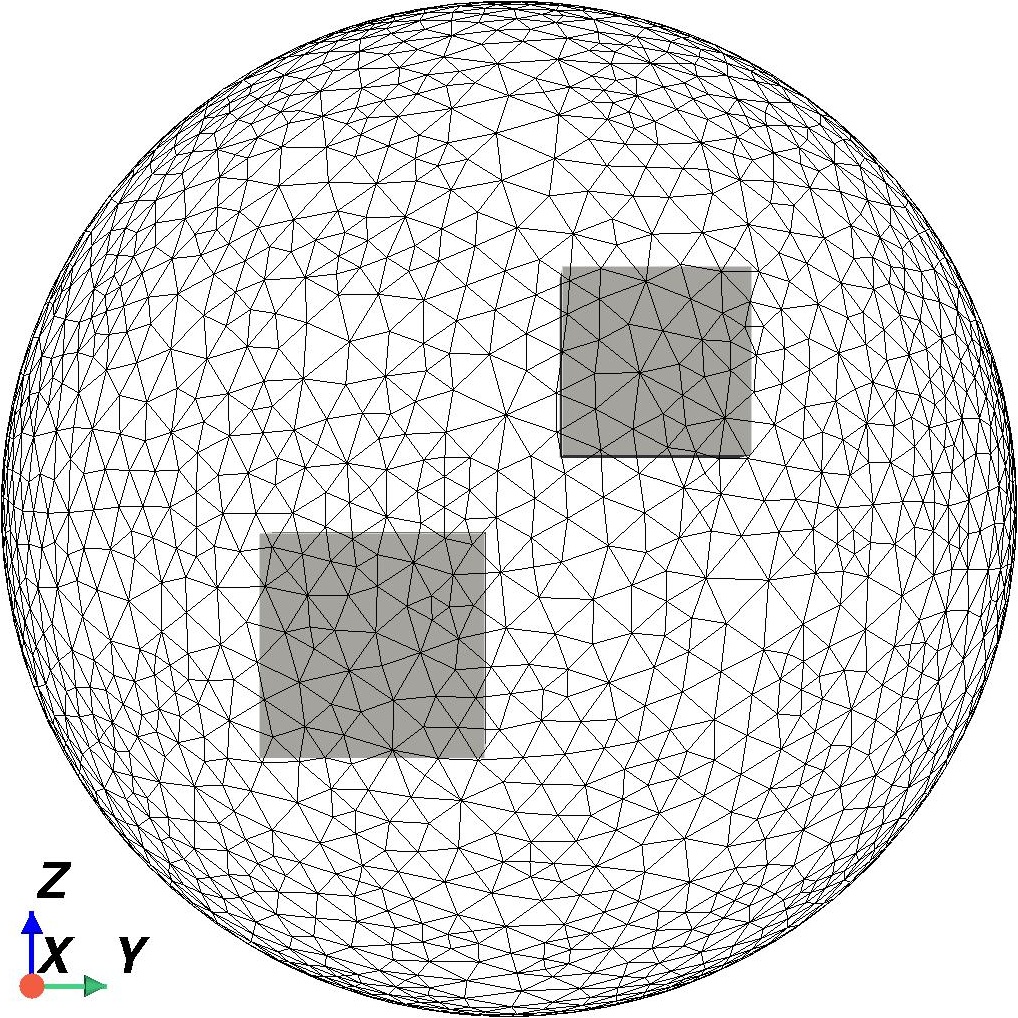}} \hfill  
\resizebox{0.235\textwidth}{!}{\includegraphics{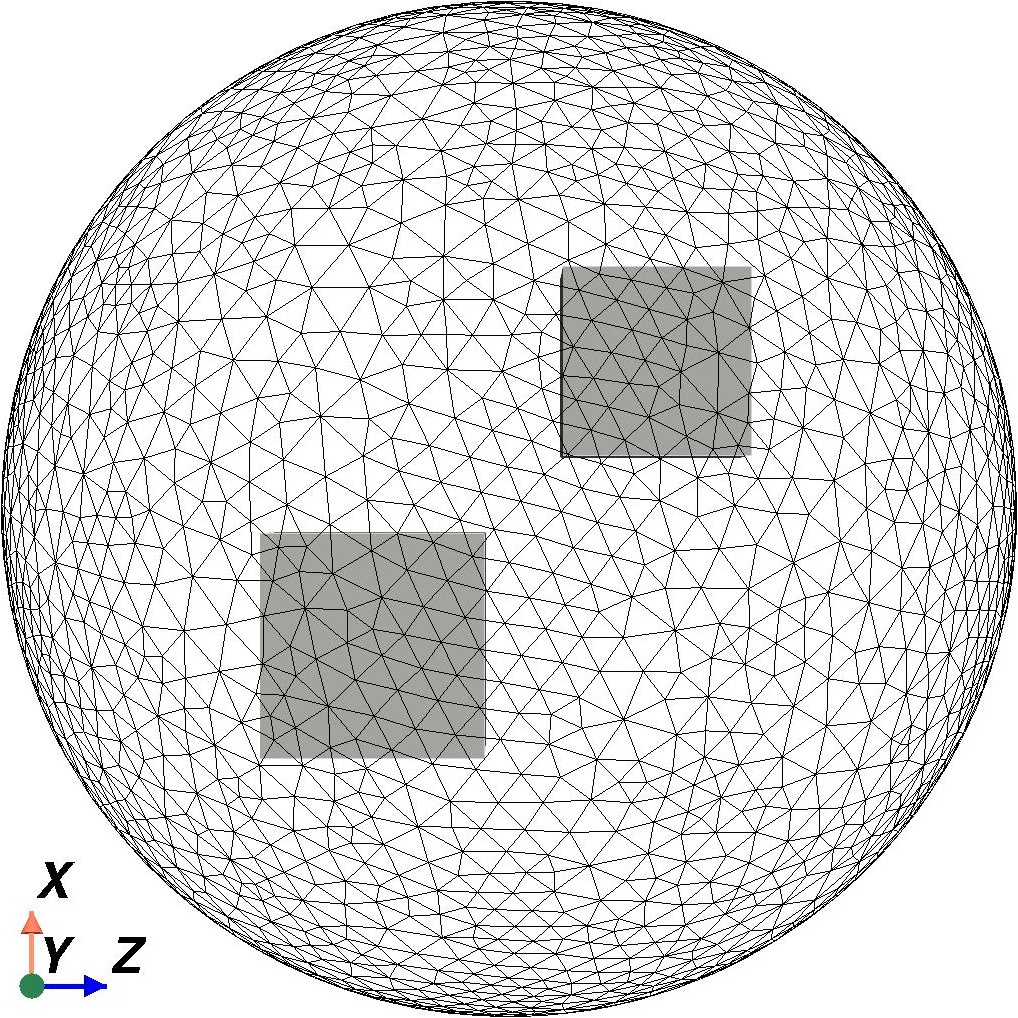}} \hfill  
\caption{Exact geometries}
\label{fig:two_cubes}
\end{figure}
\begin{figure}[htp!]
\centering
\resizebox{0.235\textwidth}{!}{\includegraphics{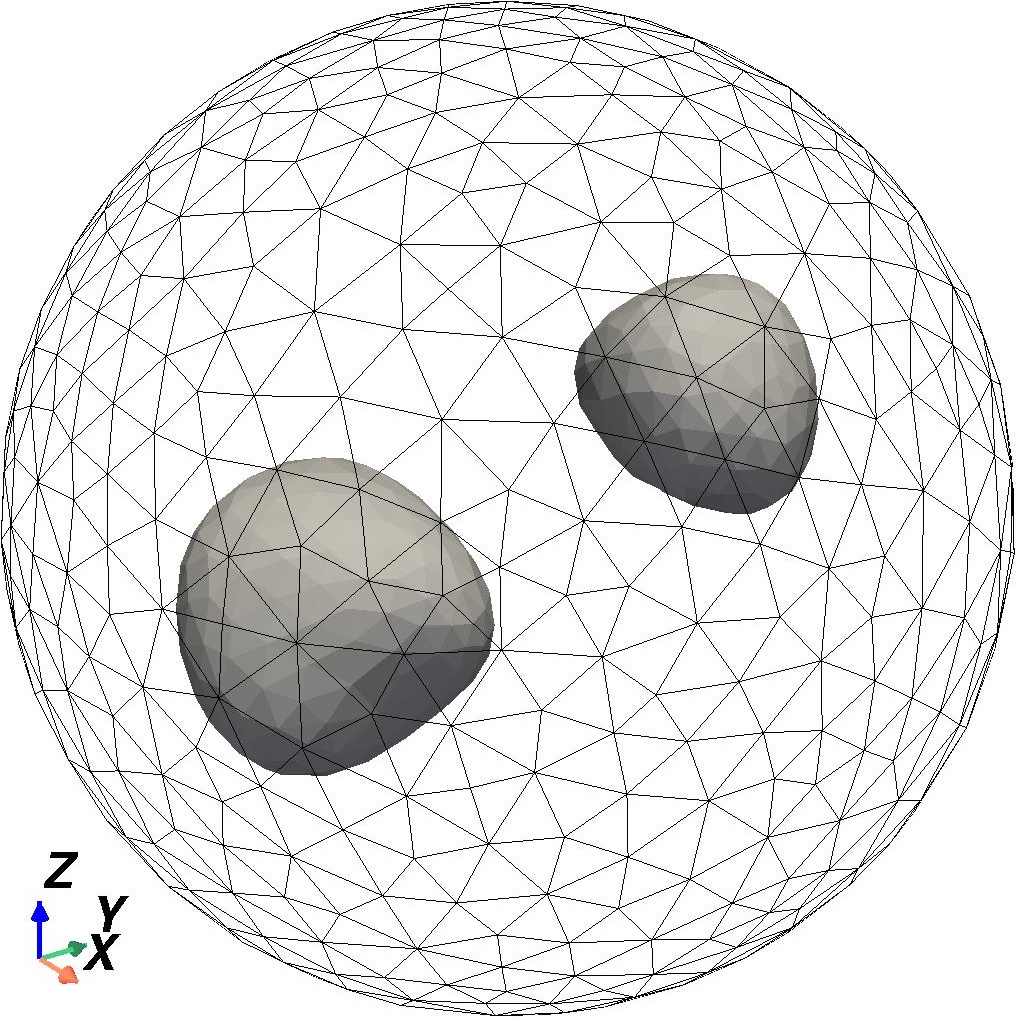}} \hfill  
\resizebox{0.235\textwidth}{!}{\includegraphics{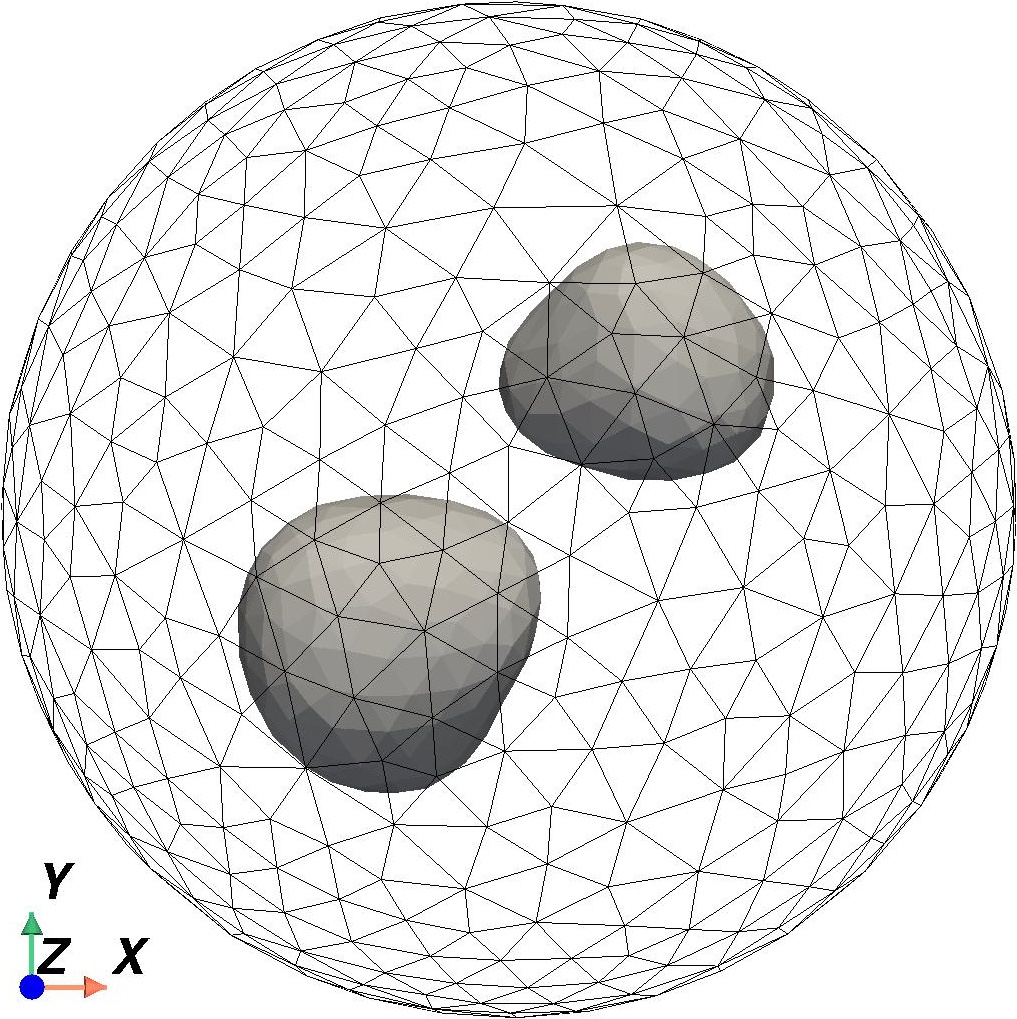}} \hfill  
\resizebox{0.235\textwidth}{!}{\includegraphics{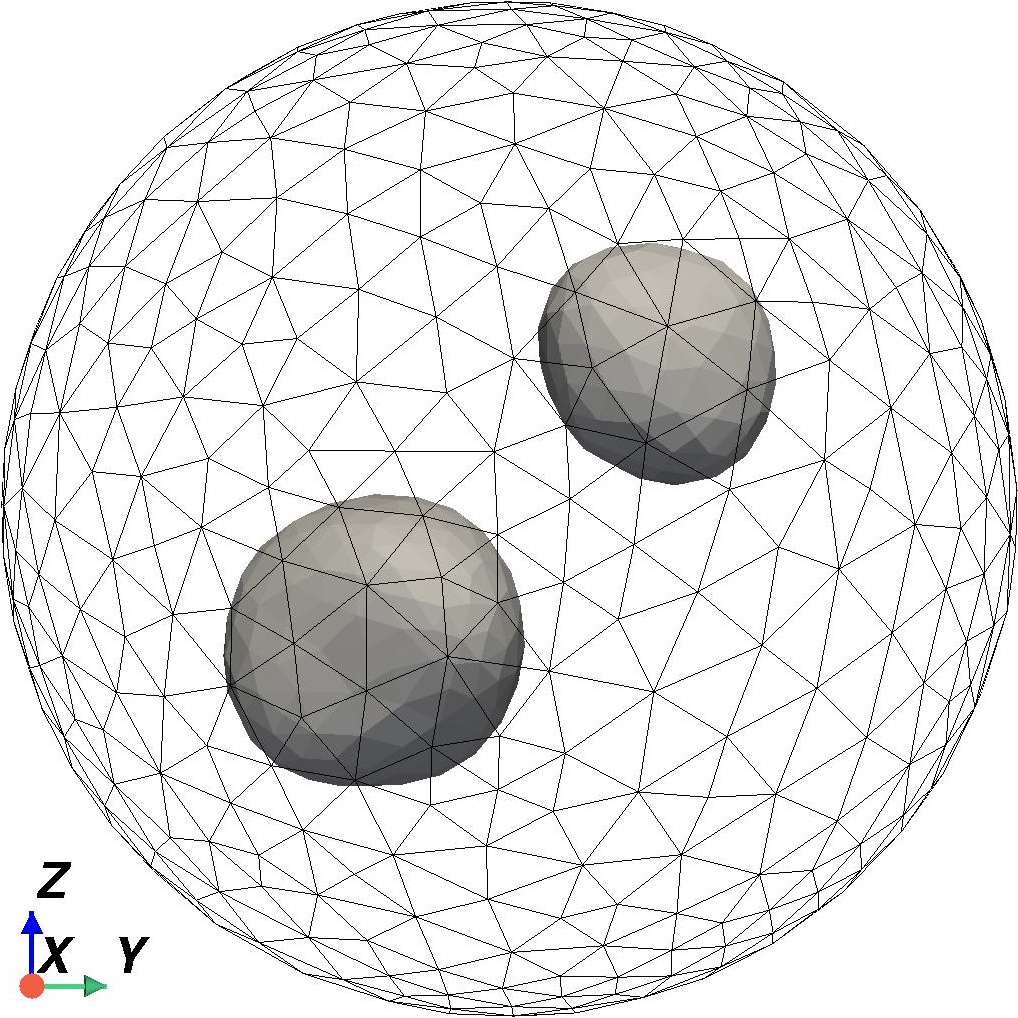}} \hfill  
\resizebox{0.235\textwidth}{!}{\includegraphics{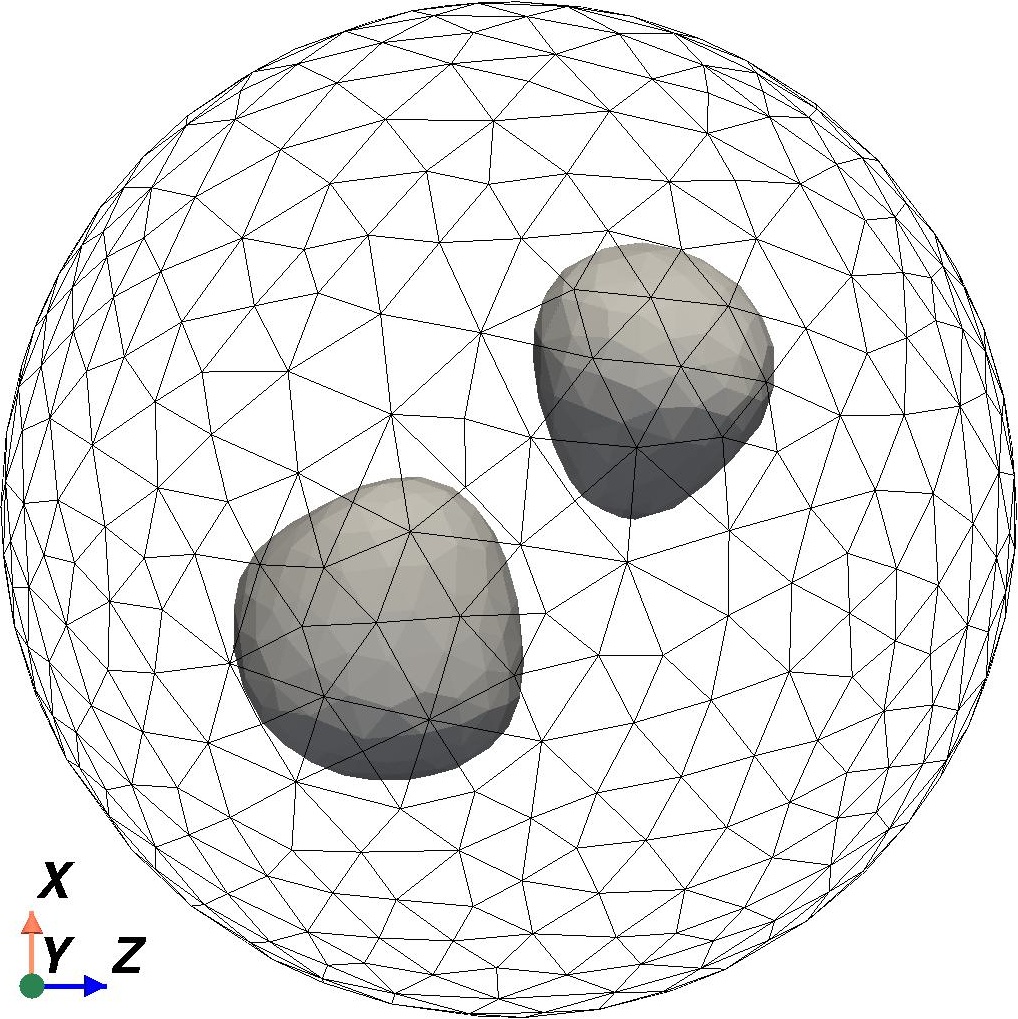}} \hfill  
\caption{Identified shapes with noisy data ($\delta = 30\%$)}
\label{fig:two_cubes_approximation}
\end{figure}
\begin{figure}[htp!]
\centering
\hfill
\resizebox{0.235\textwidth}{!}{\includegraphics{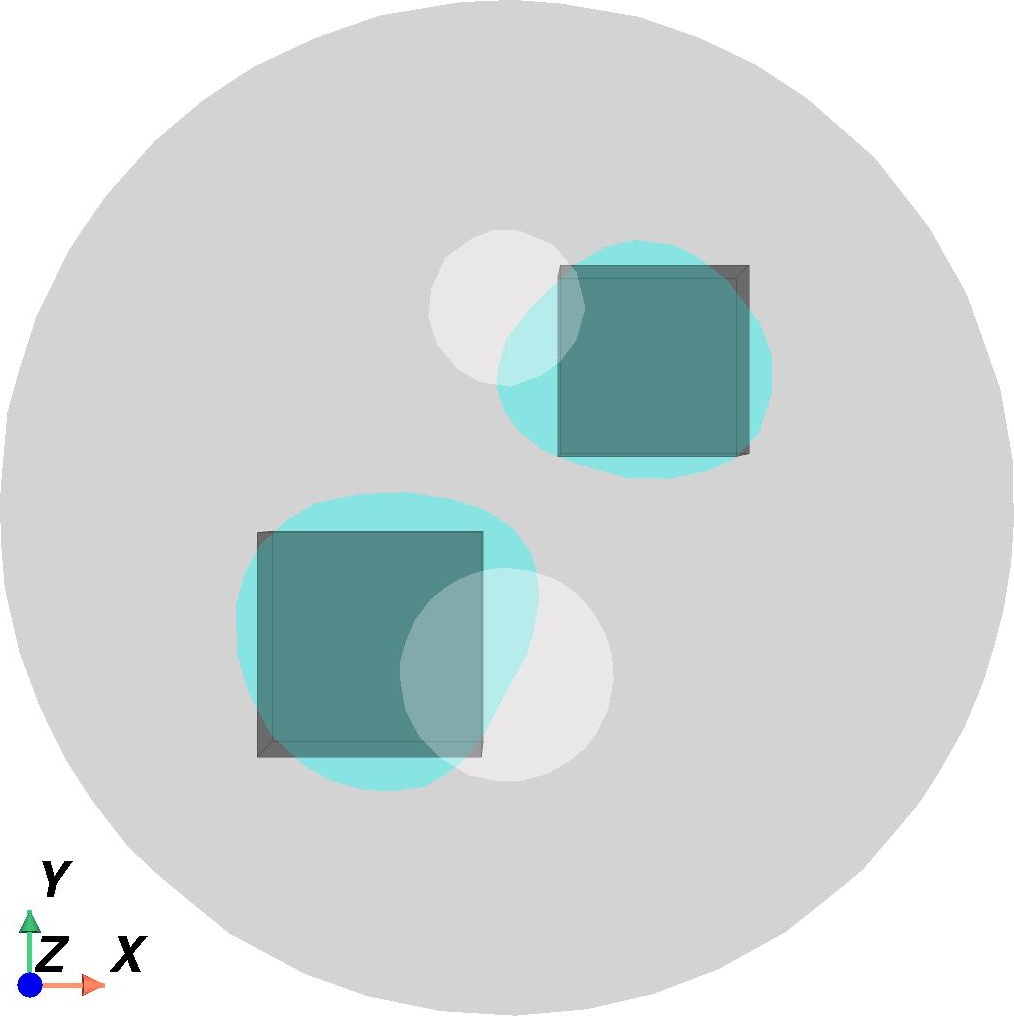}} \hfill  
\resizebox{0.235\textwidth}{!}{\includegraphics{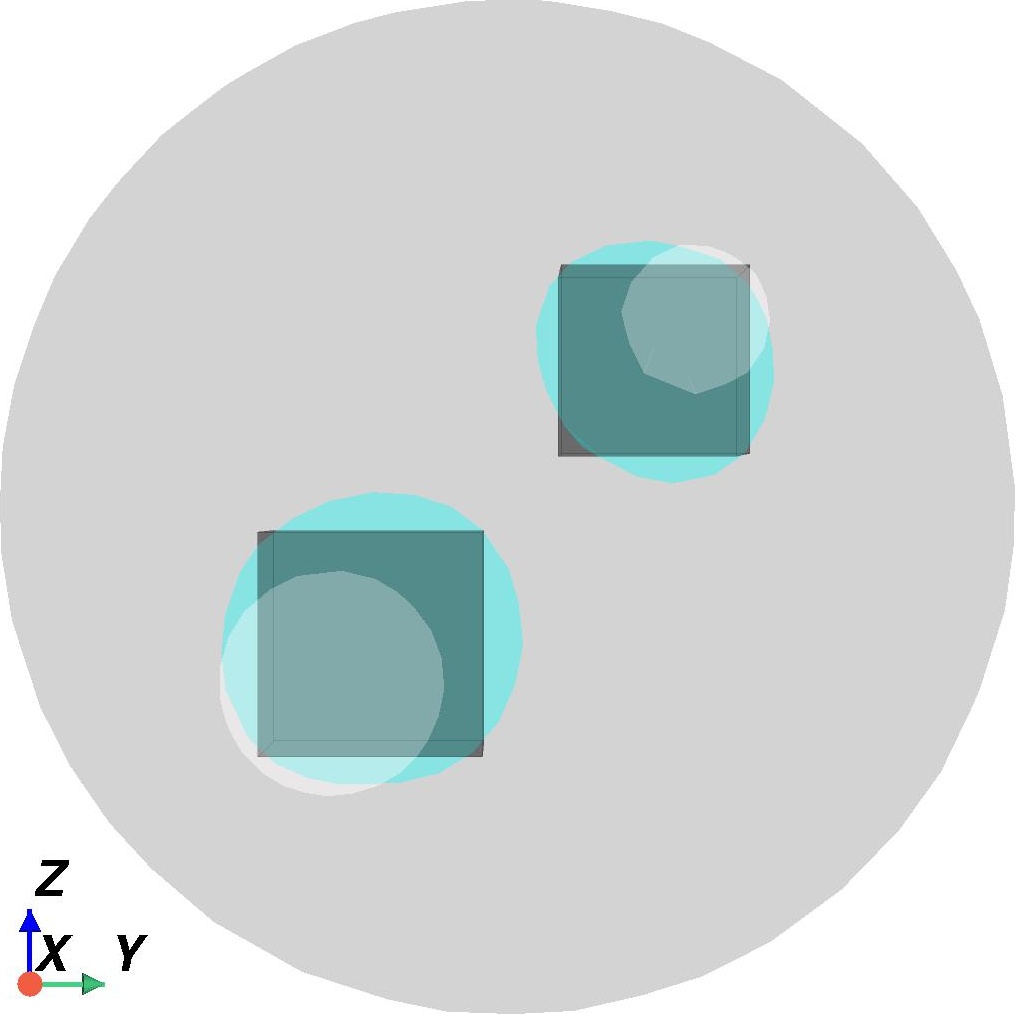}} \hfill  
\resizebox{0.235\textwidth}{!}{\includegraphics{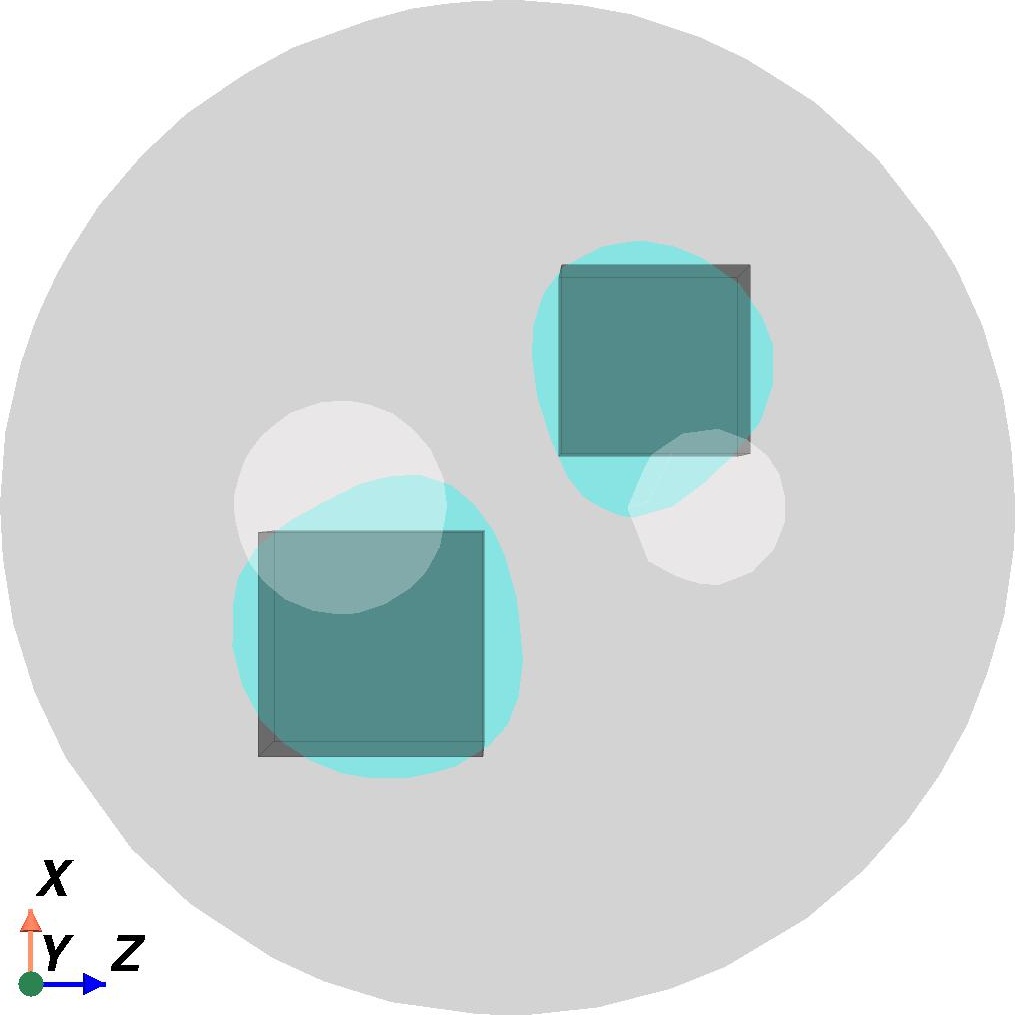}} \hfill  
\resizebox{0.235\textwidth}{!}{\includegraphics{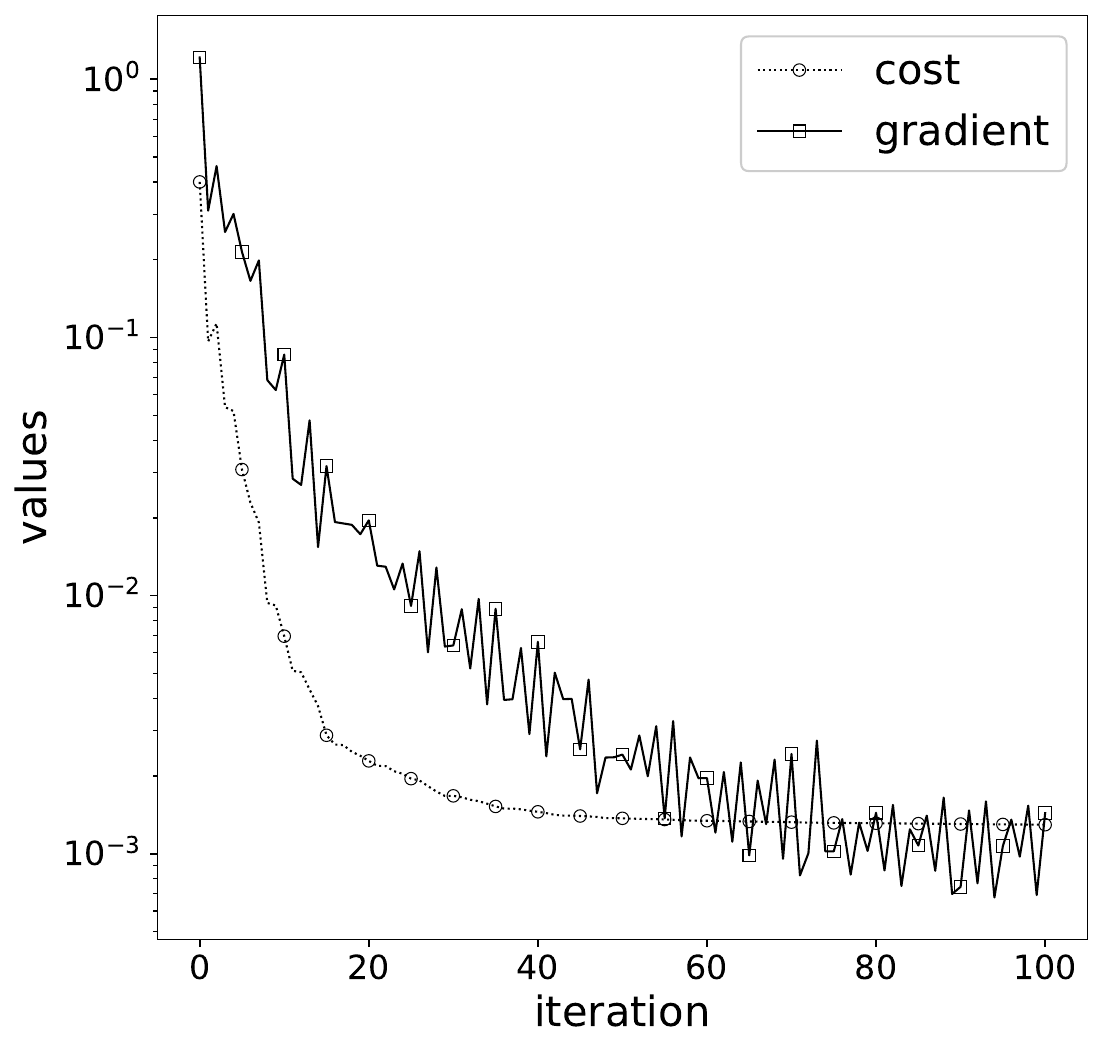}}
\caption{Cross-comparisons of exact (blackest color), initial (whitest color), and identified (cyan color) shapes corresponding to Figures \ref{fig:two_cubes} and \ref{fig:two_cubes_approximation}, along with histories of cost and gradient norms (rightmost plot)}
\label{fig:two_cubes_comparison}
\end{figure}

To conclude, we comment that in the first two examples, the algorithm took approximately $390$ seconds for $30$ iterations with remeshing every five steps. In the third example, it took around $10,000$ seconds to complete $100$ iterations with remeshing every two steps.

In general, our experiments reveal that smaller obstacles and parts distant from the measurement region present greater challenges for reconstruction. 
Nevertheless, the algorithm effectively detects the obstacle's location and provides a reasonable reconstruction of its geometry, even in the presence of noise in the measurements.

%
\section{Conclusion}\label{sec:conclusion}
The coupled complex boundary method is employed to address a shape-inverse obstacle problem in Stokes fluid flow. 
The study reformulates the inverse problem as a shape optimization setup, utilizing least-squares fitting of the imaginary part of the complex PDEs' solution in the flow domain. 
The unknown inclusion is treated as a nonparametric shape. 
The shape gradient of the associated cost functional is rigorously characterized using the rearrangement method, enabling effective numerical resolution through a Sobolev gradient descent method in a finite element setting. 
Numerical results in both two and three-dimensional cases demonstrate the method's effectiveness in detecting non-trivial obstacles, even under high-level noise-contaminated data.
\section*{Acknowledgements} 
The work of JFTR is supported by the JSPS Grant-in-Aid for Early-Career Scientists under Japan Grant Number JP23K13012.
HN is partially supported by JSPS Grants-in-Aid for Scientific Research under Grant Numbers JP20KK0058, JP21H04431, and JP20H01823.
JFTR and HN are also partially supported by the JST CREST Grant Number JPMJCR2014.
%
%
\bibliographystyle{alpha} 
\bibliography{./main}   

\appendix
\numberwithin{equation}{section}
\renewcommand{\theequation}{\thesection\arabic{equation}}
%
%
%
\section{Lemmata Proofs}\label{appendix:lemmata_proofs}
In this appendix, we present the proofs of the key lemmas employed in the previous section.
Let us first note that the following regularities hold (see, e.g., \cite{IKP2006,IKP2008}):
\begin{equation}\label{eq:regular_maps}
\left\{
\begin{aligned}
	[t \mapsto T_{t}] &\in \mathcal{C}^{1}(\mathcal{I},\mathcal{C}^{1,1}(\overline{D})^{d}),
		&& [t \mapsto T_{t}^{-1}] \in \mathcal{C}(\mathcal{I},\mathcal{C}^{1}(\overline{D})^{d}),\\
	[t \mapsto DT_{t}] &\in \mathcal{C}^{1}(\mathcal{I},\mathcal{C}^{0,1}(\overline{D})^{d\times d}),
		&& [t \mapsto (DT_{t})^{-\top}] \in \mathcal{C}^{1}(\mathcal{I},\mathcal{C}(\overline{D})^{d\times d}),\\	
	[t \mapsto \dett] &\in \mathcal{C}^{1}(\mathcal{I},\mathcal{C}(\overline{\Omega})), 
		&&\ [t \mapsto \dett] \in \mathcal{C}^{1}(\mathcal{I},\mathcal{C}^{0,1}(\overline{D})),\\
	[t \mapsto A_{t}] &\in \mathcal{C}^{1}(\mathcal{I},\mathcal{C}(\overline{\Omega})^{d \times d}),
		&&[t \mapsto A_{t}] \in \mathcal{C}(\mathcal{I},\mathcal{C}(\overline{D})^{d \times d}).
\end{aligned}
\right.
\end{equation}

\subsection{Proof of Lemma \ref{lem:transported_problem}}
	The pair $({\uu_{t}}, {p_{t}}) \in {\HH_{\Gamma, \vect{0}}^{1}(\Omega_{t})}^{d} \times \LL^{2}(\Omega_{t})$ solves the system
	\[
        \left\{
        	\begin{aligned}
        		\intOt{{\alpha} \nabla {\uu_{t}} : \nabla {\cbphi}} + i \intSt{ {\uu_{t}} \cdot {\cbphi} } &- \intOt{p_{t} ( \nabla \cdot \cbphi) }\\
			&= \intSt{(\bgg_{t} + i \ff_{t} ) \cdot \cbphi}, \quad \forall {\bphi} \in {\HH_{\Gamma, \vect{0}}^{1}(\Omega_{t})}^{d},\\
		- \intOt{\bar{\lambda} ( \nabla \cdot \uu_{t}) } &= 0, \quad \forall \lambda \in {\LL^{2}(\Omega_{t})}.
         	\end{aligned}
        \right.
	\]
	By change of variables (cf. \cite[pp. 482--484]{DelfourZolesio2011}), the fact that B$_{t} = 1$ on $\Sigma$, for all $t>0$, $\uu^{t} = {\uu_{t}} \circ T_{t} = ( {\uu_{rt}} + i {\uu_{it}})  \circ T_{t}$, and together with the identities
	\[
		(\nabla \bphi_{t}) \circ T_{t} 
			=  DT_{t}^{-\top} \nabla \bphi^{t}
			\quad \text{and} \quad 
		(\nabla \cdot \bphi_{t}) \circ T_{t} 
			= (DT_{t})^{-1} : \nabla \bphi^{t},
	\]
	where $\bphi_{t} \in \HH^{1}(\Omega_{t})^{d}$ and $\bphi^{t} \in \HH^{1}(\Omega)^{d}$, we get
	\[
        \left\{
        	\begin{aligned}
	  \intO{{\alpha} A_{t} \nabla {\uu^{t}} : \nabla {\cbphi}} + i \intS{ {\uu^{t}} \cdot {\cbphi} } - \intO{ \dett p^{t} ( \Mt^{\top} : \nabla \cbphi) }	 
	  &= \intS{ (\bgg^{t} + i \ff^{t} ) \cdot \cbphi}, \quad \forall {\bphi}\in {\Vgamma},\\
	-\intO{ \dett \bar{\lambda} ( \Mt^{\top} : \nabla {\uu^{t}}) } &= 0, \quad \forall \lambda \in {Q}.
         	\end{aligned}
        \right.
	\]
        Using the notations in \eqref{eq:transformed_forms}, we recover \eqref{eq:transformed_ccbm_weak_form}.
	
	The rest of the proof is again a routine step. The well-posedness is achieved using standard arguments similar to the proof of Proposition \ref{prop:well_posedness_of_CCBM}, which follows the same lines of argumentation for the real case. 	
	More precisely, the argumentations make use of Lemma \ref{lem:boundedness_of_sesquiliner_and_linear_forms} which exploits the properties of $\dett$, $A_{t}$, and $B_{t}$ issued in \eqref{eq:bounds_At_and_Bt} and \eqref{eq:regular_maps}.
	Concerning uniqueness of solution, the proof (with a compatibility condition imposed) is also standard, so we omit it. 
%
%
%
%
\subsection{Proof of Lemma \ref{lem:boundedness_of_the_transformed_state}}
	We first obtain a uniform bound for $\uu^{t}$ in $X$ for all $t \in \mathcal{I}$.
	To do this, we take $(\bphi,\lambda) = (\uu^{t},p^{t}) \in \Vgamma \times Q$ in \eqref{eq:transformed_ccbm_weak_form},
	utilize the properties of $A_{t}$ and $B_{t}$ given in \eqref{eq:bounds_At_and_Bt} and \eqref{eq:regular_maps}. 
	as well as the coercivity of $\aat$ on $X \times X$ (Lemma \ref{lem:boundedness_of_sesquiliner_and_linear_forms}).
	From \eqref{eq:transformed_ccbm_weak_form}, we have $\aat({\uu^{t}},{\uu^{t}}) = F^{t}({\uu^{t}})$.
	Using the relaxed $\Vgamma$-ellipticity property of $a^{t}$ (cf. Lemma \ref{lem:boundedness_of_sesquiliner_and_linear_forms}), we get the following sequence of inequalities
	\[
	\vertiii{\uu^{t}}_{X}^{2} 
		\lesssim \Re\left\{ \intO{{\alpha} A_{t} \abs{\nabla{\uu^{t}}}^{2}} + i \intS{\abs{\uu^{t}}^{2}} \right\}\\	
		\lesssim \norm{(\ff,\bgg)}_{1/2,\Sigma} \vertiii{\overline{\uu}^{t}}_{0,\Sigma}.
	\]
	Applying the trace theorem, and the fact that $\bgg^{t} = \bgg$ and $\ff^{t}=\ff$ on $\Sigma$, we get the estimate
	\begin{equation}\label{eq:bound_for_ut}
		\vertiii{\uu^{t}}_{X} \lesssim \norm{(\ff,\bgg)}_{1/2,\Sigma},
	\end{equation}
	which shows that $\uu^{t}$ is bounded in $X$ for all $t \in \mathcal{I}$.
	
	For the boundedness of $p^{t}$ in $Q$ for all $t \in \mathcal{I}$, we employ \eqref{eq:inf_sup_condition_for_transformed_problem}, which is equivalent to the following inequality condition
	\[
		\sup_{\substack{\bphi \in {\Vgamma}\\ \bphi \neq \vect{0}}} \frac{b^{t}(\bphi,{\lambda})}{\vertiii{\bphi}_{X} } \geqslant \beta_{1} \vertiii{{\lambda}}_{Q},\quad \forall {\lambda} \in Q.
	\]
	We let $\lambda = p^{t} \in Q$.
	Then, from \eqref{lem:transported_problem} and \eqref{eq:bounds_At_and_Bt} we obtain the following sequence of inequalities
	\begin{align*}
		\beta_{1} \vertiii{{p^{t}}}_{Q}
			&\leqslant \sup_{\substack{\bphi \in {\Vgamma}\\ \bphi \neq \vect{0}}} \frac{b^{t}(\bphi,{p^{t}})}{\vertiii{\bphi}_{X} }
			=\sup_{\substack{\bphi \in {\Vgamma}\\ \bphi \neq \vect{0}}} \vertiii{\bphi}_{X}^{-1}
				\Bigg\{ F^{t}(\bphi) - \aat({\uu^{t}},{\bphi}) \Bigg\}\\
			&\leqslant 
			\sup_{\substack{\bphi \in {\Vgamma}\\ \bphi \neq \vect{0}}} \vertiii{\bphi}_{X}^{-1}
				\Bigg\{ 
				 	\intS{(\bgg + i \ff ) \cdot \cbphi}
					-  \intO{{\alpha} A_{t} \nabla {\uu^{t}} : \nabla {\cbphi}}
					- \intS{{\uu^{t}} \cdot {\cbphi} } 
			 	\Bigg\}\\
			&\leqslant 
			c \sup_{\substack{\bphi \in {\Vgamma}\\ \bphi \neq \vect{0}}} \vertiii{\bphi}_{X}^{-1}
				\Bigg\{  
					\norm{(\ff,\bgg)}_{1/2,\Sigma} \vertiii{\cbphi}_{0,\Omega} 
					+  \vertiii{\nabla {\uu^{t}}}_{0,\Omega} \vertiii{\nabla {\cbphi}}_{0,\Omega}
					+ \vertiii{{\uu^{t}}}_{0,\Sigma} \vertiii{{\cbphi}}_{0,\Sigma} 
			 	\Bigg\}\\
			&\leqslant 
			c \left( \norm{(\ff,\bgg)}_{1/2,\Sigma} + \vertiii{\uu^{t}}_{X}  \right).								
	\end{align*}
	Therefore, using the estimate for $\vertiii{\uu^{t}}_{X}$ given in \eqref{eq:bound_for_ut}, we also have the estimate $\vertiii{{p^{t}}}_{Q} \lesssim \norm{(\ff,\bgg)}_{1/2,\Sigma}$, completing the proof of the lemma.
%
%
%
%
%
\subsection{Proof of Lemma \ref{lem:holder_continuity}}
	%
	The proof of the lemma mainly consists of two key steps: the first step is to show that $\lim_{t \to 0} \uu^{t} = \uu$ in $X$ and $\lim_{t \to 0} p^{t} = p$ in $Q$, and the second one is to prove that $\lim_{t\to0^{+}} \frac{1}{\sqrt{t}} ( \vertiii{\uu^{t} - \uu}_{X} + \vertiii{p^{t} - p}_{Q} ) = 0$.

	\underline{\textnormal{Step 1}.} We consider the difference between the variational equations \eqref{eq:transformed_ccbm_weak_form} and \eqref{eq:ccbm_weak_form}, and define $\yt := \uu^{t} - \uu$ and $r^{t} := p^{t} - p$.
	By making $\varepsilon > 0$ smaller if necessary, it can be shown that the following expansions hold for all $t \in \mathcal{I}$,
	\begin{equation*}\label{eq:approximations}
	\left\{
	\begin{aligned}
		\dett &= 1 + t \operatorname{div} \VV + t^{2} \tilde{\rho}(t,\VV), \quad \text{where} \ \tilde{\rho} \in \mathcal{C}(\mathbb{R},\mathcal{C}^{0,1}(D)),\\
		\Mt^{\top} &= (DT_{t})^{-1} = (id + t D\VV)^{-1} = id - t D\VV + O{(t^{2})} id, \quad \text{where}\ 0 \leqslant {{t}} \leqslant \varepsilon < |\lambda_{max}|^{-1},
	\end{aligned}
	\right.
	\end{equation*}
	where $\lambda_{max}$ is the maximum singular value of $D\VV$.
	Moreover, we denote
	\[
		\rho(t):=t^{2} \tilde{\rho}(t,\VV),
		\quad R(t) := O{(t^{2})} id,
		\quad \rho_{1}(t):=t \tilde{\rho}(t,\VV),
		\quad\text{and}\quad R_{1}(t) := O{(t)} id.\footnote{In some occasions, the remainder $\rho(t)$ and $R(t)$ may have a different structure for their exact expressions. Nevertheless, these expressions are always of order $O(t^{2})$. The same is true for $\rho_{1}(t)$ and $R_{1}(t)$. We abuse the use of these notations since the exact expressions are not actually of interest in our arguments.}
	\]
	Let us consider the variational equation
	\[
		b^{t}(\uu^{t},\lambda) - b(\uu,\lambda) 
		= -\intO{ \dett \bar{\lambda} ( \Mt^{\top} : \nabla {\uu^{t}}) } - \left( - \intO{\overline{\lambda} \nabla \cdot \uu } \right)
		= 0, \quad \forall \lambda \in Q.
	\]
	By applying the expansion for $\Mt^{\top}$ and taking $\lambda = r^{t} = p^{t} - p \in Q$ in the above equation, it can be shown that
	\begin{equation}\label{eq:identity_for_rtyt} 
	\begin{aligned}
	b(\yt,\rt) 
		= \intO{r^{t} \nabla \cdot \overline{\vect{y}}^{t} }  
	 	&= - \intO{ (t \nabla \cdot \VV + \rho(t)) r^{t} ( \Mt^{\top} : \nabla \overline{\uu}^{t})}\\
		&\qquad	- \intO{ r^{t} [ (-t D\VV + R(t) ): \nabla \overline{\uu}^{t} ] }.	
	\end{aligned}
	\end{equation}
	On the other hand, we also have the equation
	\begin{equation}\label{eq:difference_equation} 
			\aaa(\yt, \bphi) + b(\bphi,\rt) = \Phi^{t}(\bphi),\qquad \forall {\bphi} \in \Vgamma,
	\end{equation}
	where $\Phi^{t}(\bphi)$ is given by \eqref{eq:big_Phi_sup_t}, and $a: \Vgamma \times \Vgamma \to \mathbb{R}$ and $b: \Vgamma \times Q \to \mathbb{R}$ are given in \eqref{eq:forms_for_the_state_problem}.
	In above, we have used the fact that $B_{t} = 1$ on $\Sigma$, for all $t>0$, and that $\VV = \vect{0}$ on $\Sigma$.
	Choosing $\bphi = \yt \in \Vgamma$ in \eqref{eq:difference_equation} and utilizing identity \eqref{eq:identity_for_rtyt}, it follows that
	\begin{align*}
	c_{a}\vertiii{\yt}_{X}^{2}
		& \ \leqslant \left| \Phi^{t}(\yt) - b(\yt,p^{t} - p) \right| \\
		%
		%
		&\ \leqslant {{\alpha}} \left| {\mathfrak{a}_{t}} \right|_{\infty} \vertiii{\nabla {\uu^{t}}}_{0,\Omega} \vertiii{\nabla {\cyt}}_{0,\Omega}
		+ \left( \left| {\mathfrak{i}}_{t} \right|_{\infty} \abs{ \Mt^{\top} }_{\infty} +\left| {\mathfrak{m}_{t}} \right|_{\infty} \right) \vertiii{p^{t}}_{0,\Omega} \vertiii{ \nabla {\cyt} }_{0,\Omega} \\		 
		&\ \quad + {{t}} \left[ \left( \abs{\nabla \cdot \VV}_{\infty} + \abs{\rho_{1}(t)}_{\infty} \right) \abs{ \Mt^{\top} }_{\infty}
			+ \abs{D\VV}_{\infty} + \abs{R_{1}{(t)}}_{\infty} \right]
				\vertiii{r^{t}}_{0,\Omega} \vertiii{ \nabla \overline{\uu}^{t} }_{0,\Omega}.			
	\end{align*}
	Therefore, we have the estimate
	\begin{equation}\label{eq:first_estimate}
		c_{a}\vertiii{\yt}_{X}^{2}
		\leqslant {m}_{t} \vertiii{ \yt }_{X} + t {\Xi}^{t} \vertiii{r^{t}}_{Q},	
	\end{equation}
	where
	\begin{equation}\label{eq:coefficients}
	\left\{
	\begin{aligned} 
		{m}_{t} &:= {{\alpha}} \left| {\mathfrak{a}_{t}} \right|_{\infty} \vertiii {\uu^{t}}_{X}
				+ \left( \left| {I}_{t} \right|_{\infty} \abs{ \Mt^{\top} }_{\infty} +\left| {\mathfrak{m}_{t}} \right|_{\infty} \right) \vertiii{p^{t}}_{Q},\\
		{\Xi}^{t} &:=  \left[ \left( \abs{\nabla \cdot \VV}_{\infty} + \abs{\rho_{1}(t)}_{\infty} \right) \abs{ \Mt^{\top} }_{\infty} + \abs{D\VV}_{\infty} + \abs{R_{1}{(t)}}_{\infty} \right] \vertiii{ {\uu}^{t} }_{X}.	
	\end{aligned}
	\right.
	\end{equation}
	Note that the quantities ${m}_{t}$ and ${\Xi}^{t}$ are finite for all $t \in \mathcal{I}$ by Lemma \ref{lem:boundedness_of_the_transformed_state}.

	Now, to proceed further, we take $\lambda = r^{t} = p^{t}-p \in Q$ in the inf-sup condition \eqref{eq:inf_sup_nts} and consider equation \eqref{eq:difference_equation} to obtain
	\begin{equation}\label{eq:estimate_for_pt_final}
	\begin{aligned}
		\vertiii{{r^{t}}}_{Q} 
		 \leqslant \beta_{0}^{-1} \sup_{\substack{\bphi \in {\Vgamma}\\ \bphi \neq \vect{0}}} \frac{b(\bphi,{r^{t}})}{\vertiii{\bphi}_{X} }
		&=  \beta_{0}^{-1} \sup_{\substack{\bphi \in {\Vgamma}\\ \bphi \neq \vect{0}}} \vertiii{\bphi}_{X}^{-1} 
			\Big\{ \Phi^{t}(\bphi) - \aaa(\yt, \bphi) \Big\}\\
		&\leqslant  \beta_{0}^{-1} \left( {m}_{t} + \max\{ {{\mu}}, 1\} \vertiii{ \yt }_{X} \right).
	\end{aligned}
	\end{equation}
	Going back to \eqref{eq:first_estimate} and utilizing the above estimate, we get 
	\begin{equation}\label{eq:second_estimate_eliminating_pt}
	\begin{aligned}
		c_{a}\vertiii{\yt}_{X}^{2}
		&\leqslant t {\Xi}^{t} \beta_{0}^{-1} {m}_{t}  + \left(  {m}_{t}  + t {\Xi}^{t} \beta_{0}^{-1} \max\{ {{\mu}}, 1\} \right) \vertiii{ \yt }_{X}.
	\end{aligned}
	\end{equation}	
	We apply Peter-Paul inequality to the second summand in above inequaltiy to obtain the following estimate
	\begin{align*}
		\left(  {m}_{t}  + t {\Xi}^{t} \beta_{0}^{-1} \max\{ {{\mu}}, 1\} \right) \vertiii{ \yt }_{X}
			& \leqslant \frac{\left(  {m}_{t}  + t {\Xi}^{t} \beta_{0}^{-1} \max\{ {{\mu}}, 1\} \right)^{2}}{2\varepsilon_{1}} + \frac{\varepsilon_{1}}{2} \vertiii{ \yt }_{X}^{2}.
	\end{align*}
	for some constant $\varepsilon_{1} > 0$.
	We choose and fixed $\varepsilon_{1}$ such that $\bar{c}:=c(c_{a},\varepsilon_{1}) := 2c_{a} - \varepsilon_{1} > 0$ so that, from our first estimate \eqref{eq:second_estimate_eliminating_pt}, we get
	\begin{equation}\label{eq:final_estimate_for_yt_bounded_above}
	\vertiii{\yt}_{X}
		\leqslant \bar{c}^{-\frac{1}{2}} \left(  t {\Xi}^{t} \beta_{0}^{-1} {m}_{t} + \frac{\left(  {m}_{t}  + t {\Xi}^{t} \beta_{0}^{-1} \max\{ {{\mu}}, 1\} \right)^{2}}{2\varepsilon_{1}} \right)^{1/2}.	
	\end{equation}
	Because $\dett |\Mt\nn|^{-1} = (1) (|id \nn|^{-1}) = 1$ at $t=0$, then, in view of \eqref{eq:coefficients} together with Lemma \ref{lem:boundedness_of_the_transformed_state}, \eqref{eq:convergence_of_vector_valued_functions_1}, and \eqref{eq:convergence_of_vector_valued_functions_3}, we see that ${m}_{t} \to 0$ as $t \to 0$.
	Moreover, since ${\Xi}^{t}$ is uniformly bounded for all $t \in \mathcal{I}$ by Lemma \ref{lem:boundedness_of_the_transformed_state}, we deduce -- by Lebesgue's dominated convergence theorem -- that
	\begin{equation}\label{eq:limit_of_ut}
	\lim_{t \to 0} \vertiii{\yt}_{X} = 0
		\qquad \Longleftrightarrow \qquad
		\lim_{t \to 0} \uu^{t} = \uu \quad \text{in $X$}.
	\end{equation}

	Similarly, based from the above discussion, we know that the quantities on the right side of the inequality \eqref{eq:estimate_for_pt_final} vanish as $t \to 0$.
	Therefore, we also have the limit
	\begin{equation}\label{eq:limit_of_pt}
	\lim_{t \to 0} \vertiii{r^{t}}_{Q} = 0
		\qquad \Longleftrightarrow \qquad	
		\lim_{t \to 0} p^{t} = p \quad \text{in $Q$}.
	\end{equation}

	\underline{\textnormal{Step 2}.}	 For the second part of the proof, we start by noting that, for sufficiently small $t > 0$, $\frac{1}{t}\yt \in \Vgamma$ and $\frac{1}{t}r^{t} \in Q$.
	In addition, we recall that the derivatives of $\dett$ and $\Mt$ with respect to $t$ exists in $L^{\infty}(\Omega)$ and $L^{\infty}(\Omega)^{d\times d}$, respectively.
	Now, to finish the proof, we go back to inequality \eqref{eq:final_estimate_for_yt_bounded_above} to obtain, after dividing by $t > 0$, the following estimate
	\begin{align*}
		\frac{1}{t}\vertiii{\yt}_{X}^{2}
		\leqslant \bar{c}^{-1} \left( {\Xi}^{t} \beta_{0}^{-1} {m}_{t} + \frac{1}{2\varepsilon_{1}} \left(  \frac{1}{\sqrt{t}}{m}_{t}  + \sqrt{t} {\Xi}^{t} \beta_{0}^{-1} \max\{ {{\mu}}, 1\} \right)^{2} \right).	
	\end{align*}
	Observe that we have
	\[
	\frac{1}{\sqrt{t}}{m}_{t} 
		= \sqrt{t} {{\mu}} \left| \frac{\mathfrak{a}_{t}}{t} \right|_{\infty} \vertiii {\uu^{t}}_{X}\
			+ \sqrt{t}\left( \left| \frac{{\mathfrak{i}}_{t}}{t} \right|_{\infty} \abs{ \Mt^{\top} }_{\infty} + \left| \frac{\mathfrak{m}_{t}}{t} \right|_{\infty} \right) \vertiii{p^{t}}_{Q}.
	\]
	Thus, using \eqref{eq:regular_maps}, \eqref{eq:limits_of_maps}, \eqref{eq:convergence_of_vector_valued_functions_3}, Lemma \ref{lem:boundedness_of_the_transformed_state}, \eqref{eq:limit_of_ut}, and \eqref{eq:limit_of_pt}, we deduce that the following limit holds
	\begin{equation}\label{eq:limit_ratio_yt_over_t} 
		\lim_{t \to 0^{+} }\frac{1}{\sqrt{t}}\vertiii{\uu^{t} - \uu}_{X} = 0.
	\end{equation}

	Similarly, from estimate \eqref{eq:estimate_for_pt_final}, we know that (after dividing by $\sqrt{t} > 0$) the following inequality holds
	\begin{align*}
		\frac{1}{\sqrt{t}}\vertiii{{r^{t}}}_{Q} \leqslant \beta_{0}^{-1} \left[ \frac{1}{\sqrt{t}} {m}_{t} + \max\{ {{\mu}}, 1\} \left( \frac{1}{\sqrt{t}} \vertiii{ \yt }_{X} \right) \right].
	\end{align*}
	Again, applying \eqref{eq:regular_maps}, \eqref{eq:limits_of_maps}, \eqref{eq:convergence_of_vector_valued_functions_3}, and Lemma \ref{lem:boundedness_of_the_transformed_state}, but now combined with \eqref{eq:limit_ratio_yt_over_t}, we infer that
	\begin{equation}\label{eq:limit_ratio_rt_over_t} 
		\lim_{t \to 0^{+} }\frac{1}{\sqrt{t}}\vertiii{p^{t} - p}_{Q} = 0,
	\end{equation}	
	proving the lemma.
%
%
%

\Addresses

\end{document}